\documentclass[a4paper, 11pt,reqno]{amsart}
\usepackage[top = 1in, bottom = 1in, left=1.1in, right=1.1in]{geometry}

\usepackage[utf8]{inputenc}
\usepackage[T1]{fontenc}

\usepackage{amssymb,amsmath,amsthm,graphicx,xcolor,mathtools,mathrsfs,amsfonts}
\usepackage{hyperref}
\hypersetup{colorlinks, linkcolor={red!50!black}, citecolor={green!50!black}, urlcolor={blue!50!black}}

\usepackage[
backend=biber,
natbib=true,
style=alphabetic,
isbn=false,
maxbibnames=9
]{biblatex}
\usepackage{csquotes}
\renewbibmacro{in:}{}
\DeclareFieldFormat{pages}{#1}

\DeclareFieldFormat{title}{\mkbibquote{#1}}
\DeclareFieldFormat[misc]{date}{Preprint (#1)}

\renewbibmacro*{doi+eprint+url}{
	\iftoggle{bbx:url}
	{\iffieldundef{doi}{\usebibmacro{url+urldate}}{}}{}
	\newunit\newblock
	\iftoggle{bbx:eprint}
	{\usebibmacro{eprint}}
	{}
	\newunit\newblock
	\iftoggle{bbx:doi}
	{\printfield{doi}}
	{}
}

\usepackage{mathtools}
\usepackage{stmaryrd}
\usepackage[capitalize,nameinlink,noabbrev]{cleveref} 
\makeatletter
\AtBeginDocument{
  \let\fixcref@old@optarg\refstepcounter@optarg
  \def\refstepcounter@optarg[#1]#2{%
    \fixcref@old@optarg[#1]{#2}%
    \@namedef{cref@#2@alias}{#1}}%
}
\makeatother
\crefname{notation}{Notation}{Notations} 
  
\usepackage{enumitem}
\usepackage{dsfont} 

\newcommand*{\mc}[1]{\mathcal{#1}}
\newcommand*{\ms}[1]{\mathsf{#1}}
\newcommand*{\mo}[1]{\mathbf{#1}}

\newcommand{\N}{\mathbb{N}}
\newcommand{\Nn}{\mathbb{N}_0}
\newcommand{\Z}{\mathbb{Z}}

\newcommand{\R}{\mathbb{R}}
\newcommand{\Rp}{\R_{\geq0}}
\newcommand{\Rpp}{\R_{>0}}
\newcommand{\nnset}[1]{{\llbracket #1\rrbracket}}

\newcommand{\lab}{\left[} 
\newcommand{\rab}{\right]} 
\newcommand{\ab}[1]{\lab #1 \rab} 
\newcommand{\abOf}[1]{\!\ab{#1}} 
\newcommand{\lb}{\left(} 
\newcommand{\rb}{\right)} 
\newcommand{\br}[1]{\lb #1 \rb} 
\newcommand{\brOf}[1]{\!\br{#1}} 
\newcommand{\abs}[1]{\left| #1 \right|} 
\newcommand{\normof}[1]{\Vert#1\Vert}
\newcommand{\normOf}[1]{\left\Vert#1\right\Vert}
\newcommand{\innerProduct}[2]{\left\langle#1\,,\, #2\right\rangle}
\newcommand{\ip}[2]{\innerProduct{#1}{#2}}
\renewcommand{\Pr}{\mathbf{P}} 

\newcommand{\PrOf}[1]{\Pr\brOf{#1}}
\newcommand{\E}{\mathbf{E}} 
\newcommand{\Eof}[1]{\E[#1]}

\newcommand{\EOf}[1]{\E\abOf{#1}}
\newcommand{\V}{\mathbf{V}} 
\newcommand{\Vof}[1]{\V[#1]}
\newcommand{\VOf}[1]{\V\abOf{#1}}

\newcommand{\Cov}{\mathbf{Cov}} 
\newcommand{\Covof}[1]{\Cov(#1)}
\newcommand{\CovOf}[1]{\Cov\brOf{#1}}
\newcommand{\Cor}{\mathbf{Cor}} 

\newcommand{\CorOf}[1]{\Cor\brOf{#1}}
\newcommand{\opNormof}[1]{\Vert#1\Vert_{\mathsf{op}}}
\newcommand{\setByEleInText}[2]{\{#1 \mid #2\}}
\newcommand{\tr}{^{\!\top}\!} 
\newcommand{\pr}{^\prime}
\newcommand{\compl}{^\mathsf{c}}%
\newcommand{\ind}{\mathds{1}}
\newcommand{\indOf}[1]{\ind_{\!#1}}%
\DeclareMathOperator{\landauO}{\mathcal{O}}%
\newcommand{\eqcm}{\,,} 
\newcommand{\eqfs}{\,.}
\newcommand{\dl}{\mathrm{d}} 
\renewcommand{\subset}{\subseteq}
\DeclareMathOperator*{\argmin}{arg\,min}

\newcommand{\ol}[2]{\overline{#1#2}}
\newcommand{\pwa}{\phi}
\newcommand{\setC}{\mc C}
\newcommand{\diam}{\mathsf{diam}}
\newcommand{\Op}{\mathbf{O}_{\Pr}}
\newcommand{\CAT}{\mathsf{CAT}}

\newcommand{\Vf}{\mathbf{V}_{\mathbf{F}}} 
\newcommand{\Vfof}[1]{\Vf[#1]}
\newcommand{\VfOf}[1]{\Vf\abOf{#1}}

\newcommand{\Vm}{\mathbf{V}_{\mathbf{M}}} 
\newcommand{\Vmof}[1]{\Vm[#1]}
\newcommand{\VmOf}[1]{\Vm\abOf{#1}}
\newcommand{\covQ}{\operatorname{cov}_{\!\mc Q}}
\newcommand{\covQof}[1]{\covQ(#1)}
\DeclareMathOperator{\trace}{trace} 

\newcommand{\traceOf}[1]{\trace\brOf{#1}}

\newcommand{\mubar}{\bar\mu}
\newcommand{\mbar}{\bar m}
\newcommand{\muhat}{\hat\mu_n}
\newcommand{\mhat}{\hat m_n}

\newcommand{\VmsOf}[2]{\Vm^{#1}\abOf{#2}}

\newcommand{\sigmax}[1]{\sigma^{\max}_{#1}}

\makeatletter
\renewcommand{\paragraph}{\@startsection{paragraph}{4}%
  \z@{.5\linespacing\@plus.7\linespacing}{-.5em}%
  {\normalfont\itshape}}
\makeatother
\setlist[enumerate]{font=\upshape}

\AtBeginEnvironment{tabular}{\footnotesize}
\defbibheading{bibliography}[\refname]{\section*{#1}}

\theoremstyle{plain}
\newtheorem{theorem}{Theorem}[section]
\newtheorem{corollary}[theorem]{Corollary}
\newtheorem{proposition}[theorem]{Proposition}
\newtheorem{lemma}[theorem]{Lemma}
\theoremstyle{definition}
\newtheorem{definition}[theorem]{Definition}
\newtheorem{example}[theorem]{Example}
\newtheorem{notation}[theorem]{Notation}
\newtheorem{remark}[theorem]{Remark}

\begin{document}

\setcounter{tocdepth}{2}

\title[Fréchet Means of Dependent Data]{Finite-Sample Bounds for Fréchet Means of Dependent Data in Metric Spaces with Curvature Bounded Above}

\author[Z. F. Rico]{Zoraida F. Rico}
\address{Department of Decision Sciences, Bocconi University}
\email{\href{mailto:zoraida.rico@unibocconi.it}{zoraida.rico@unibocconi.it}}
\urladdr{\url{https://orcid.org/0009-0002-4972-909X}}

\author[C. Sch\"otz]{Christof Sch\"otz}
\address{Department of Aerospace and Geodesy, TUM School of Engineering and Design, Technical University of Munich}
\email{\href{mailto:christof.schoetz@tum.de}{christof.schoetz@tum.de}}
\urladdr{\url{https://orcid.org/0000-0003-3528-4544}}


\begin{abstract}

We prove finite-sample risk bounds for empirical quadratic and power Fréchet means in geodesic metric spaces with curvature bounded above for observations with arbitrary joint law. We compare the sample with exchangeable copies in which one observation is redrawn independently of the original sample and the others move to preserve the joint law. These displacements control the error, with no entropy or dimension assumptions. For quadratic Fréchet means, this yields a general ``coupling bound'' in any separable complete $\CAT(\kappa)$ space, where for $\kappa>0$ the observations must lie in a closed convex set of circumradius below $\pi/(2\sqrt{\kappa})$.
This bound is exact on the real line for nonnegatively correlated Gaussian observations and for comonotone functions of a common factor.
We apply it in two settings. Under local dependence, the resulting risk bound grows linearly with the average size of the dependency neighborhoods. Under transport sensitivity and for bounded metric variances, summable sensitivities of the conditional laws to a change in one observation preserve the parametric mean squared error rate, while nonsummable polynomial decay gives explicit slower rates.
We verify transport sensitivity for stationary reversible Markov chains with Wasserstein-contractive kernels and for bi-Lipschitz images of Gaussian sequences. In the more challenging setting of heavy-tailed, dependent data with possibly infinite variance, we consider power Fréchet means of order $\alpha\in(1,2)$. Under nonpositive curvature, we bound their $L^\alpha$ risk assuming only finite $\alpha$-th moments.

\end{abstract}

\maketitle


\section{Introduction}\label{sec:intro}

The accuracy of an average depends on both the variability of the individual observations and the dependence among them.
For square-integrable observations $X_1,\dots,X_n$ in a Hilbert space, the mean squared error of the sample mean $\overline X_n:=n^{-1}\sum_{i=1}^nX_i$ is an average of covariances:
\begin{equation}\label{eq:intro:covariance}
    \EOf{\normOf{\overline X_n-\EOf{\overline X_n}}^2}
    =\frac{1}{n^2}\sum_{i,k=1}^n\traceOf{\CovOf{X_i,X_k}}
    \eqfs
\end{equation}
The diagonal terms ($i=k$) give the error for independent observations with the same marginal laws; the off-diagonal terms ($i\neq k$) can increase or decrease it.

Fréchet means extend averaging to metric spaces by minimizing an average squared distance~\cite{Frechet1948RandomElements-AIHP}.
For random variables $Y_1,\dots,Y_n$ with finite second moments in a metric space $(\mc Q,d)$, the population and the empirical Fréchet mean are
\begin{equation}\label{eq:intro:means}
    \mbar=\argmin_{q\in\mc Q}\frac{1}{n}\sum_{i=1}^n\EOf{d(Y_i,q)^2},
    \qquad
    \mhat=\argmin_{q\in\mc Q}\frac{1}{n}\sum_{i=1}^n d(Y_i,q)^2
    \eqfs
\end{equation}
In a Hilbert space, they reduce to the expectation and the sample mean in \eqref{eq:intro:covariance}.
Fréchet means provide location parameters for observations such as phylogenetic trees, positive definite matrices, and probability distributions, equipped with appropriate metrics~\cite{BilleraEtAl2001PhylogeneticTrees-AAM, BhatiaHolbrook2006MatrixMeans-LAA,PanaretosZemel2020Wasserstein-Springer}.
Such observations are often dependent. Random objects recorded over time form time series~\cite{ZhuMueller2023Autoregressive-JRSSB,BulteSorensen2024Autoregressive-arxiv}, measurements at nearby locations are spatially dependent, and Markov chain Monte Carlo produces dependent samples, for instance of phylogenetic trees~\cite{MageeEtAl2024Phylogenetic-BA}. This dependence affects the precision of the empirical Fréchet mean. For these nonlinear averages, however, the linear expansion in \eqref{eq:intro:covariance} is unavailable. We therefore ask which metric properties of the joint law yield explicit finite-sample bounds on the accuracy of averaging.

We address this question through couplings that replace one observation by an independent draw while preserving the joint law of the sample.
Under dependence, preserving this law generally requires other observations to move as well. Their displacements control the risk of the empirical mean. We develop this principle for quadratic Fréchet means in Hadamard spaces, extend it to spaces with a positive upper curvature bound, and obtain bounds for power Fréchet means under weaker moment assumptions.

\subsection{A coupling principle}\label{ssec:intro:contrib}

We first consider quadratic Fréchet means of observations $\mo Y=(Y_1,\dots,Y_n)$ in a separable Hadamard space $(\mc Q,d)$, that is, a complete geodesic space of nonpositive curvature~\cite{Sturm2003NPC-ContempMath}, in which the means in \eqref{eq:intro:means} exist and are unique. This setting includes infinite-dimensional Hilbert spaces and metric trees, and does not require smoothness or local coordinates.
We assume finite second moments, but neither independence nor identical marginal laws.

For independent observations, a direct route to a risk bound is \emph{replace-one stability}~\cite{BousquetElisseeff2002Stability-JMLR,ShalevShwartzEtAl2010Learnability-JMLR}, which has been used for Fréchet means in~\cite{Escande2024Minimizers-JMLR,BrunelSerres2026FiniteSample-arxiv,Schoetz2026Transformed-arxiv,Schoetz2026Weighted-arxiv}.
A \emph{variance inequality}~\cite{Sturm2003NPC-ContempMath,Schoetz2025VarianceInequalities-EJP} bounds the risk $\Eof{d(\mhat,\mbar)^2}$ by the expected difference between the objectives in \eqref{eq:intro:means} at $\mhat$.
After replacing $Y_i$ by an independent copy $Y_i\pr$, a \emph{quadruple inequality}~\cite{Sturm2003NPC-ContempMath,Schoetz2024Quadruple-MIA,Schoetz2026Trapezoid-arxiv} bounds this difference by the expected product of $d(Y_i,Y_i\pr)$ and the movement of the empirical mean due to this replacement.
A \emph{contraction property}~\cite[Theorem~6.3]{Sturm2003NPC-ContempMath} bounds this movement by $d(Y_i,Y_i\pr)/n$, yielding
\begin{equation}\label{eq:intro:independent}
    \EOf{d(\mhat,\mbar)^2} \leq \frac{1}{n^2}\sum_{i=1}^n\VmOf{Y_i}
    \eqcm
\end{equation}
where $\Vmof{Y_i}=\frac{1}{2}\Eof{d(Y_i,Y_i\pr)^2}$ is the metric variance, which equals $\trace(\Covof{Y_i,Y_i})$ in Hilbert spaces; there, \eqref{eq:intro:independent} is an equality by \eqref{eq:intro:covariance}.
These arguments rest on two properties of the replacement.
First, $Y_i\pr$ is independent of $\mhat$, so the population objective at $\mhat$ can be expressed through $Y_i\pr$.
Second, swapping $Y_i$ and $Y_i\pr$ leaves the joint law of the original and the modified sample unchanged, which permits a symmetrization.

Under dependence, a one-coordinate replacement generally cannot have both properties.
An independent draw of $Y_i$ changes the joint law of the sample, whereas a draw from the conditional law of $Y_i$ given the other observations preserves this law but is not independent of them.
We retain the independent draw and let the remaining coordinates move so that the swap symmetry holds again.
The result is an \emph{exchangeable copy with $i$-independence}, a sample $\mo Y^i=(Y_1^i,\dots,Y_n^i)$ such that $(\mo Y,\mo Y^i)\stackrel d=(\mo Y^i,\mo Y)$ and $Y_i^i\perp\mo Y$ (\cref{def:dep:exchangeable}).
Exchangeability here is a property of the coupling, not an assumption on the observations. An independent copy of the entire sample is always admissible, but to obtain useful error bounds, we want to find copies that move few coordinates or move them little.

We call this extension of the replace-one argument the \emph{compensated replace-one stability technique}.
It yields the coupling bound of \cref{thm:2fmh:dependent},
\begin{equation}\label{eq:intro:coupling}
    \EOf{d(\mhat,\mbar)^2} \leq \frac{1}{2n^2}\sum_{i,k=1}^n \EOf{d(Y_i,Y_i^i)\,d(Y_k,Y_k^i)}\eqfs
\end{equation}
The diagonal terms form the right-hand side of \eqref{eq:intro:independent}, which is recovered for independent observations by leaving the other coordinates unchanged.
The off-diagonal terms are a nonnegative penalty for the additional movements required by the chosen coupling.

In a Hilbert space, \eqref{eq:intro:coupling} is a metric form of \eqref{eq:intro:covariance}.
Let $\mo X^i=(X_1^i,\dots,X_n^i)$ be an exchangeable copy of $\mo X=(X_1,\dots,X_n)$ with $i$-independence.
Then $\Eof{\langle X_i-X_i^i\,,\, X_k-X_k^i\rangle}=2\trace(\Covof{X_i,X_k})$ (\cref{sssec:2fmh:illu:general}) and \eqref{eq:intro:covariance} and \eqref{eq:intro:coupling} become
\begin{align*}
    \EOf{\normOf{\overline X_n-\EOf{\overline X_n}}^2}
    &= \frac1{2n^2}\sum_{i,k=1}^n \EOf{\ip{X_i-X_i^i}{X_k-X_k^i}}
    \\
    &\leq \frac1{2n^2}\sum_{i,k=1}^n \EOf{\normof{X_i-X_i^i}\,\normof{X_k-X_k^i}}
    \eqfs
\end{align*}
So \eqref{eq:intro:coupling} in a Hilbert space replaces each inner product by the product of norms, which in a metric space becomes $d(Y_i,Y_i^i)\,d(Y_k,Y_k^i)$.
The diagonal terms ($i=k$) are unchanged, but the off-diagonal terms ($i\neq k$) become nonnegative; unlike in \eqref{eq:intro:covariance}, they can therefore only increase the bound above the independent-sample bound \eqref{eq:intro:independent}.
Still, the bound is exact in natural dependent settings: suitable copies achieve equality for real jointly Gaussian observations with nonnegative covariances (\cref{prp:2fmh:gaussian-1d}), and an independent copy of the entire sample already does so for comonotone functions of a common real factor (\cref{prp:real:monotone-factor}).

\subsection{Results}\label{ssec:intro:consequences}

We make the coupling bound explicit under two forms of dependence and extend it to positive curvature and to power Fréchet means.

\paragraph*{Local dependence}
Suppose that $Y_i$ is independent of the observations outside a neighborhood $N_i\ni i$, as for spatial data with finite-range dependence.
Then an exchangeable copy $\mo Y^i$ with $i$-independence can be chosen to move only the coordinates in $N_i$ (\cref{prp:dep:neighborhood-copy}).
Denoting the average neighborhood size as $\bar N$, \cref{thm:2fmh:neighborhood} gives
\begin{equation*}
    \EOf{d(\mhat,\mbar)^2} \leq \sqrt2\,\max_{i}\VmOf{Y_i}\,\frac{\bar N}{n}
    \eqfs
\end{equation*}

\paragraph*{Transport sensitivity}
Suppose that the conditional laws of the remaining sample given $Y_i=y$ and $Y_i=y\pr$ can be coupled so that each $Y_k$ moves by at most $b_{ik}\,d(y,y\pr)$ in expectation (\cref{def:dep:transport}).
Then \cref{thm:2fmh:decl} gives
\begin{equation*}
    \EOf{d(\mhat,\mbar)^2} \leq \frac{1}{n^2}\sum_{i=1}^n\VmOf{Y_i}\sum_{k=1}^n b_{ik} \eqfs
\end{equation*}
With bounded metric variances, bounded row sums of $(b_{ik})$ preserve the rate $n^{-1}$, and nonsummable polynomial decay of $b_{ik}$ in $\lvert i-k\rvert$ gives explicit slower rates, which fractional Gaussian noise attains (\cref{cor:2fmh:decl:profile} and the discussion after it).
Transport sensitivity holds for Gaussian sequences, their bi-Lipschitz images, and stationary reversible Markov chains whose kernel contracts the $1$-Wasserstein distance by a factor $\rho<1$ (\cref{ssec:dep:examples}); for such chains, the bound is at most $\frac{1+\rho}{1-\rho}\Vmof{Y_1}/n$, which is the independent-data bound \eqref{eq:intro:independent} for the effective sample size $n(1-\rho)/(1+\rho)$ (see the discussion after \cref{cor:2fmh:decl:profile}).
For the refresh chain, which keeps its current value with probability $\rho$ and otherwise draws a fresh value from its marginal law, the transport-sensitivity bound is exact in Hilbert spaces (\cref{prp:2fmh:refresh}).
Without reversibility, the condition can fail: an $\ms{AR}(1)$ process with Rademacher innovations and parameter $\rho\in(0,1/2)$ has no summable sensitivities, and our bounds do not attain its risk rate $n^{-1}$ (\cref{rem:dep:ar1,rem:2fmh:decl:limits}).

\paragraph*{Positive curvature}
Let $\mc Q$ be a separable complete $\CAT(\kappa)$ space with $\kappa>0$, and let the observations take values in a closed convex set of circumradius $\varrho<\pi/(2\sqrt\kappa)$.
Then \cref{thm:pos:dependent} gives
\begin{equation*}
    \EOf{d(\mhat,\mbar)^2} \leq \frac{C_\Delta^4}{2n^2}\sum_{i,k=1}^n \EOf{d(Y_i,Y_i^i)\,d(Y_k,Y_k^i)}
    \eqcm
\end{equation*}
where $C_\Delta:=\Delta/\sin\Delta$ with $\Delta:=2\varrho\sqrt\kappa$.
Thus \eqref{eq:intro:coupling} and its consequences above hold up to curvature-dependent factors, which tend to one as $\kappa\to0$ (\cref{sec:pos}).

\paragraph*{Power Fréchet means}
Replacing $d^2$ in \eqref{eq:intro:means} by $d^\alpha$ with $1<\alpha<2$ defines power Fréchet means, again denoted by $\mbar$ and $\mhat$.
For observations with an arbitrary joint law and finite $\alpha$-th moments in a Hadamard space, \cref{thm:power:dep} gives
\begin{equation*}
    \EOf{d(\mhat,\mbar)^\alpha}^{\frac1\alpha}
    \leq
    c_\alpha\,\sigma_\alpha^{\frac{2-\alpha}\alpha}\br{\frac{1}{n^2}\sum_{i,k=1}^n\EOf{d(Y_i,Y_i^i)^{\frac\alpha2}\,d(Y_k,Y_k^i)^{\frac\alpha2}}^{\frac{2\alpha-2}\alpha}}^{\frac{1}{2}}
    \eqcm
\end{equation*}
where $\sigma_\alpha:=n^{-1}\sum_{j=1}^n\Eof{d(Y_j,\mbar)^\alpha}$ and $c_\alpha$ depends only on $\alpha$; at $\alpha=2$, the bound has the form of \eqref{eq:intro:coupling}.
This bound thus combines a broad metric space setting, arbitrary dependence, and heavy tails with possibly infinite variance.
Suitably adapted results also hold under local dependence and transport sensitivity (\cref{thm:power:structured,thm:power:decl}).

\subsection{Related work}\label{ssec:intro:related}

The study of Fréchet means in the finite-sample setting complements their classical asymptotic theory~\cite{Ziezold1977ExpectedFigures-Reidel, BhattacharyaPatrangenaru2003LargeSample-AOS}.
For independent observations, a common route to convergence rates uses empirical-process bounds under entropy conditions \cite{Schoetz2019ConvergenceRates-EJS, AhidarCoutrixEtAl2020Barycenters-PTRF}.
These conditions can restrict the applicability of the resulting bounds in infinite-dimensional spaces.
Other approaches obtain dimension-free rates under additional geometric assumptions~\cite{LeGouicEtAl2023FastConvergence-JEMS} or use stability arguments~\cite{Escande2024Minimizers-JMLR, BrunelSerres2026FiniteSample-arxiv, Schoetz2026Transformed-arxiv,Schoetz2026Weighted-arxiv}.
Our compensated replace-one stability technique extends these arguments to dependent observations. As in the independent case, we compare two fitted means directly rather than control the empirical objective uniformly over the parameter space. We do not impose entropy conditions, and the bounds have no explicit dependence on the dimension of the space.
For independent observations with a common Fréchet mean, the bound \eqref{eq:intro:independent} is contained in the proofs of \cite[Theorems~3 and~4]{BrunelSerres2026FiniteSample-arxiv}.
For observations whose Fréchet means may differ, the same bound, with the mixture mean as the target, is given in \cite[Theorem~3.1]{Schoetz2026Weighted-arxiv}.

The main tools from the literature are quadruple and variance inequalities. In Hadamard spaces, quadruple inequalities are available for $d^2$~\cite{Reshetnyak1968InextensibleMappings-SMJ,Sturm2003NPC-ContempMath} and $d^\alpha$~\cite{Schoetz2019ConvergenceRates-EJS,Schoetz2024Quadruple-MIA}, and variance inequalities for quadratic~\cite{Sturm2003NPC-ContempMath} and power Fréchet means~\cite{Schoetz2025VarianceInequalities-EJP}. Under a positive curvature upper bound, the quadruple inequality for $d^2$ is shown in~\cite{Schoetz2026Trapezoid-arxiv} and the variance inequality for quadratic Fréchet means in~\cite{BrunelSerres2026FiniteSample-arxiv,Schoetz2026Weighted-arxiv}.

For dependent observations, ergodic theorems for nonlinear means are available in nonpositively curved spaces and more general metric settings~\cite{Austin2011ErgodicTheorem-JTA,Navas2013ErgodicTheorem-ETDS,AntezanaEtAl2023InductiveMeans-ETDS,Jaffe2024InfiniteDimensions-arxiv}, including for Fréchet mean sets of Markov chains on finite metric spaces~\cite[Theorem~4.7]{EvansJaffe2024MeanSets-Bernoulli}.
The $L^2$ limit points of Fréchet means of dependent observations in Hadamard spaces are identified in~\cite[Theorem~2]{KoestenbergerStark2025SignalRecovery-JMVA}. For time series in Hadamard spaces generated by iterated random functions, Bult\'e and S{\o}rensen~\cite[Theorem~4.1]{BulteSorensen2024Autoregressive-arxiv} obtain root-$n$ consistency of the quadratic Fréchet mean under geometric moment contraction, second moments, and a local entropy condition.

Coupling-based notions of dependence incorporate metric information, as in the $\tau$-coefficients of~\cite{DedeckerPrieur2005Dependence-PTRF} and geometric moment contraction~\cite{WuShao2004IteratedRandom-JAP}.
Our transport-sensitivity matrix $\mo b$ is also related to interdependence and coupling matrices used in covariance and concentration inequalities; see \cref{ssec:dep:relations} for a comparison of their conditioning structures.

Exchangeable pairs are a standard tool in Stein's method~\cite{Chatterjee2007Concentration-PTRF}.
For dependent observations, a common construction is to resample one coordinate from its conditional law given the others, leaving them unchanged. Our construction instead redraws one coordinate from its marginal law, independently of the original sample, and adjusts the remaining coordinates so that the two samples form an exchangeable pair.
Stability-based generalization bounds for mixing samples~\cite{MohriRostamizadeh2010Stability-JMLR} discard observations near the replaced coordinate and approximate the remaining blocks by independent ones, with an error controlled by mixing coefficients.
Our couplings preserve the joint law of the sample exactly, and thus avoid this approximation error.

Power Fréchet means with $\alpha<2$ are robust to heavy tails \cite{Schoetz2019ConvergenceRates-EJS,Schoetz2026Transformed-arxiv,Schoetz2026Weighted-arxiv}: our bounds for them require only finite $\alpha$-th moments. In linear spaces, robust mean estimators have been analyzed under dependence~\cite{Zhang2021RobustEstimation-SS,XuEtAl2026MedianOfMeans-SLADS,WhitehouseEtAl2026HeavyTailed-SPA}.
In Banach spaces~\cite{Minsker2015GeometricMedian-Bernoulli} and in metric spaces~\cite{KimEtAl2025Robust-AISTATS,YunPark2023Exponential-Bernoulli}, median aggregation and generalized median-of-means give exponential guarantees from independent inputs; robust sub-Gaussian estimation of Fréchet means is studied in~\cite{BartlEtAl2025SubGaussian-arxiv}.

\subsection{Outline}\label{ssec:intro:outline}

\Cref{sec:pre} fixes the setting and notation. \Cref{sec:dep} introduces exchangeable copies, dependency neighborhoods, and transport sensitivity, constructs copies for several classes of processes, and relates the conditions to other notions of dependence. \Cref{sec:2fmh} develops the quadratic theory in Hadamard spaces and its Hilbert-space comparisons. \Cref{sec:pos} treats quadratic Fréchet means under positive curvature, and \cref{sec:power} develops the bounds for power Fréchet means in Hadamard spaces. Technical proofs and detailed computations for the examples are collected in the appendix.

\section{Setting and notation}\label{sec:pre}

For $n\in \N$, define $\nnset n := \{1, \dots, n\}$, and let $\Nn := \N\cup\{0\}$. For a finite set $A$, $\# A$ denotes its cardinality.
For $a,b\in\R$, $a\vee b := \max\{a,b\}$. For random variables $X,Y$, we denote by $\mc L(X)$ the law of $X$, $X\stackrel{d}{=}Y$ if $\mc L(X)=\mc L(Y)$, and $X\perp Y$ if $X$ and $Y$ are independent; $\mc L(X\mid Y=y)$ denotes a (fixed version of a) regular conditional distribution of $X$ given $Y=y$; iid abbreviates independent and identically distributed. Constants denoted by $c_\alpha$ are positive and finite and depend only on $\alpha$; their value may change from one occurrence to the next. We use the conventions $0^0:=1$ and $0\cdot\infty:=0$.

\paragraph*{Metric spaces and transport}
Let $(\mc Q,d)$ be a metric space.
For $q,p\in\mc Q$, we use the short notation $\ol qp := d(q,p)$.
Throughout, $\mc Q$ is complete and separable, that is, a Polish space; in \cref{sec:2fmh,sec:power} it is a Hadamard space and in \cref{sec:pos} a complete $\CAT(\kappa)$ space, see \cite{BridsonHaefliger1999NPC-Springer,Sturm2003NPC-ContempMath}.
Polishness avoids measurability issues and guarantees the existence of the regular conditional distributions used below.

Let $\mc P(\mc Q)$ be the set of Borel probability measures on $\mc Q$.
For $s\in\Rpp$, let $\mc P^s(\mc Q) \subseteq \mc P(\mc Q)$ be the set of $\mu$ with $\int \ol yq^s \dl \mu(y) < \infty$ for one, and hence all, $q\in\mc Q$; we then say that $\mu$, and any random variable $Y$ with law $\mc L(Y)=\mu$, has a \emph{finite $s$-th moment}.
For $\mu,\nu\in\mc P(\mc Q)$, let $\Pi(\mu,\nu)$ be the set of probability measures on $\mc Q^2$ with marginals $\mu$ and $\nu$, and, for $s\in\Rpp$, define the \emph{$s$-Wasserstein cost}
\begin{equation}\label{eq:pre:wasserstein}
    W_s(\mu,\nu) := \inf_{\pi\in\Pi(\mu, \nu)} \int \ol yz^s\, \dl\pi(y,z)
    \eqfs
\end{equation}
In particular, $W_1$ is the $1$-Wasserstein distance; for $s\geq1$, $W_s^{1/s}$ is the $s$-Wasserstein distance, and for $s\in(0,1]$, $W_s$ is the $1$-Wasserstein distance with respect to the metric $d^s$.

\paragraph*{Means}
Let $\alpha\in(1,2]$ and $\mu\in\mc P^{\alpha-1}(\mc Q)$.
Fix a reference point $o\in\mc Q$.
Any minimizer of
\begin{equation}\label{eq:pre:mean}
    q\mapsto\int\brOf{\ol yq^\alpha-\ol yo^\alpha}\dl\mu(y)
\end{equation}
is called an \emph{$\alpha$-Fréchet mean} of $\mu$; the integrand is $\mu$-integrable because $\lvert \ol yq^\alpha-\ol yo^\alpha\rvert\leq\alpha(\ol yo^{\alpha-1}+\ol oq^{\alpha-1})\ol oq$, and the set of minimizers does not depend on $o$.
For $\alpha=2$ we speak of the ($2$-)Fréchet mean or quadratic Fréchet mean, and for $\alpha\in(1,2]$ of the power Fréchet mean of order $\alpha$ or the $\alpha$-Fréchet mean.
In all settings of this paper, the minimizer exists and is unique; see \cref{ssec:2fmh:pre,ssec:pos:pre,ssec:power:prelim}.
We denote it by $m_\alpha(\mu)$.
If $\mu\in\mc P^\alpha(\mc Q)$, then $m_\alpha(\mu)$ minimizes $q\mapsto\int\ol yq^\alpha\dl\mu(y)$.
If $\mc Q$ is a Hilbert space, $m_2(\mu)$ is the expectation of $\mu$.

\paragraph*{Variances}
For $s\in\Rpp$ and $\mu\in\mc P(\mc Q)$, we call
\begin{equation}\label{eq:pre:variances}
    \Vf^s(\mu) := \inf_{q\in\mc Q}\int \ol yq^s\dl\mu(y)
    \quad\text{and}\quad
    \Vm^s(\mu) := \frac{1}{2}\iint \ol y{y\pr}^s\dl\mu(y)\dl\mu(y\pr)
\end{equation}
the \emph{Fréchet $s$-variance} and the \emph{metric $s$-variance} of $\mu$, both with values in $[0,\infty]$.
Following \cite{DubeyMueller2020TimeVarying-JRSSB,SongMueller2026Dispersion-JASA,Schoetz2026Weighted-arxiv}, $\Vf:=\Vf^2$ is the \emph{Fréchet variance} and $\Vm:=\Vm^2$ the \emph{metric variance}.
For a random variable $Y$, we define $\Vf^{s}[Y]:=\Vf^s(\mc L(Y))$ and $\Vm^{s}[Y]:=\Vm^s(\mc L(Y))$. Then, $\Vm^{s}[Y]=\frac{1}{2}\Eof{\ol Y{Y\pr}^s}$ for an independent copy $Y\pr$ of $Y$. In addition, if $\mc L(Y)\in\mc P^2(\mc Q)$ has a $2$-Fréchet mean $m$, $\Vfof{Y}=\Eof{\ol Ym^2}$.
Square brackets indicate that the argument is a random variable, which enters only through its law.

\paragraph*{Sampling notation}
Throughout, $\mo Y=(Y_1,\dots,Y_n)$ is an $n$-tuple of $\mc Q$-valued random variables with an arbitrary joint law and marginal laws $\mu_i:=\mc L(Y_i)$, $i\in\nnset n$.
We define
\begin{equation}\label{eq:pre:sampling}
    \mubar := \frac{1}{n}\sum_{i=1}^n\mu_i
    \eqcm\quad
    \muhat := \frac{1}{n}\sum_{i=1}^n\delta_{Y_i}
    \eqcm\quad
    \mbar := m_\alpha(\mubar)
    \eqcm\quad
    \mhat := m_\alpha(\muhat)
    \eqcm
\end{equation}
where $\alpha=2$ in \cref{sec:2fmh,sec:pos} and $\alpha\in(1,2)$ in \cref{sec:power}.
The target $\mbar$ is the $\alpha$-Fréchet mean of the average marginal law; it is the common mean if all $\mu_i$ share one.
When, for $i\in\nnset n$, a second $n$-tuple $\mo Y^i=(Y_1^i,\dots,Y_n^i)$ has been chosen (\cref{sec:dep}), we define $\muhat^i:=\frac{1}{n}\sum_{k=1}^n\delta_{Y_k^i}$ and $\mhat^i:=m_\alpha(\muhat^i)$.

\cref{tab:pre:symbols} lists the recurring symbols.

\begin{table}[ht]
\caption{Recurring symbols}\label{tab:pre:symbols}
\begin{tabular}{@{}ll@{}}
\hline
$\ol qp$, $W_s(\mu,\nu)$ & distance, $s$-Wasserstein cost; \eqref{eq:pre:wasserstein}\\
$m_\alpha(\mu)$ & $\alpha$-Fréchet mean; \eqref{eq:pre:mean}\\
$\Vf^s$, $\Vm^s$; $\Vf$, $\Vm$ & Fréchet and metric $s$-variance; \eqref{eq:pre:variances}\\
$\mu_i$, $\mubar$, $\muhat$, $\mbar$, $\mhat$ & marginals, mixture, empirical measure, target, estimator; \eqref{eq:pre:sampling}\\
$\mo Y^i$, $\muhat^i$, $\mhat^i$ & copy with $i$-independence, its empirical measure and mean; \cref{def:dep:exchangeable}, \cref{sec:pre}\\
$N_i$, $N_{\max}$, $\bar N$ & dependency neighborhoods and their sizes; \cref{def:dep:neighborhood}, \eqref{eq:dep:L}\\
$\mo b=(b_{ik})$ & transport-sensitivity coefficients; \cref{def:dep:transport}\\
$a_n(\rho)$ & refresh-chain risk factor; after \cref{prp:2fmh:refresh}\\
$\kappa$, $\setC$, $\varrho$, $\Delta$, $C_\Delta$ & curvature bound, data domain, circumradius, curvature parameters; \eqref{eq:pos:Delta}\\
$\pwa$, $\sigma_s$, $\sigmax s$, $V^{\max}_s$ & exponent, moments, and metric variances for power means; \eqref{eq:power:alpha-phi}, \eqref{eq:power:sigma-def}, \cref{not:power:standing}\\
$\Phi$, $D_n$, $D_n^\ast$, $\omega_n$ & conversion map, displacement moments, fluctuation; \eqref{eq:power:dep:notation}\\
$\psi_\alpha$ & adapted loss; \eqref{eq:power:loss}\\
$\mc B_1,\dots,\mc B_\chi$ & partition into mutually independent families; \cref{ssec:power:strucdep}\\
$r_n$ & normalized maximal row sum of $b_{ik}^{\alpha-1}$; \cref{ssec:power:decl}\\
\hline
\end{tabular}
\end{table}

\section{Dependence through couplings}\label{sec:dep}

All risk bounds of \cref{sec:2fmh,sec:pos,sec:power} compare the sample $\mo Y$ with a second sample in which one observation has been redrawn independently.
For independent observations, one simply replaces that observation.
Under dependence, such a replacement changes the joint law, and the remaining observations have to move as well to restore it; how far they move measures the dependence.
This section introduces the resulting couplings (\cref{ssec:dep:copies}) and two ways of constructing them with small displacements, dependency neighborhoods (\cref{ssec:dep:neighborhood}) and transport sensitivity (\cref{ssec:dep:transport}).
We then verify transport sensitivity for several classes of processes (\cref{ssec:dep:examples}) and compare it with other notions of dependence (\cref{ssec:dep:relations}).
In this section, $(\mc Q,d)$ is a Polish space; no curvature bound is needed.

\subsection{Exchangeable copies}\label{ssec:dep:copies}

\begin{definition}\label{def:dep:exchangeable}
    Let $n\in\N$.
    Let $\mo Y := (Y_1, \dots, Y_n)$ be an $n$-tuple of random variables $Y_i$ with values in $\mc Q$.
    Let $i \in \nnset n$.
    An \emph{exchangeable copy of $\mo Y$ with $i$-independence} is a second $n$-tuple $\mo Y^i = (Y_1^i, \dots, Y_n^i)$ such that $(\mo Y, \mo Y^i) \stackrel{d}{=} (\mo Y^i, \mo Y)$ and $Y_i^i$ is independent of $\mo Y$.
    In this case, we call $(\mo Y, \mo Y^i)$ an \emph{exchangeable pair of $n$-tuples with $i$-independence}.
\end{definition}

Exchangeability is a property of the coupling of $\mo Y$ with $\mo Y^i$, not an assumption on the observations.
Such copies exist for every joint law: an independent copy of $\mo Y$ is admissible.
It moves every coordinate, however, and the bounds below are informative only for copies that move few coordinates, or move them little.
For independent observations, there is a canonical copy that moves only the $i$-th coordinate: let $Y_1\pr,\dots,Y_n\pr$ be independent, independent of $\mo Y$, with $Y_j\pr\stackrel d=Y_j$, and set
\begin{equation}\label{eq:dep:replacement}
    Y_j^i := \begin{cases}
        Y_j, & j\neq i,\\
        Y_i\pr, & j=i.
    \end{cases}
\end{equation}
Then $\mo Y^i$ is an exchangeable copy of $\mo Y$ with $i$-independence.
This coordinate replacement underlies the stability arguments for independent data \cite{Escande2024Minimizers-JMLR,BrunelSerres2026FiniteSample-arxiv,Schoetz2026Weighted-arxiv}.

Under dependence, an exchangeable copy with $i$-independence redraws $Y_i$ and moves the other coordinates to restore the joint law, in a way that treats the two samples symmetrically.
The displacements $\ol{Y_k}{Y_k^i}$, $k\neq i$, then quantify how strongly the $i$-th observation influences the rest of the sample, in the metric of $\mc Q$.
The risk bounds of \cref{sec:2fmh,sec:pos,sec:power} are expressed through mixed moments of the form $\Eof{\ol{Y_i}{Y_i^i}^{r}\,\ol{Y_k}{Y_k^i}^{s}}$.
For instance, for a jointly Gaussian real sequence, the copy $Y_k^i:=Y_k-\beta_{ik}(Y_i-Y_i^i)$ with the regression coefficients $\beta_{ik}:=\Covof{Y_i,Y_k}/\Vof{Y_i}$ is admissible (\cref{prp:2fmh:gaussian-1d}), and its displacements are $\lvert \beta_{ik}\rvert\,\lvert Y_i-Y_i^i\rvert$; the transport-sensitivity coefficients of \cref{def:dep:transport} generalize $\lvert \beta_{ik}\rvert$.
The next two subsections give general constructions.

\subsection{Dependency neighborhoods}\label{ssec:dep:neighborhood}

The first construction applies under local dependence, encoded by dependency neighborhoods in the sense of \cite[Definition~3.5]{Ross2011Stein-PS}; this is a form of the local-dependence condition (LD1) of \cite{ChenShao2004LocalDependence-AOP}.

\begin{definition}\label{def:dep:neighborhood}
    Let $\mo Y=(Y_1,\dots,Y_n)$ and $N_i\subseteq\nnset n$ for each $i\in\nnset n$.
    The set $N_i$ is a \emph{dependency neighborhood} of $Y_i$ if $i\in N_i$ and $Y_i$ is independent of $(Y_j)_{j\in N_i\compl}$, where $N_i\compl=\nnset n\setminus N_i$.
\end{definition}

The following coupling lemma underlies the copies under local dependence.

\begin{lemma}[Proof in Appendix \ref{proof_lmm:aux:couple:triple}]\label{lmm:aux:couple:triple}
    Let $(A,B,C)$ be random variables taking values in Polish spaces $S_A, S_B, S_C$, respectively, and assume that $A$ and $C$ are independent.
    If $A$ or $B$, or both, are absent, they are replaced by constants in one-point spaces.
    Then there exists a triple $(A\pr,B\pr,C\pr)$, defined on a possibly extended probability space, such that
    \begin{enumerate}[label=(\roman*)]
        \item $(A,B,C,A\pr,B\pr,C\pr) \stackrel{d}{=} (A\pr,B\pr,C\pr,A,B,C)$
        \item $A\pr = A$ almost surely,
        \item $B$ and $B\pr$ are independent with identical distribution conditionally on $A$, and
        \item $C\pr$ is independent of $(A,B,C)$.
    \end{enumerate}
\end{lemma}

\begin{proposition}[Copies under local dependence]\label{prp:dep:neighborhood-copy}
    Let $N_i$ be a dependency neighborhood of $Y_i$.
    Then, on a possibly extended probability space, there exists an exchangeable copy $\mo Y^i$ of $\mo Y$ with $i$-independence such that
    \begin{enumerate}[label=(\roman*)]
        \item $Y_k^i=Y_k$ almost surely for every $k\in N_i\compl$;
        \item $Y_i^i$ is an independent copy of $Y_i$;
        \item for every $k\in N_i\setminus\{i\}$, the variables $Y_k$ and $Y_k^i$ are conditionally independent and identically distributed given $(Y_j)_{j\in N_i\compl}$; in particular, $Y_k^i\stackrel d=Y_k$.
    \end{enumerate}
\end{proposition}

\begin{proof}
    Define $A := (Y_j)_{j\in N_i\compl}$, $B := (Y_j)_{j\in N_i\setminus\{i\}}$, and $C := Y_i$.
    By assumption, $A$ and $C$ are independent.
    \cref{lmm:aux:couple:triple} gives $(A\pr,B\pr,C\pr)$ such that $(A,B,C,A\pr,B\pr,C\pr) \stackrel{d}{=} (A\pr,B\pr,C\pr,A,B,C)$, $A\pr=A$ almost surely, $B$ and $B\pr$ are conditionally iid given $A$, and $C\pr \perp (A,B,C)$.
    Define an $n$-tuple $\mo Y^i=(Y_1^i, \dots, Y_n^i)$ by
    \begin{equation*}
        A\pr =: (Y_j^i)_{j\in N_i\compl}\eqcm\quad B\pr =: (Y_j^i)_{j\in N_i\setminus\{i\}} \eqcm\quad  C\pr =: Y_i^i
        \eqfs
    \end{equation*}
    Exchangeability of $(\mo Y,\mo Y^i)$ is the first property of the lemma, and $Y_i^i=C\pr$ is independent of $\mo Y$, so $\mo Y^i$ is an exchangeable copy of $\mo Y$ with $i$-independence.
    Items (i)--(iii) are the remaining properties of the lemma; $Y_i^i\stackrel d=Y_i$ follows from exchangeability.
\end{proof}

Only the coordinates in $N_i$ move, so the dependence enters the risk bounds through the sizes $\#N_i$, which we summarize by
\begin{equation}\label{eq:dep:L}
    N_{\max} := \max_{i\in\nnset n}\#N_i
    \eqcm\quad
    \bar N := \frac{1}{n}\sum_{i=1}^n\#N_i\leq N_{\max}
    \eqfs
\end{equation}

\begin{remark}[Related local-dependence conditions]\label{rem:dep:neighborhood-related}
    Call a graph $G$ on $\nnset n$ a \emph{dependency graph} for $\mo Y$ if $(Y_i)_{i\in I}$ and $(Y_j)_{j\in J}$ are independent whenever $I$ and $J$ are disjoint and no edge of $G$ joins them \cite{BaldiRinott1989DependencyGraphs-AOP}.
    The closed neighborhoods of a dependency graph are dependency neighborhoods, so $N_{\max}$ is at most one plus the maximum degree of the graph.
    For $\ell\in\Nn$, call $\mo Y$ \emph{$\ell$-dependent} if $(Y_i)_{i\in I}$ and $(Y_j)_{j\in J}$ are independent whenever $I,J\subseteq\nnset n$ are nonempty with $\min\setByEleInText{\lvert i-j\rvert}{i\in I,\ j\in J}>\ell$; this is $m$-dependence in the terminology of \cite[Definition~1]{HoeffdingRobbins1948CLT-DMJ}.
    The band graph, with edges $\{i,j\}$ for $0<\lvert i-j\rvert\leq\ell$, is a dependency graph precisely when $\mo Y$ is $\ell$-dependent, and then $N_i=\setByEleInText{j\in\nnset n}{\lvert i-j\rvert\leq\ell}$ and $\#N_i\leq2\ell+1$.
    In the opposite direction, $j\notin N_i$ forces $Y_i\perp Y_j$, so $\#\setByEleInText{j\in\nnset n}{Y_i\not\perp Y_j}\leq\#N_i$.

    Neither implication reverses, as Bernstein's example shows; see, e.g., \cite{Stoyanov2013Counterexamples-Dover}.
    Let $Y_1,Y_2$ be independent Rademacher variables and $Y_3:=Y_1Y_2$.
    The three variables are pairwise independent, yet each is a function of the other two: the count above equals $1$, the smallest dependency neighborhoods have $\#N_i=2$, and every dependency graph for $\mo Y$ has a closed neighborhood of size $3$.
    Dependency neighborhoods therefore ask for more than a bound on the number of dependent pairs and less than a dependency graph: they use the independence of $Y_i$ from $(Y_j)_{j\in N_i\compl}$ as a whole, and never independence between two such families.
\end{remark}

\subsection{Transport sensitivity}\label{ssec:dep:transport}

Under local dependence, each $Y_i$ is independent of the family
$(Y_j)_{j\notin N_i}$. We now allow dependence across the whole sample and quantify how the conditional law of the remaining observations varies with the value of $Y_i$. The condition below requires a single coupling of the conditional laws given $Y_i=y$ and $Y_i=y\pr$ that satisfies coordinatewise Lipschitz-type transport bounds. The coefficients $b_{ik}$ in these bounds may decay with $\lvert i-k\rvert$, but no such decay is required in the definition.

\begin{definition}\label{def:dep:transport}
    Let $\mo Y = (Y_1,\dots,Y_n)$ be an $n$-tuple of $\mc Q$-valued random variables, let $s\in(0,1]$, and let $\mo b = (b_{ik})_{i,k\in\nnset n}\in\Rp^{n\times n}$ with $b_{ii}=1$.
    We define $\mo Y_{-i} := (Y_k)_{k\neq i}$, and fix a regular conditional distribution $K_i(y):=\mc L(\mo Y_{-i}\mid Y_i=y)$.
    We call $\mo Y$ \emph{$\mo b$-sensitive at exponent $s$} if for every $i\in\nnset n$ there is a probability kernel $(y,y\pr)\mapsto\pi^i_{y,y\pr}$ from $\mc Q^2$ to $\mc Q^{n-1}\times\mc Q^{n-1}$ such that, for $\mu_{i}\otimes\mu_{i}$-almost every $(y,y\pr)$, $\pi^i_{y,y\pr}$ is a coupling of $K_i(y)$ and $K_i(y\pr)$, $\pi^i_{y\pr,y}$ is the transpose of $\pi^i_{y,y\pr}$ (its image under exchanging the two blocks), and
    \begin{equation}\label{eq:dep:transport}
        \br{\int \ol{z_k}{z_k\pr}^{\,s}\;\dl\pi^i_{y,y\pr}(z,z\pr)}^{1/s}
        \;\leq\;
        b_{ik}\,\ol y{y\pr}
        \eqcm\quad k\neq i
        \eqfs
    \end{equation}
\end{definition}

For $s=1$, this says that displacing the $i$-th observation by $\varepsilon$ displaces the $k$-th one by at most $b_{ik}\varepsilon$ on average, and a single coupling achieves this for all $k$ simultaneously.
The symmetry requirement is no restriction: if $\pi^i$ satisfies the other requirements, then so does $\frac{1}{2}(\pi^i_{y,y\pr}+(\pi^i_{y\pr,y})\tr)$, which is symmetric, because the left-hand side of \eqref{eq:dep:transport}, raised to the power $s$, is linear in the coupling.
For Gaussian vectors, the operator norms of the regression matrices are admissible, and at every exponent they are the smallest admissible values (\cref{prp:dep:gaussian}).
By Lyapunov's inequality, $\mo b$-sensitivity at exponent $s$ implies $\mo b$-sensitivity at every exponent $s\pr\in(0,s]$, with the same couplings and the same matrix $\mo b$; smaller exponents are thus weaker requirements.
The results for $2$-Fréchet means use the exponent $s=1$, and those for power Fréchet means of order $\alpha\in(1,2)$ the exponent $s=\alpha/2$.
The couplings yield exchangeable copies whose displacements are controlled by $\mo b$.

\begin{proposition}[Copies under transport sensitivity]\label{prp:dep:transport-copy}
    Let $\mo Y$ be $\mo b$-sensitive at exponent $s\in(0,1]$ and let $i\in\nnset n$.
    Then, on a possibly extended probability space, there exists an exchangeable copy $\mo Y^i$ of $\mo Y$ with $i$-independence such that $Y_i^i$ is an independent copy of $Y_i$ and, for every $r\geq0$, every $s\pr\in(0,s]$, and every $k\in\nnset n$,
    \begin{equation}\label{eq:dep:mixed-moment}
        \EOf{\ol{Y_i}{Y_i^i}^{\,r}\,\ol{Y_k}{Y_k^i}^{\,s\pr}}
        \;\leq\;
        b_{ik}^{\,s\pr}\,\EOf{\ol{Y_i}{Y_i^i}^{\,r+s\pr}}
        \eqfs
    \end{equation}
    Both sides may be infinite.
\end{proposition}

\begin{proof}
Let $K_i$ and $\pi^i$ be the conditional laws and coupling
kernel from \cref{def:dep:transport}. We first specify the joint law of the two samples. Take $Y_i,Y_i^i$ independently with law $\mu_i$ and, conditionally on $(Y_i,Y_i^i)=(y,y\pr)$, draw $(\mo Y_{-i},\mo Y^i_{-i})\sim\pi^i_{y,y\pr}$.

For $\mu_i\otimes\mu_i$-almost every $(y,y\pr)$, the first marginal of $\pi^i_{y,y\pr}$ is $K_i(y)$. Thus $\mo Y$ has the prescribed law, and $\mo Y_{-i}$ is conditionally independent of $Y_i^i$ given $Y_i$. Since $Y_i$ and $Y_i^i$ are independent, this gives $Y_i^i\perp\mo Y$. The second marginal is $K_i(y\pr)$, so
$\mo Y^i\stackrel d=\mo Y$.

Moreover, $\mu_i\otimes\mu_i$ is symmetric and $\pi^i_{y\pr,y}=(\pi^i_{y,y\pr})\tr$ almost everywhere. Hence the constructed pair is exchangeable. As in the proof of \cref{lmm:aux:couple:triple},
disintegrating this joint law with respect to its first sample realizes $\mo Y^i$ on an extension of the probability space carrying the given $\mo Y$.

It remains to prove \eqref{eq:dep:mixed-moment}. Fix $r\geq0$ and $s\pr\in(0,s]$. For $k\neq i$, Lyapunov's inequality and
\eqref{eq:dep:transport} give, almost surely,
$(\int f^{s\pr}\dl\pi^i_{y,y\pr})^{1/s\pr} \leq(\int f^{s}\dl\pi^i_{y,y\pr})^{1/s}$ for $f=\ol{z_k}{z_k\pr}$. 
For $k=i$, the claim \eqref{eq:dep:mixed-moment} is an identity with $b_{ii}=1$.
    For $k\neq i$, note that $\ol{Y_i}{Y_i^i}^{\,r}$ is measurable with respect to $\sigma(Y_i,Y_i^i)$, whereas, given $(Y_i,Y_i^i)=(y,y\pr)$, the pair $(\mo Y_{-i},\mo Y^i_{-i})$ has law $\pi^i_{y,y\pr}$.
    Hence
    \begin{align*}
        \EOf{\ol{Y_i}{Y_i^i}^{\,r}\,\ol{Y_k}{Y_k^i}^{\,s\pr}}
        =
        \EOf{\ol{Y_i}{Y_i^i}^{\,r}\;\EOf{\ol{Y_k}{Y_k^i}^{\,s\pr}\mid Y_i,Y_i^i}} \leq
        b_{ik}^{\,s\pr}\,\EOf{\ol{Y_i}{Y_i^i}^{\,r+s\pr}}
        \eqfs
    \end{align*}
\end{proof}

For quadratic Fréchet means, we only need the case $s=1$ and $r=s\pr=1$ of \eqref{eq:dep:mixed-moment}.
Since $Y_i^i$ is an independent copy of $Y_i$, this case gives $\Eof{\ol{Y_i}{Y_i^i}\,\ol{Y_k}{Y_k^i}}\leq 2b_{ik}\Vmof{Y_i}$.

\subsection{Examples}\label{ssec:dep:examples}

\begin{proposition}[Gaussian vectors]\label{prp:dep:gaussian}
Let $p\in\N$, and let $\mo Y=(Y_1,\dots,Y_n)$ be jointly Gaussian in $(\R^p)^n$ with $\Sigma_{ik}:=\Covof{Y_i,Y_k}\in\R^{p\times p}$.
Then $\mo Y$ is $\mo b$-sensitive at every exponent $s\in(0,1]$, with $b_{ii}=1$ and $b_{ik}=\opNormof{\Sigma_{ki}\Sigma_{ii}^{+}}$ for $k\neq i$, where $\Sigma_{ii}^{+}$ is the Moore--Penrose
pseudoinverse of $\Sigma_{ii}$. At each exponent $s\in(0,1]$, these off-diagonal coefficients are the smallest admissible ones.
If $\Sigma_{ii}$ is invertible, then $\Sigma_{ii}^{+}=\Sigma_{ii}^{-1}$.
\end{proposition}

\begin{proof}
Fix $i\in\nnset n$. We first construct a coupling of the conditional samples. For each $k\neq i$, define $B_{ki}:=\Sigma_{ki}\Sigma_{ii}^{+}$ and $R_k:=Y_k-B_{ki}Y_i$. If $v\in\ker\Sigma_{ii}$, then $v\tr Y_i$ is almost surely constant, so $\Sigma_{ki}v=0$. Thus $\ker\Sigma_{ii}\subseteq\ker\Sigma_{ki}$, and $B_{ki}\Sigma_{ii}=\Sigma_{ki}$. Then, $\Covof{R_k,Y_i}
    =\Sigma_{ki}-B_{ki}\Sigma_{ii}=0$

By joint Gaussianity, the residual vector $(R_k)_{k\neq i}$
is independent of $Y_i$. We then use the conditional laws $\mc L(\mo Y_{-i}\mid Y_i=y)=\mc L\brOf{(R_k+B_{ki}y)_{k\neq i}}$. Couple the laws at $y$ and $y\pr$ using the same residual
vector. This defines a measurable coupling kernel, whose transpose
is obtained by interchanging $y$ and $y\pr$. For every $k\neq i$, the displacement satisfies $\normof{z_k-z_k\pr} = \normof{B_{ki}(y-y\pr)} \leq \opNormof{B_{ki}}\normof{y-y\pr}$. This pointwise bound proves \eqref{eq:dep:transport} at every exponent $s\in(0,1]$.
It remains to prove minimality. Let $s\in(0,1]$, and let $\mo b$ be admissible at exponent $s$ with couplings $\pi^i_{y,y\pr}$.
For $\mu_{i}\otimes\mu_{i}$-almost every $(y,y\pr)$, under $\pi^i_{y,y\pr}$ both $U:=z_k-B_{ki}y$ and $U\pr:=z_k\pr-B_{ki}y\pr$ have the law of $R_k$.

Since $z_k-z_k\pr=B_{ki}(y-y\pr)+U-U\pr$, the triangle inequality and subadditivity of $x\mapsto x^s$ give
\[\normof{B_{ki}(y-y\pr)}^s \leq \int\normof{z_k-z_k\pr}^s\,
\dl\pi^i_{y,y\pr}(z,z\pr)+C \leq \widetilde b_{ik}^{\,s}\normof{y-y\pr}^s+C,
\]
where $C:=2\EOf{\normof{R_k-\Eof{R_k}}^s}<\infty$ does not depend on $y,y\pr$ or on the coupling.

The difference of two independent draws from $\mu_i$ has law $N(0,2\Sigma_{ii})$, whose support is $\operatorname{ran}\Sigma_{ii}$. By continuity, the preceding inequality therefore gives $ \normof{B_{ki}u}^s \leq  \widetilde b_{ik}^{\,s}\normof u^s+C$ for every $u\in\operatorname{ran}\Sigma_{ii}.$
Replacing $u$ by $\lambda u$, dividing by $\lambda^s$, and letting $\lambda\to\infty$, we obtain
$\normof{B_{ki}u} \leq \widetilde b_{ik}\normof u$  for every 
$u\in\operatorname{ran}\Sigma_{ii}$.  Finally, $B_{ki}$ vanishes on
$\ker\Sigma_{ii}=(\operatorname{ran}\Sigma_{ii})^\perp$,
so its operator norm is determined by its restriction
to $\operatorname{ran}\Sigma_{ii}$. Hence $\widetilde b_{ik}\geq\opNormof{B_{ki}}$, as claimed.
\end{proof}

\begin{proposition}[Bi-Lipschitz images]\label{prp:dep:link}
    Let $s\in(0,1]$, let $\mc H$ be a separable Hilbert space, let $\mc K\subseteq\mc H$ be closed, and let $\mo Z=(Z_1,\dots,Z_n)$ be a $\mc K$-valued tuple that is $\mo b$-sensitive at exponent $s$.
    Let $g\colon\mc K\to\mc Q$ satisfy $\lambda\normof{z-z\pr}\leq \ol{g(z)}{g(z\pr)}\leq \Lambda\normof{z-z\pr}$ for all $z,z\pr\in\mc K$ and some $0<\lambda\leq \Lambda$.
    Then $\mo Y := (g(Z_t))_{t\in\nnset n}$ is $\mo b\pr$-sensitive at exponent $s$, where $b\pr_{ii}:=1$ and $b\pr_{ik}:=\tfrac{\Lambda}{\lambda}b_{ik}$ for $k\neq i$.
\end{proposition}

\begin{proof}
    The map $g$ is injective, and $g^{-1}\colon g(\mc K)\to\mc K$ is $\lambda^{-1}$-Lipschitz.
    As $g(\mc K)$ is complete, it is closed in $\mc Q$, so $Z_i=g^{-1}(Y_i)$ is a Borel function of $Y_i$ and $\sigma(Y_i)=\sigma(Z_i)$.
    
Fix $i\in\nnset n$, and take the coupling kernel $\pi^i_{z,z\pr}$ for $\mo Z$ from \cref{def:dep:transport}. Apply $g$ to every coordinate of both samples. For $\mc L(Z_i)\otimes\mc L(Z_i)$-almost every $(z,z\pr)$, the resulting measure couples the conditional laws of $\mo Y_{-i}$ given $Y_i=g(z)$ and $Y_i=g(z\pr)$. Parametrizing these couplings by $y=g(z)$ and
$y\pr=g(z\pr)$ gives a measurable kernel, since $g^{-1}$ is continuous. The kernel inherits symmetry from $\pi^i$ and can be
extended by a fixed symmetric probability measure outside $g(\mc K)^2$.

For each $k\neq i$, the Lipschitz bounds on $g$ and
\eqref{eq:dep:transport} give
    \begin{align*}
        \br{\int\ol{g(z_k)}{g(z_k\pr)}^{\,s}\dl\pi^i_{z,z\pr}}^{1/s}
        \leq \Lambda\br{\int\normof{z_k-z_k\pr}^{s}\dl\pi^i_{z,z\pr}}^{1/s}
\leq \Lambda b_{ik}\normof{z-z\pr}
        \leq \tfrac{\Lambda}{\lambda}b_{ik}\,\ol{g(z)}{g(z\pr)}
        \eqfs
    \end{align*}
\end{proof}

In \cref{prp:dep:link}, the data are modeled as the image of a latent process whose dependence is understood, and the geometry of $\mc Q$ enters only through the distortion $\Lambda/\lambda$. This distortion equals $1$ when $g$ is an isometry.

For symmetric positive definite $m\times m$ matrices with the Log-Cholesky metric, the map $P\mapsto\lfloor L(P)\rfloor+\log D(L(P))$, where $L(P)$ is the Cholesky factor, $\lfloor\cdot\rfloor$ its strictly lower triangular part, and $D(\cdot)$ its diagonal part, is an isometry onto $\R^{m(m+1)/2}$ \cite[Section~3]{Lin2019Cholesky-SIMAX}, so \cref{prp:dep:link} applies to its inverse with $\mc K=\R^{m(m+1)/2}$, for instance with a Gaussian latent sequence as in \cref{prp:dep:gaussian}.

For the $2$-Wasserstein space of distributions on a compact interval, quantile functions give an isometry onto a closed, bounded, convex set $\mc K\subseteq L^2(0,1)$, see, e.g., \cite{PanaretosZemel2020Wasserstein-Springer}, so the latent process must take values in this set; nondegenerate Gaussian processes do not, since their supports are unbounded and $\mc K$ is bounded.
In both examples $\mc Q$ is flat, and the Fréchet mean is a linear average in the latent coordinates.
Curved spaces require maps $g$ that are only bi-Lipschitz.

For positive definite $m\times m$ matrices, $m\geq2$, with the affine-invariant metric, no such map onto the whole space exists: the space contains totally geodesic hyperbolic planes, which admit no bi-Lipschitz embedding into a Hilbert space, since they contain copies of the finite complete binary trees with uniformly bounded distortion~\cite{Sarkar2011Delaunay-GD}, and these trees do not embed with uniformly bounded distortion into any superreflexive Banach space, in particular not into a Hilbert space~\cite[Theorem~1]{Bourgain1986Superreflexivity-IJM}.
There, \cref{prp:dep:link} applies only to processes confined to smaller regions, such as bounded balls, on which the distortion $\Lambda/\lambda$ grows with the radius.

For the phylogenetic tree spaces of \cite{BilleraEtAl2001PhylogeneticTrees-AAM}, the map $v$ that sends a tree to its vector of split lengths is an isometry on each orthant, and the metric is the length metric of the orthants, so that $\normof{v(x)-v(y)}\leq\ol xy$.
Conversely, $\ol xy\leq\sqrt2\,\normof{v(x)-v(y)}$, and the constant $\sqrt2$ cannot be improved. Indeed, let $x$ and $y$ be trees. We first bound $\ol xy$ by constructing a path through the face spanned by their common splits. Let $a$ be the Euclidean norm of the lengths of the splits present in $x$ but not in $y$, and define $b$ with the roles of $x$ and $y$ reversed. Let $c$ be the Euclidean norm of the length differences on the common splits. Then $\normof{v(x)-v(y)}^2=a^2+b^2+c^2$.  If $a+b=0$, the trees lie in a common orthant, and the straight segment between them gives $\ol xy\leq c=\normof{v(x)-v(y)}$.

Suppose now that $a+b>0$, and set $t:=a/(a+b)$. Choose $p$ in the face spanned by the common splits, with lengths obtained by linear interpolation from those of $x$ to those of $y$ at parameter $t$. The straight segments from $x$ to $p$ and from $p$ to $y$ lie in the orthants of $x$ and $y$, respectively.  Their total length is
    \begin{equation*}
        \sqrt{a^2+t^2c^2} + \sqrt{b^2+(1-t)^2c^2} = \sqrt{(a+b)^2+c^2}
        \eqfs
    \end{equation*}
    Hence
    \begin{equation*}
        \ol xy \leq \sqrt{(a+b)^2+c^2} \leq \sqrt{2a^2+2b^2+c^2} \leq \sqrt2\,\normof{v(x)-v(y)}
        \eqfs
    \end{equation*}
 It remains to show that the constant $\sqrt2$ cannot be improved. Consider trees with four labelled endpoints, that is, three leaves and the root in the setting of \cite{BilleraEtAl2001PhylogeneticTrees-AAM}. The three possible internal splits are pairwise incompatible. Thus every path between trees with different internal splits must pass through the face with no internal splits.

Take $x$ and $y$ with different internal splits of the same length $a>0$, and identical pendant edge lengths when these are included. Then $c=0$ and $\normof{v(x)-v(y)}=\sqrt2\,a$.
    For every point $p$ of the face with no internal splits,
    \begin{equation*}
        \ol xp\geq\normof{v(x)-v(p)}\geq a
        \quad\text{and}\quad
        \ol py\geq\normof{v(p)-v(y)}\geq a
        \eqfs
    \end{equation*}
Every path from $x$ to $y$ therefore has length at least $2a$. Combining this with the upper bound gives $\ol xy=2a=\sqrt2\,\normof{v(x)-v(y)}$.
Hence $g=v^{-1}$ satisfies the assumptions of \cref{prp:dep:link} with $\lambda=1$ and $\Lambda=\sqrt2$ on the closed cone $\mc K=v(\mc Q)$, which is not convex in general.

\begin{proposition}[Reversible contractive chains]\label{prp:dep:kernel}
    Let $\rho\in[0,1)$ and $s\in(0,1]$.
    Let $\mo Y=(Y_1,\dots,Y_n)$ be $n$ consecutive observations of a stationary Markov chain on $\mc Q$ with marginal law $\mu$ and a transition kernel $P$ that is reversible with respect to $\mu$, that is, $\mu(\dl x)P(x,\dl y)=\mu(\dl y)P(y,\dl x)$.
    Assume that
    \begin{equation}\label{eq:dep:kernel-contraction}
        W_s\brOf{P(q,\cdot),P(p,\cdot)}\leq\rho^s\,\ol qp^s
        \quad\text{for all } q,p\in\mc Q
        \eqfs
    \end{equation}
    Then $\mo Y$ is $\mo b$-sensitive at exponent $s$ with $b_{ik}=\rho^{\lvert i-k\rvert}$.
\end{proposition}

\begin{remark}\label{rem:dep:kernel}
    For $s=1$ and $P(x,\cdot)\in\mc P^1(\mc Q)$ for all $x$, \eqref{eq:dep:kernel-contraction} says that $P$ has coarse Ricci curvature at least $1-\rho$ in the sense of \cite[Definition~3]{Ollivier2009RicciCurvature-JFA}.
    Metropolis--Hastings chains and random-scan Gibbs samplers are reversible.
    For a nonreversible chain, the conclusion holds if, in addition, a version $\hat P$ of the backward kernel, $\mu(\dl x)P(x,\dl y)=\mu(\dl y)\hat P(y,\dl x)$, satisfies \eqref{eq:dep:kernel-contraction}.
    Forward contraction alone does not suffice: the $\ms{AR}(1)$ process with Rademacher innovations of \cref{rem:dep:ar1} below has a $W_1$-contractive kernel, but its present determines its past through an expanding map.
    For linear averages, Joulin and Ollivier~\cite{JoulinOllivier2010Curvature-AOP} derive nonasymptotic error bounds for Markov chain Monte Carlo under forward contraction, without reversibility; we do not know whether forward contraction yields nonasymptotic bounds on the risk of Fréchet means by other arguments.
    The Markov property cannot be dropped.
    For independent Rademacher variables $\varepsilon_1,\varepsilon_2$, the periodic sequence $Y_{3j+1}=\varepsilon_1$, $Y_{3j+2}=\varepsilon_2$, $Y_{3j+3}=\varepsilon_1\varepsilon_2$ is stationary with independent consecutive pairs; its one-step kernel is therefore constant, reversible, and satisfies \eqref{eq:dep:kernel-contraction} with $\rho=0$, so that the conclusion would be $b_{ik}=0$ for $k\neq i$.
    However, $Y_{3j+1}=Y_1$ for all $j$, so every coupling of the conditional laws given $Y_1=y$ and $Y_1=y\pr$ moves these coordinates by $\lvert y-y\pr\rvert$.
    Hence every admissible matrix has $b_{1,3j+1}\geq1$, and the row sum $\sum_{k}b_{1k}$ is at least $n/3$.
    Accordingly, the sample mean does not converge to the common mean $0$ of the marginals: $\Eof{(n^{-1}\sum_{i=1}^nY_i)^2}=1/3$ for all $n\in3\N$.
\end{remark}

\begin{proof}
We prove the nonreversible statement of \cref{rem:dep:kernel}; for a reversible chain, $\hat P=P$ is a version of the backward kernel.  Fix $i\in\nnset n$. By the Markov property, conditionally on $Y_i=y$, the future $(Y_{i+1},\dots,Y_n)$ and the past $(Y_{i-1},\dots,Y_1)$ are independent, the future is a Markov chain with kernel $P$ started at $y$, and the past, in reversed time, is a Markov chain with kernel $\hat P$ started at $y$; this defines a version $K_i$ of $\mc L(\mo Y_{-i}\mid Y_i=y)$. The cost $\ol{z}{z\pr}^s$ is continuous and nonnegative, and $(q,p)\mapsto(P(q,\cdot),P(p,\cdot))$ is measurable into $\mc P(\mc Q)^2$ with the weak topology, so \cite[Corollary~5.22]{Villani2009OptimalTransport-Springer} gives a measurable choice $(q,p)\mapsto\gamma_{q,p}$ of optimal plans between $P(q,\cdot)$ and $P(p,\cdot)$ for this cost.

The symmetrized plans $\tilde\gamma_{q,p}:=\frac{1}{2}(\gamma_{q,p}+\gamma_{p,q}\tr)$ are again couplings of $P(q,\cdot)$ and $P(p,\cdot)$ with cost $W_s(P(q,\cdot),P(p,\cdot))\leq\rho^s\ol qp^s$, depend measurably on $(q,p)$, and satisfy $\tilde\gamma_{p,q}=\tilde\gamma_{q,p}\tr$.
    
Define $\hat\gamma_{q,p}$ in the same way from $\hat P$. Given $(y,y\pr)$, let $(Z_i,Z_i\pr):=(y,y\pr)$ and, for $t\geq i$, draw $(Z_{t+1},Z_{t+1}\pr)\sim\tilde\gamma_{Z_t,Z_t\pr}$; for $t\leq i$, draw $(Z_{t-1},Z_{t-1}\pr)\sim\hat\gamma_{Z_t,Z_t\pr}$, independently of the forward steps.

By the Ionescu Tulcea theorem, this defines a probability kernel, and we let $\pi^i_{y,y\pr}$ be the law of $((Z_k)_{k\neq i},(Z_k\pr)_{k\neq i})$.
Its marginals are $K_i(y)$ and $K_i(y\pr)$, and the symmetry of $\tilde\gamma$ and $\hat\gamma$ gives $\pi^i_{y\pr,y}=(\pi^i_{y,y\pr})\tr$.
Furthermore, $\Eof{\ol{Z_{t+1}}{Z_{t+1}\pr}^s\mid Z_t,Z_t\pr}\leq\rho^s\ol{Z_t}{Z_t\pr}^s$, so induction gives $\Eof{\ol{Z_{i+h}}{Z_{i+h}\pr}^s}\leq\rho^{sh}\ol y{y\pr}^s$, and likewise backward. This is \eqref{eq:dep:transport} with $b_{ik}=\rho^{\lvert i-k\rvert}$.
\end{proof}

A simple instance of \cref{prp:dep:kernel} is the refresh chain.
For $\rho\in[0,1)$ and $\mu\in\mc P(\mc Q)$, the refresh chain is the stationary Markov chain with marginal law $\mu$ that, at each step, retains its current observation with probability $\rho$ and otherwise draws a fresh observation independently from $\mu$.
The refresh chain is used in \cref{prp:2fmh:refresh}, in \eqref{eq:pos:refresh}, and in \cref{exa:2fmh:decl:refresh-tripod,exa:pos:sphere,exa:power:decl:refresh}.

\begin{proposition}[Refresh chain]\label{prp:dep:refresh}
    Let $n\in\N$, $\mu\in\mc P(\mc Q)$, and $\rho\in[0,1)$.
    Let $\mo Y=(Y_1,\dots,Y_n)$ consist of $n$ consecutive observations of the stationary Markov chain with marginal law $\mu$ and transition kernel
    \begin{equation}\label{eq:dep:refresh:kernel}
        P(x,\cdot)=\rho\,\delta_x+(1-\rho)\,\mu \eqcm\quad x\in\mc Q
        \eqfs
    \end{equation}
    Then, for every $s\in(0,1]$, $\mo Y$ is $\mo b$-sensitive at exponent $s$ with $b_{ik}=\rho^{\lvert i-k\rvert/s}$, $i,k\in\nnset n$.
    Moreover, the couplings can be chosen such that
    \begin{equation}\label{eq:dep:refresh:transport}
        \int \ol{z_k}{z_k\pr}^{\,s}\, \dl\pi^i_{y,y\pr}(z,z\pr) = \rho^{\abs{i-k}}\,\ol y{y\pr}^{\,s}
        \eqcm\quad k\neq i
        \eqcm
    \end{equation}
    so that \eqref{eq:dep:transport} holds with equality.
    If $\mu$ is not a Dirac measure, these coefficients are the smallest admissible ones.
\end{proposition}

\begin{proof}
    The measure \[\mu(\dl x)P(x,\dl y) = \rho\,\mu(\dl x)\delta_x(\dl y) +(1-\rho)\mu(\dl x)\mu(\dl y)\] is symmetric in $x,y$, so $\mu$ is invariant and the chain is reversible.

    We construct the conditional coupling.
    Independently at each of the $n-1$ gaps between observations, place a block boundary with probability $1-\rho$.
    Conditionally on the resulting partition, assign independent labels with distribution $\mu$ to the blocks, and let each observation equal the label of its block.
    This construction has initial law $\mu$ and transition kernel \eqref{eq:dep:refresh:kernel}, and therefore has the same law as $\mo Y$.

    Fix $i\in\nnset n$.
    Given the partition and the labels of the blocks not containing $i$, the label $Y_i$ still has law $\mu$; hence $Y_i$ is independent of the partition together with those labels.
    Consequently, a version of $\mc L(\mo Y_{-i}\mid Y_i=y)$ is obtained by drawing the partition and the labels of the blocks not containing $i$ as above, assigning the label $y$ to the block containing $i$, and omitting coordinate $i$.

    For $y,y\pr\in\mc Q$, use the same partition and the same labels on all blocks not containing $i$ in two copies, and assign the labels $y$ and $y\pr$, respectively, to the block containing $i$.
    Let $\pi^i_{y,y\pr}$ be the joint law of the two tuples with coordinate $i$ omitted.
    This is a coupling of the two conditional laws; it is measurable in $(y,y\pr)$, since there are finitely many partitions and the construction is measurable for each partition, and interchanging $y$ and $y\pr$ transposes it.
    Under $\pi^i_{y,y\pr}$, we have $\ol{z_k}{z_k\pr}=\ol y{y\pr}$ if $k$ lies in the block of $i$, and $\ol{z_k}{z_k\pr}=0$ otherwise.
    Since $i$ and $k$ belong to the same block with probability $\rho^{\lvert i-k\rvert}$, we obtain \eqref{eq:dep:refresh:transport}, which is \eqref{eq:dep:transport} with $b_{ik}=\rho^{\lvert i-k\rvert/s}$, even with equality.

    For minimality, let $k\neq i$ and $h:=\lvert i-k\rvert$.
    The $k$-th marginal of $K_i(y)$ is $\rho^h\delta_y+(1-\rho^h)\mu$.
    For $y,y\pr\in\mc Q$, the function $f:=\min\{\ol\cdot{y\pr}^s,\ol y{y\pr}^s\}$ is $1$-Lipschitz with respect to the metric $d^s$, so every coupling of the $k$-th marginals of $K_i(y)$ and $K_i(y\pr)$ has cost at least $\rho^h(f(y)-f(y\pr))=\rho^h\ol y{y\pr}^s$.
    Hence every admissible coefficient satisfies $b_{ik}^s\ol y{y\pr}^s\geq\rho^h\ol y{y\pr}^s$ for $\mu\otimes\mu$-almost every $(y,y\pr)$, and, if $\mu$ is not a Dirac measure, this set contains pairs with $y\neq y\pr$, so that $b_{ik}\geq\rho^{h/s}$.
\end{proof}

The coupling is built from a block representation of the chain: block boundaries between consecutive observations are placed independently with probability $1-\rho$, the blocks receive independent labels with law $\mu$, and each observation equals the label of its block.
The blocks are determined by the latent refresh decisions, so adjacent blocks may carry the same label when $\mu$ has atoms.
The coefficients themselves also follow from \cref{prp:dep:kernel}, as the kernel \eqref{eq:dep:refresh:kernel} is reversible and satisfies \eqref{eq:dep:kernel-contraction} with $\rho^{1/s}$ in place of $\rho$.

\subsection{Relation to other notions of dependence}\label{ssec:dep:relations}

\begin{remark}[$\tau$-dependence]\label{rem:dep:tau}
    Let $\mo Y$ be $\mo b$-sensitive at exponent $1$, with marginals in $\mc P^1(\mc Q)$.
    For $i\neq k$, the coefficient of Dedecker and Prieur~\cite[Section~7, (7.1)]{DedeckerPrieur2005Dependence-PTRF}, written in Wasserstein form by Kantorovich--Rubinstein duality, satisfies
    \begin{equation}\label{eq:dep:tau}
        \tau\brOf{\sigma(Y_i),Y_k} := \EOf{W_1\brOf{\mc L(Y_k\mid Y_i),\mu_k}}
        \leq
        b_{ik}\,\EOf{\ol{Y_i}{Y_i\pr}}
        \leq
        b_{ik}\sqrt{2\VmOf{Y_i}}
        \eqcm
    \end{equation}
    where $Y_i\pr$ is an independent copy of $Y_i$.
    The first inequality follows from convexity of $W_1$ in its second argument and averaging \eqref{eq:dep:transport} over $\mu_i\otimes\mu_i$.
    The second follows from the Cauchy--Schwarz inequality under the additional assumption $\mu_i\in\mc P^2(\mc Q)$.

    This bound concerns conditioning on a single observation rather than on the whole past.
    For Markov chains and $k>i$, $\tau(\sigma(Y_i),Y_k) = \tau(\sigma(Y_j:j\leq i),Y_k)$.
    This need not hold in general.
    For instance, let $Y_1,\dots,Y_{n-1}$ be iid $N(0,1)$ and $Y_n:=(n-1)^{-1/2}\sum_{j<n}Y_j$; then \cref{prp:dep:gaussian} gives admissible coefficients $b_{ik}\leq(n-1)^{-1/2}$ for $i\neq k$, whereas
    \begin{equation*}
        \tau(\sigma(Y_1,\dots,Y_{n-1}),Y_n)=2/\sqrt\pi
        \eqfs
    \end{equation*}
    Conversely, neither geometric decay of the forward coefficients $\tau(\sigma(Y_j:j\leq i),Y_{i+h})$ in $h$ (\cref{rem:dep:ar1}) nor vanishing of all $\tau(\sigma(Y_i),Y_k)$, $i\neq k$ (\cref{rem:dep:conditional-index}), guarantees small transport-sensitivity coefficients.
    The condition \eqref{eq:dep:transport} is uniform in the conditioning value and symmetric in time, since it conditions on the single variable $Y_i$ rather than on a past, and it couples all remaining coordinates simultaneously.
\end{remark}

\begin{remark}[An expanding past]\label{rem:dep:ar1}
    Transport sensitivity treats past and future coordinates alike, and this forces large coefficients for processes whose past is an expanding function of their present.
    Consider the stationary real-valued process $Y_{t+1}=\rho Y_t+\xi_{t+1}$ with iid innovations $\xi_t$ uniform on $\{-1,+1\}$ and $0<\rho<1/2$.
    Here the present determines the past, $Y_{i-h}=T^h(Y_i)$ for a map $T$ with slope $\rho^{-1}$ on each sign branch, and every admissible coefficient satisfies $b_{i,i-h}\geq\rho^{-h}$.
    Up to an affine change of scale, the case $\rho=1/2$ is the classical example of an autoregressive process that is not strong mixing~\cite{Andrews1984NonStrongMixing-JAP}, for the same reason.
    The process nevertheless satisfies geometric moment contraction in the sense of \cite{WuShao2004IteratedRandom-JAP}, and its forward coefficients $\tau(\sigma(Y_j:j\leq i),Y_{i+h})$ decay geometrically in $h$.
    Hence neither geometric moment contraction nor forward $\tau$-dependence implies $\mo b$-sensitivity with summable coefficients, even on $\R$.
    With iid $N(0,1)$ innovations and $0<\rho<1$, in contrast, \cref{prp:dep:gaussian} gives $b_{ik}=\rho^{\lvert i-k\rvert}$.
    The consequences for the risk bounds are discussed in \cref{rem:2fmh:decl:limits}.
\end{remark}

\begin{remark}[Pairwise indices]\label{rem:dep:conditional-index}
    For identically distributed, nondegenerate observations, Wiesel's Wasserstein correlation coefficient \cite[Definition~2.1]{Wiesel2022Association-Bernoulli}, evaluated at $\mc L(Y_i,Y_k)$, and the conditional index of dependence of \cite[Section~4.2]{CatalanoLavenant2025Dependence-arxiv} with $p=1$ equal $\tau(\sigma(Y_i),Y_k)/\Eof{\ol{Y_k}{Y_k\pr}}$, where $Y_k\pr$ is an independent copy of $Y_k$, so \eqref{eq:dep:tau} bounds them by $b_{ik}$.
    Such pairwise indices cannot replace the simultaneous conditional couplings required by $\mo b$-sensitivity.
    In Bernstein's example of \cref{rem:dep:neighborhood-related}, all pairwise indices vanish, and the pairwise analogue $W_1(\mc L(Y_k\mid Y_i=y),\mc L(Y_k\mid Y_i=y\pr))\leq b_{ik}\,\ol y{y\pr}$ of \eqref{eq:dep:transport}, which corresponds to the Lipschitz-kernel assumption of \cite[Assumption~6.1]{Wiesel2022Association-Bernoulli}, holds with $b_{ik}=0$ for $i\neq k$.
    However, the conditional laws of $(Y_2,Y_3)$ given $Y_1=1$ and $Y_1=-1$ are uniform on $\{(1,1),(-1,-1)\}$ and on $\{(1,-1),(-1,1)\}$, so every coupling moves exactly one coordinate by $2=\lvert 1-(-1)\rvert$, and hence $b_{12}^s+b_{13}^s\geq1$ at every exponent $s\in(0,1]$; in particular, $b_{12}+b_{13}\geq1$ at exponent $1$.
\end{remark}

The matrix $\mo b$ is related to the interdependence matrices of dependent systems.
Dobrushin's interdependence matrix \cite{Dobrushin1970ConditionalDistributions-TPA} bounds, in the $1$-Wasserstein distance with respect to a metric on the state space, how the conditional law of one coordinate given all others changes with a single other coordinate; for the discrete metric, this distance is the total variation distance.
The total-variation form controls covariances of Gibbs measures \cite{Foellmer1982Covariance-JFA}, and Wu \cite{Wu2006Poincare-AOP} uses the $1$-Wasserstein form for Poincar\'e and transportation inequalities.
Marton's contraction matrices \cite{Marton2004Concentration-AOP} likewise condition on all coordinates outside a patch and use a Wasserstein distance for the Euclidean metric, whereas the mixing matrices of \cite{KontorovichRamanan2008Concentration-AOP} and the Marton couplings of \cite{Paulin2015Concentration-EJP} condition on the past and use the Hamming metric.
In contrast, \eqref{eq:dep:transport} conditions on a single coordinate and couples all remaining coordinates simultaneously.
For a systematic account of dependence measures built from Wasserstein distances, see \cite{CatalanoLavenant2025Dependence-arxiv}.
The term sensitivity is also used in \cite{AnsariEtAl2026Dependence-Bernoulli} for a real-valued dependence measure built from how conditional distributions vary with the conditioning value; there it quantifies the strength of dependence between two variables, whereas here $\mo b$ enters an error bound.

\section{Quadratic Fréchet means in Hadamard spaces}\label{sec:2fmh}

In this section, $(\mc Q,d)$ is a separable Hadamard space and $\alpha=2$ in the sampling notation \eqref{eq:pre:sampling}.
After the preliminaries (\cref{ssec:2fmh:pre}), we state the general coupling bound (\cref{ssec:2fmh:main}) and compare it with the exact error in Hilbert spaces (\cref{ssec:2fmh:illu}).
We then apply it under dependency neighborhoods (\cref{ssec:2fmh:strucdep}) and under transport sensitivity (\cref{ssec:2fmh:decl}).

\subsection{Preliminaries}\label{ssec:2fmh:pre}

Let $(\mc Q,d)$ be a separable Hadamard space.
For $\mu\in\mc P^1(\mc Q)$, the $2$-Fréchet mean $m_2(\mu)$ exists and is unique \cite[Proposition~4.3]{Sturm2003NPC-ContempMath}.
We use the following properties of Hadamard spaces and of $2$-Fréchet means therein; they are stated for finite second moments, which is the setting of our results, and extend to finite first moments \cite[Propositions~4.3, 4.4, Theorem~6.3]{Sturm2003NPC-ContempMath}.

\begin{enumerate}
    \item[(QI)] \phantomsection\label{sqr:QI}
    \textbf{Quadruple Inequality.}
    For $q,p,y,z\in\mc Q$, we have \cite[Proposition~2.4]{Sturm2003NPC-ContempMath}
    \begin{equation*}
        \ol yq^2 - \ol yp^2 - \ol zq^2 + \ol zp^2 \leq 2 \, \ol qp \, \ol yz
        \eqfs
    \end{equation*}
    \item[(VI)] \phantomsection\label{sqr:VI}
    \textbf{Variance Inequality.}
    Let $\mu \in \mc P^2(\mc Q)$ and $m:=m_2(\mu)$.
    Then for all $q\in\mc Q$, we have \cite[Proposition~4.4]{Sturm2003NPC-ContempMath}
    \begin{equation*}
        \ol {m}{q}^2 \leq \int \brOf{\ol yq^2 - \ol y{m}^2} \dl \mu(y)
        \eqfs
    \end{equation*}
    \item[(CP)] \phantomsection\label{sqr:CP}
    \textbf{Contraction Property.}
    Let $\mu,\nu \in \mc P^2(\mc Q)$.
    Then, with the $1$-Wasserstein distance $W_1$ of \eqref{eq:pre:wasserstein}, we have \cite[Theorem~6.3]{Sturm2003NPC-ContempMath}
    \begin{equation*}
        \ol {m_2(\mu)}{m_2(\nu)}\leq W_1(\mu, \nu)
        \eqfs
    \end{equation*}
\end{enumerate}

For $\mu\in\mc P^2(\mc Q)$, the Fréchet and metric variances of \eqref{eq:pre:variances} are comparable,
\begin{equation}\label{eq:2fmh:variances}
    \Vf(\mu) \leq \Vm(\mu) \leq 2\Vf(\mu)
    \eqcm
\end{equation}
where the first inequality is \cite[Theorem~2]{SongMueller2026Dispersion-JASA} and the second follows from the triangle inequality \cite[Lemma~B.10]{Schoetz2026Weighted-arxiv}; they coincide if $\mc Q$ is a closed convex subset of a Hilbert space.
The metric variance controls the concentration of the sample Fréchet mean.
For iid data $Y_1, \dots, Y_n$ with law in $\mc P^2(\mc Q)$, we have
\begin{equation}\label{eq:2fmh:iid}
    \EOf{\ol{\mbar}{\mhat}^2} \leq \frac{1}{n}\VmOf{Y_1}
    \eqfs
\end{equation}
This bound is contained in the proof of \cite[Theorem~3]{BrunelSerres2026FiniteSample-arxiv}, case $\kappa\leq0$, whose statement gives the weaker $2\Vfof{Y_1}/n$.
Under an additional lower curvature bound, \cite[Corollary~3.4]{LeGouicEtAl2023FastConvergence-JEMS} gives the sharper $\Vfof{Y_1}/n$; as noted there, the argument does not extend to spaces without such a bound, such as the tripod of \cref{exa:2fmh:decl:refresh-tripod}.
Bound \eqref{eq:2fmh:iid} is an equality if $\mc Q$ is a closed convex subset of a Hilbert space.

\subsection{The coupling bound}\label{ssec:2fmh:main}

We bound the distance between the sample $2$-Fréchet mean $\mhat$ and the population $2$-Fréchet mean $\mbar$ of \eqref{eq:pre:sampling}.
The dependence enters the error bound through the displacements of exchangeable copies (\cref{def:dep:exchangeable}).

The coupling bounds of \cref{thm:2fmh:dependent,thm:pos:dependent,thm:power:dep} rest on one inequality, which we state for both geometric settings.
It is the exchangeable counterpart of \cite[Lemma~C.2]{Schoetz2026Weighted-arxiv}, where the $i$-th observation is resampled independently.

\begin{lemma}[Compensated replace-one stability, proof in \cref{proof_lmm:supp:loo}]\label{lmm:supp:loo}
    Assume one of the following.
    \begin{enumerate}[label=(\alph*)]
        \item $(\mc Q,d)$ is a separable Hadamard space, $\alpha\in(1,2]$, and $\mu_i\in\mc P^\alpha(\mc Q)$ for all $i\in\nnset n$; put $C_{\mathrm{QI}}:=2^{2-\alpha}\alpha$.
        \item $\alpha=2$, $(\mc Q,d)$, $\setC$, and $C_\Delta$ are as in \cref{ssec:pos:pre}, and $\mu_i\in\mc P(\setC)$ for all $i\in\nnset n$; put $C_{\mathrm{QI}}:=2C_\Delta$.
    \end{enumerate}
    Let $\mo Y$, $\mubar$, $\mbar$, $\muhat$, and $\mhat$ be as in \eqref{eq:pre:sampling}, and, for $i\in\nnset n$, let $\mo Y^i$ be an exchangeable copy of $\mo Y$ with $i$-independence and $\mhat^i=m_\alpha(\muhat^i)$.
    Then $\mhat$ and $\mhat^i$ are random variables, and
    \begin{equation}\label{eq:supp:loo}
    \begin{aligned}
        \EOf{\int\brOf{\ol y{\mhat}^\alpha-\ol y{\mbar}^\alpha}\dl\mubar(y)}
        +\EOf{\frac{1}{n}\sum_{j=1}^n\brOf{\ol{Y_j}{\mbar}^\alpha-\ol{Y_j}{\mhat}^\alpha}}
     \leq
        \frac{C_{\mathrm{QI}}}{2n}\sum_{i=1}^n\EOf{\ol{Y_i}{Y_i^i}^{\alpha-1}\,\ol{\mhat}{\mhat^i}}.
    \end{aligned}
    \end{equation}
   
\end{lemma}

In the first term of \eqref{eq:supp:loo}, the integral over $y$ is evaluated at the random point $\mhat$.
   Both terms on the left-hand side are nonnegative, by the minimizing properties of $\mbar$ and $\mhat$.

\begin{theorem}[Coupling bound]\label{thm:2fmh:dependent}
    Let $\mo Y = (Y_1, \dots, Y_n)$ be an $n$-tuple of $\mc Q$-valued random variables with arbitrary joint distribution and marginals $\mu_i \in \mc P^2(\mc Q)$, and let $\mbar$ and $\mhat$ be as in \eqref{eq:pre:sampling} with $\alpha=2$.
    For $i \in \nnset n$, let $\mo Y^i$ be an exchangeable copy of $\mo Y$ with $i$-independence, and let $\mhat^i:=m_2(\muhat^i)$.
    Then
    \begin{equation}\label{eq:2fmh:result_intermediate}
        \EOf{\ol {\mbar}{\mhat}^2}
        \leq
        \frac1{2n}\sum_{i=1}^n \EOf{\ol{Y_i}{Y_i^i}\, \ol{\mhat}{\mhat^i}},
    \end{equation}
    and 
    \begin{equation}\label{eq:2fmh:result_final}
        \EOf{\ol{\mbar}{\mhat}^2}
        \leq
        \frac1{2n^2}\sum_{i,k=1}^n \EOf{\ol{Y_i}{Y_i^i}\, \ol{Y_k}{Y_k^i}}
        \eqfs
    \end{equation}
\end{theorem}

\begin{proof}
    Apply \cref{lmm:supp:loo}\,(a) with $\alpha=2$, so that $C_{\mathrm{QI}}=2$.
    By \hyperref[sqr:VI]{(VI)} for $\mubar$ at $q=\mhat$, the integral in the first term of \eqref{eq:supp:loo} is at least $\ol{\mbar}{\mhat}^2$.
    The second term equals $\Eof{\int(\ol y{\mbar}^2-\ol y{\mhat}^2)\dl\muhat(y)}$, and \hyperref[sqr:VI]{(VI)} for $\muhat$ at $q=\mbar$ bounds this integral from below by $\ol{\mbar}{\mhat}^2$.
    Hence
    \begin{equation*}
        2\,\EOf{\ol{\mbar}{\mhat}^2}
        \leq
        \frac{1}{n}\sum_{i=1}^n \EOf{\ol{Y_i}{Y_i^i} \, \ol{\mhat}{\mhat^i}}
        \eqcm
    \end{equation*}
    which is \eqref{eq:2fmh:result_intermediate}.
    By \hyperref[sqr:CP]{(CP)} and the index-wise coupling $n^{-1}\sum_{k=1}^n\delta_{(Y_k,Y_k^i)}$ of $\muhat$ and $\muhat^i$,
    \begin{equation}\label{eq:supp:2fmh:cp-index}
        \ol{\mhat}{\mhat^i}\leq W_1(\muhat, \muhat^i) \leq \frac{1}{n} \sum_{k=1}^n \ol{Y_k}{Y_k^i}
        \eqcm
    \end{equation}
    and inserting this into \eqref{eq:2fmh:result_intermediate} gives \eqref{eq:2fmh:result_final}.
\end{proof}

We call \eqref{eq:2fmh:result_intermediate} the intermediate bound and \eqref{eq:2fmh:result_final} the coordinatewise bound.
Observe that the proof uses the compensated replace-one stability technique: it combines \hyperref[sqr:VI]{(VI)} and \hyperref[sqr:CP]{(CP)} with \cref{lmm:supp:loo}, which rests on \hyperref[sqr:QI]{(QI)} and is shared with \cref{sec:pos,sec:power}.
Replace-one stability arguments for independent data \cite{Escande2024Minimizers-JMLR,BrunelSerres2026FiniteSample-arxiv,Schoetz2026Weighted-arxiv,Schoetz2026Transformed-arxiv} use the coordinate-replacement copy \eqref{eq:dep:replacement}, which resamples only the $i$-th observation; here, for dependent data, the whole $n$-tuple $\mo Y$ is replaced by an exchangeable copy $\mo Y^i$ with $i$-independence.
Unlike distances to $\mbar$, the displacements $\ol{Y_k}{Y_k^i}$ and the variances $\Vmof{Y_k}$ and $\Vfof{Y_k}$ in the bounds of this section measure fluctuations of the individual observations; the bounds therefore vanish for deterministic observations, also at different locations.

\begin{corollary}\label{cor:2fmh:dependent}
    In the setting of \cref{thm:2fmh:dependent}, each of the following terms is an upper bound for $\Eof{\ol{\mbar}{\mhat}^2}$:
    \begin{enumerate}[label=(\roman*)]
        \item
        $\frac1{2n}\sum_{i=1}^n \Eof{\ol{Y_i}{Y_i^i} \, \ol{\mhat}{\mhat^i} } \leq \frac1{2n}\sum_{i=1}^n \Eof{\ol{Y_i}{Y_i^i}^2}^\frac{1}{2} \Eof{\ol{\mhat}{\mhat^i}^2}^\frac{1}{2}$
        \item
        $\frac1{2n}\sum_{i=1}^n \Eof{\ol{Y_i}{Y_i^i} \, W_1(\muhat, \muhat^i) } \leq \frac1{2n}\sum_{i=1}^n \Eof{\ol{Y_i}{Y_i^i}^2}^\frac{1}{2} \Eof{W_1(\muhat, \muhat^i) ^2}^\frac{1}{2}$
        \item
        $\frac1{2n^2}\sum_{i,k=1}^n \Eof{\ol{Y_i}{Y_i^i} \, \ol{Y_k}{Y_k^i} } \leq \frac{1}{n^2}\sum_{i=1}^n \sqrt{\Vmof{Y_i}} \sum_{k=1}^n \sqrt{\frac{1}{2} \Eof{\ol{Y_k}{Y_k^i}^2}}$
        \item
        $\frac{1}{n^2}\sum_{i,k=1}^n \sqrt{\Vmof{Y_i} \Vmof{Y_k}}$
    \end{enumerate}
\end{corollary}

\begin{proof}
    The first three items follow from \eqref{eq:2fmh:result_intermediate}, \eqref{eq:supp:2fmh:cp-index}, and the Cauchy--Schwarz inequality. For (iii), also $\Eof{\ol{Y_i}{Y_i^i}^2}=2\Vmof{Y_i}$, as $Y_i^i$ is an independent copy of $Y_i$.
    For the last item,  apply (iii) with $\mo Y^i$ an independent
    copy of the entire sample $\mo Y$ for each $i$. For these copies, $\frac{1}{2}\Eof{\ol{Y_k}{Y_k^i}^2}=\Vmof{Y_k}$ for every $i,k\in\nnset n$, which gives the bound.
\end{proof}

The term in (iv) depends only on the marginal laws $\mu_i$, not on the dependence structure of $\mo Y$; it is the right-hand side of (iii) for an independent copy $\mo Y^i$ of $\mo Y$, which is admissible for every joint law.
In a Hilbert space, it is attained by the fully dependent sample $Y_1=\dots=Y_n$.

\begin{remark}\label{rem:2fmh:onMainThm}
    \begin{enumerate}[label=(\roman*)]
        \item
        To obtain the best possible bound from \cref{thm:2fmh:dependent,cor:2fmh:dependent}, we take the infimum of the error bound over all couplings $\pi_i$ of $(\mo Y, \mo Y^i)$ under which they form an exchangeable pair of $n$-tuples with $i$-independence.
        In particular, we can write
        \begin{equation*}
            \EOf{\ol{\mbar}{\mhat}^2}
            \leq
            \frac{1}{n^2} \sum_{i=1}^n \sqrt{\VmOf{Y_i}} \inf_{\pi_i} \sum_{k=1}^n \sqrt{\frac{1}{2} \E_{\pi_{i}}\abOf{\ol{Y_k}{Y_k^i}^2}}
            \eqfs
        \end{equation*}
        \item
        For independent observations, the coordinate-replacement copies \eqref{eq:dep:replacement} give
        \begin{equation*}
            \EOf{\ol{\mbar}{\mhat}^2}
            \leq
            \frac1{2n^2} \sum_{i=1}^n \EOf{\ol{Y_i}{Y_i^i}^2}
            =
            \frac{1}{n^2} \sum_{i=1}^n \VmOf{Y_i}
            \eqfs
        \end{equation*}
        This is \cite[Theorem~3.1]{Schoetz2026Weighted-arxiv} with equal weights; in Hilbert spaces, it is an equality.
        \item
        Unlike in Hilbert spaces, the risk in Hadamard spaces is in general not a function of the laws of the pairs $(Y_j,Y_k)$, $j,k\in\nnset n$; see \cref{exa:supp:pairlaws}.
        \item         The coordinatewise bound \eqref{eq:2fmh:result_final} extends to weighted Fr\'echet means.
        Let $w_1,\dots,w_n\geq0$ with $\sum_iw_i=1$, let $\bar\mu_w:=\sum_iw_i\mu_i$ and $\hat\mu_w:=\sum_iw_i\delta_{Y_i}$, and let $\bar m_w:=m_2(\bar\mu_w)$ and $\hat m_w:=m_2(\hat\mu_w)$.
        Then
        \begin{equation*}
            \EOf{\ol{\bar m_w}{\hat m_w}^2}\leq\frac{1}{2}\sum_{i,k=1}^nw_iw_k\EOf{\ol{Y_i}{Y_i^i}\,\ol{Y_k}{Y_k^i}}
            \eqfs
        \end{equation*}
        Indeed, the proof of \cref{thm:2fmh:dependent} is linear in the weights of the empirical measure.
        Let $\hat m_w^i:=m_2(\sum_kw_k\delta_{Y_k^i})$.
        The proof of \cref{lmm:supp:loo}\,(a) with $\alpha=2$ and the weights $w_i$ in place of $1/n$ gives
        \begin{align*}
            &\EOf{\int\brOf{\ol y{\hat m_w}^2-\ol y{\bar m_w}^2}\dl\bar\mu_w(y)}
            +\EOf{\int\brOf{\ol y{\bar m_w}^2-\ol y{\hat m_w}^2}\dl\hat\mu_w(y)}
            \\&\quad\leq
            \sum_{i=1}^nw_i\EOf{\ol{Y_i}{Y_i^i}\,\ol{\hat m_w}{\hat m_w^i}}
            \eqcm
        \end{align*}
        as Fubini's theorem, \hyperref[sqr:QI]{(QI)}, and the distributional identities used there hold verbatim for each $i$.
        By \hyperref[sqr:VI]{(VI)} for $\bar\mu_w$ at $\hat m_w$ and for $\hat\mu_w$ at $\bar m_w$, each term on the left-hand side is at least $\Eof{\ol{\bar m_w}{\hat m_w}^2}$.
        By \hyperref[sqr:CP]{(CP)} and the coupling $\sum_kw_k\delta_{(Y_k,Y_k^i)}$, we have $\ol{\hat m_w}{\hat m_w^i}\leq\sum_kw_k\ol{Y_k}{Y_k^i}$.
        Combining these three bounds gives the claim.
    \end{enumerate}
\end{remark}

\subsection{Illustration in Hilbert spaces}\label{ssec:2fmh:illu}

In Hilbert spaces, the sample and population $2$-Fréchet means reduce to the arithmetic mean and the expectation, respectively. The mean squared error of the arithmetic mean has two exact expressions: one in terms of covariances, and one in terms of the exchangeable copies of \cref{thm:2fmh:dependent}.

\subsubsection{General comparison}\label{sssec:2fmh:illu:general}

Let $\mc H$ be a separable Hilbert space.
Let $X_1, \dots, X_n$ be $\mc H$-valued random variables with finite second moment, and let $\overline X_n := n^{-1}\sum_{i=1}^n X_i$.
For $i \in \nnset n$, let $\mo X^i = (X_1^i,\dots,X_n^i)$ be an exchangeable copy of $\mo X=(X_1,\dots,X_n)$ with $i$-independence, and define
\begin{equation}\label{eq:2fmh:illu:delta}
    D_{ik} := X_k - X_k^i \eqcm\quad k \in \nnset n
\eqfs
\end{equation}
Then, on the one hand, we have the exact identities
\begin{equation}\label{eq:2fmh:illu:identity}
    \EOf{\normOf{\overline X_n - \EOf{\overline X_n}}^2} =
    \frac{1}{n^2} \sum_{i,k=1}^n \traceOf{\CovOf{X_i, X_k}}
    =\frac1{2n^2}\sum_{i,k=1}^n \EOf{\ip{D_{ii}}{D_{ik}}}
    \eqcm
\end{equation}
where the second identity holds for each summand, $\Eof{\langle D_{ii}\,,\, D_{ik}\rangle}=2\trace(\Covof{X_i,X_k})$, by exchangeability and as $X_i^i$ is independent of $\mo X$.
On the other hand, \cref{thm:2fmh:dependent} yields the bound
\begin{equation}\label{eq:2fmh:illu:bound}
    \EOf{\normOf{\overline X_n - \EOf{\overline X_n}}^2}
    \le
    \frac1{2n^2}\sum_{i,k=1}^n \EOf{\normOf{D_{ii}}\,\normOf{D_{ik}}}
    \eqfs
\end{equation}
Hence, the gap between \eqref{eq:2fmh:illu:identity} and \eqref{eq:2fmh:illu:bound} is precisely the loss of exactness induced by the Cauchy--Schwarz inequality
\begin{equation}\label{eq:2fmh:illu:csi}
    \ip{D_{ii}}{D_{ik}}
    \le
    \normOf{D_{ii}}\,\normOf{D_{ik}}
    \eqfs
\end{equation}
Since $\Eof{\langle D_{ii}\,,\, D_{ik}\rangle}$ does not depend on the choice of coupling, a coupling of $\mo X$ and $\mo X^i$ is most useful for the coordinatewise bound \eqref{eq:2fmh:result_final} if it makes the norm-weighted angular defect $\Eof{\normof{D_{ii}}\normof{D_{ik}}-\langle D_{ii}\,,\, D_{ik}\rangle}$ small.

\subsubsection{Sharpness for Gaussians}\label{sssec:2fmh:illu:gauss}

Under sign conditions on the covariances, the bounds of \cref{thm:2fmh:dependent} are sharp for Gaussian random variables.
Let us first illustrate this on the real line.
Let $\mo X = (X_1, \dots, X_n)$ be an $n$-tuple of $\R$-valued random variables with finite second moment and $\Vof{X_j}>0$ for all $j\in\nnset n$.
For $i\in\nnset n$, let $\mo X^i$ be an exchangeable copy of $\mo X$ with $i$-independence.
For the differences from \eqref{eq:2fmh:illu:delta}, we have $\Eof{D_{ii}D_{ik}} = 2\Covof{X_i, X_k}$, and hence $\Eof{\lvert D_{ii}D_{ik}\rvert}\geq2\lvert \Covof{X_i,X_k}\rvert$, for every such copy.
Suppose that the correlations of the copies satisfy
\begin{equation}\label{eq:2fmh:illu:corr}
    \CorOf{X_k, X_k^i} = 1 - \CorOf{X_i, X_k}^2
    \eqcm\quad i,k\in\nnset n
    \eqfs
\end{equation}
Then, by \cref{lmm:supp:corr}, $\Eof{D_{ii}^2}^{1/2}\Eof{D_{ik}^2}^{1/2}=2\lvert \Covof{X_i,X_k}\rvert=\lvert \Eof{D_{ii}D_{ik}}\rvert$, so the Cauchy--Schwarz inequality for $\Eof{D_{ii}D_{ik}}$ is an equality.
Hence $D_{ik}$ is almost surely a multiple of $D_{ii}$, and comparing $\Eof{D_{ii}D_{ik}}=2\Covof{X_i,X_k}$ with $\Eof{D_{ii}^2}=2\Vof{X_i}$ identifies the multiple:
\begin{equation}\label{eq:2fmh:illu:regression}
    D_{ik} = \beta_{ik}D_{ii}
    \eqcm\quad
    \beta_{ik}:=\frac{\CovOf{X_i,X_k}}{\VOf{X_i}}
    \eqcm\quad i,k\in\nnset n
    \eqfs
\end{equation}
Conversely, \eqref{eq:2fmh:illu:regression} implies \eqref{eq:2fmh:illu:corr}, so the correlation identity and the regression identity are equivalent.
Under them, $\lvert D_{ii}D_{ik}\rvert=\lvert \beta_{ik}\rvert D_{ii}^2$, so $\Eof{\lvert D_{ii}D_{ik}\rvert}=2\lvert \Covof{X_i,X_k}\rvert$, and the coordinatewise bound \eqref{eq:2fmh:result_final} equals $n^{-2}\sum_{i,k=1}^n\lvert \Covof{X_i,X_k}\rvert$.
In addition, $\mhat-\mhat^i=\overline X_n-\overline X_n^i=(n^{-1}\sum_{k=1}^n\beta_{ik})D_{ii}$, so the intermediate bound \eqref{eq:2fmh:result_intermediate} equals
\begin{equation}\label{eq:2fmh:illu:rowsum}
    \frac1{2n}\sum_{i=1}^n\EOf{D_{ii}^2}\,\abs{\frac{1}{n}\sum_{k=1}^n\beta_{ik}}
    =
    \frac{1}{n^2}\sum_{i=1}^n\abs{\sum_{k=1}^n\CovOf{X_i,X_k}}
    \eqfs
\end{equation}
So far, we have assumed that the copies satisfy \eqref{eq:2fmh:illu:corr}.
Such copies always exist when $\mo X$ is jointly Gaussian; this is the content of the next proposition.

\begin{proposition}[Real Gaussian vectors]\label{prp:2fmh:gaussian-1d}
    Let $n\in\N$.
    Let $\mo X=(X_1,\dots,X_n)$ be jointly Gaussian in $\R^n$ with $\Vof{X_j}>0$ for all $j\in\nnset n$.
    Then, for $i \in\nnset n$, on a possibly extended probability space, we can extend $\mo X$ to a Gaussian random vector $(\mo X, \mo X^i)$ in $\R^{2n}$ such that $(\mo X, \mo X^i)$ is an exchangeable pair with $i$-independence and \eqref{eq:2fmh:illu:corr} and \eqref{eq:2fmh:illu:regression} hold.
    For these copies, the bounds \eqref{eq:2fmh:result_intermediate} and \eqref{eq:2fmh:result_final} of \cref{thm:2fmh:dependent} equal
    \begin{equation}\label{eq:2fmh:illu:gaussian-1d-sharp}
        \frac{1}{n^2}\sum_{i=1}^n\abs{\sum_{k=1}^n\CovOf{X_i,X_k}}
        \quad\text{and}\quad
        \frac{1}{n^2}\sum_{i,k=1}^n \abs{\CovOf{X_i,X_k}}
        \eqcm
    \end{equation}
    respectively.
    In particular, \eqref{eq:2fmh:result_intermediate} equals $\Vof{\overline X_n}$ if and only if $\sum_{k=1}^n\Covof{X_i,X_k}\geq0$ for all $i\in\nnset n$, and \eqref{eq:2fmh:result_final} equals $\Vof{\overline X_n}$ if and only if $\Covof{X_i,X_k} \geq 0$ for all $i,k\in\nnset n$.
\end{proposition}

\begin{proof}
 We create $\mo X^i$ by first letting $X_i^i$ have the same distribution as $X_i$ and be independent of $\mo X$, and then setting $X_k^i := X_k - \beta_{ik}(X_i - X_i^i)$ with $\beta_{ik}$ as in \eqref{eq:2fmh:illu:regression}, which is well defined because $\Vof{X_i}>0$. 
 
 As $\beta_{ik}=\Sigma_{ki}\Sigma_{ii}^{+}$ in the notation of \cref{prp:dep:gaussian} with $p=1$, the pair $(\mo X,\mo X^i)$ has the law of the pair that the proof of \cref{prp:dep:transport-copy} builds from the translation couplings in the proof of \cref{prp:dep:gaussian}.
Hence $(\mo X, \mo X^i)$ is an exchangeable pair with $i$-independence, and it is jointly Gaussian as a linear image of $(\mo X, X_i^i)$.

By construction, $D_{ik} = X_k-X_k^i = \beta_{ik} D_{ii}$ pointwise, which is \eqref{eq:2fmh:illu:regression}. We also observe that as $X_i^i$ is independent of $X_k$,
    \begin{align*}
        \VOf{X_k} \CorOf{X_k, X_k^i}
        &=
        \CovOf{X_k, X_k - \beta_{ik} D_{ii}}
        \\&=
        \VOf{X_k} - \frac{\CovOf{X_i, X_k}^2}{\VOf{X_i}}
        =
        \VOf{X_k} \br{1 - \CorOf{X_i, X_k}^2}
        \eqcm
    \end{align*}
    which is \eqref{eq:2fmh:illu:corr}. The remaining claims follow from \eqref{eq:2fmh:illu:corr}, \eqref{eq:2fmh:illu:regression}, and the computation leading to \eqref{eq:2fmh:illu:rowsum} in \cref{sssec:2fmh:illu:gauss}.
\end{proof}

The coordinatewise bound \eqref{eq:2fmh:result_final} cannot exploit cancellation between positive and negative covariances, by the lower bound $\Eof{\lvert D_{ii}D_{ik}\rvert}\geq2\lvert \Covof{X_i,X_k}\rvert$ above; hence the Gaussian construction minimizes \eqref{eq:2fmh:result_final} over all exchangeable copies with $i$-independence, even when this bound exceeds the true variance.
The intermediate bound \eqref{eq:2fmh:result_intermediate} does exploit such cancellation.
For a Gaussian $\ms{AR}(1)$ sequence with negative parameter $\rho\geq-1/2$, the intermediate bound equals $\Vof{\overline X_n}$, whereas the coordinatewise bound exceeds it (\cref{ssec:aux:ar1:negative}).

The next result extends the `if' part of the last claim of \cref{prp:2fmh:gaussian-1d}, a sufficient condition for equality in the coordinatewise bound \eqref{eq:2fmh:result_final}, to Gaussian vectors in $\R^p$ whose cross-covariance matrices are nonnegative multiples of a common positive definite matrix.

\begin{proposition}[Gaussian vectors with a common covariance shape]\label{prp:2fmh:gaussian-md}
    Let $n,p\in\N$, let $\mo X=(X_1,\dots,X_n)$ be jointly Gaussian in $(\R^p)^n$, and define $\Sigma_{ik}:=\Covof{X_i,X_k}\in\R^{p\times p}$. Assume that there exist a positive definite matrix $\Sigma_0\in\R^{p\times p}$ and nonnegative scalars $c_{ik}$, $i,k\in\nnset n$, with $c_{ii}>0$ for all $i\in\nnset n$, such that
    \begin{equation}\label{eq:2fmh:gaussian-md:common-shape}
        \Sigma_{ik}=c_{ik}\Sigma_0
        \quad
        \text{for all } i,k\in\nnset n
        \eqfs
    \end{equation}
    Then, on a possibly extended probability space, there exist exchangeable copies $\mo X^i=(X_1^i,\dots,X_n^i)$ of $\mo X$ with $i$-independence, $i\in\nnset n$, such that each pair $(\mo X,\mo X^i)$ is jointly Gaussian and
    \begin{equation*}
        \EOf{\normOf{\overline X_n-\EOf{\overline X_n}}^2}
        =
        \frac{1}{2n^2}\sum_{i,k=1}^n \EOf{\normOf{X_i-X_i^i}\,\normOf{X_k-X_k^i}}
        \eqfs
    \end{equation*}
\end{proposition}
\begin{proof}
By \eqref{eq:2fmh:gaussian-md:common-shape} and the symmetry of $\Sigma_0$, we have $\Sigma_{ki}\Sigma_{ii}^{-1}=(c_{ik}/c_{ii})I_p$, where the inverse exists because $\Sigma_{ii}=c_{ii}\Sigma_0$ is positive definite.

As in the proof of \cref{prp:2fmh:gaussian-1d}, letting $X_i^i$ have the same distribution as $X_i$ and be independent of $\mo X$, and setting $X_k^i := X_k - (c_{ik}/c_{ii})(X_i - X_i^i)$, yields a jointly Gaussian exchangeable pair $(\mo X,\mo X^i)$ with $i$-independence.

By construction, $D_{ik}=X_k-X_k^i=(c_{ik}/c_{ii})D_{ii}$ pointwise with $c_{ik}/c_{ii}\geq0$, so $\normof{D_{ii}}\normof{D_{ik}}=\langle D_{ii}, D_{ik}\rangle$, the Cauchy--Schwarz inequality \eqref{eq:2fmh:illu:csi} is an equality, and \eqref{eq:2fmh:illu:bound} coincides with \eqref{eq:2fmh:illu:identity}.
\end{proof}

Thus the coordinatewise bound \eqref{eq:2fmh:result_final} of \cref{thm:2fmh:dependent} is attained with equality. 

The copies in \cref{prp:2fmh:gaussian-1d,prp:2fmh:gaussian-md} are $X_k^i=X_k-\Sigma_{ki}\Sigma_{ii}^{-1}(X_i-X_i^i)$, which \cref{prp:dep:transport-copy} builds from the translation couplings in the proof of \cref{prp:dep:gaussian}.
For them, $\normof{D_{ik}}=b_{ik}\normof{D_{ii}}$ with the coefficients $b_{ik}$ of \cref{prp:dep:gaussian}, so the bound of \cref{thm:2fmh:decl} equals the coordinatewise bound \eqref{eq:2fmh:result_final}.
In particular, \cref{prp:2fmh:gaussian-md} also follows from \cref{thm:2fmh:decl} below and \cref{prp:dep:gaussian}: with $b_{ik}=c_{ik}/c_{ii}$ and $\Vmof{X_i}=c_{ii}\trace(\Sigma_0)$, the bound \eqref{eq:2fmh:decl:main-bound} equals $n^{-2}\sum_{i,k=1}^n\trace(\Covof{X_i,X_k})$.

\subsubsection{Comonotone functions of a common factor}\label{sssec:2fmh:illu:monotone}

An independent copy of $\mo X$ is always admissible (\cref{ssec:dep:copies}).
The next proposition gives a dependent setting on the real line in which this copy satisfies $\lvert D_{ii}D_{ik}\rvert=D_{ii}D_{ik}$ almost surely for all $i,k\in\nnset n$.
Comparing \eqref{eq:2fmh:illu:identity} and \eqref{eq:2fmh:illu:bound}, we conclude that the coordinatewise bound of \cref{thm:2fmh:dependent} holds with equality.

\begin{proposition}[Comonotone functions
of a common factor]\label{prp:real:monotone-factor} Let $U$ be a real-valued random variable and $f_1,\dots,f_n\colon\R\to\R$ monotone in the same direction.
Define $X_i := f_i(U)$, $i\in\nnset n$. Assume that $X_1,\dots,X_n$ are square-integrable, set $\mo X := (X_1,\dots, X_n)$, and let $\mo X^i = (X_1^i,\dots, X_n^i)$ be an independent copy of $\mo X$. Then $(\mo X,\mo X^i)$ is an exchangeable pair with $i$-independence, and 
    \[\VOf{\overline X_n} = \frac1{2n^2}\sum_{i,k=1}^n \EOf{\abs{X_i-X_i^i}\,\abs{X_k-X_k^i}} = \frac{1}{n^2}\sum_{i,k=1}^n \CovOf{X_i, X_k} \eqfs\]
\end{proposition}

\begin{proof}
For $u,v \in \R$, the differences $f_i(u) - f_i(v)$ and $f_k(u) - f_k(v)$ have the same sign, as $f_i$ and $f_k$ are monotone in the same direction.
As the claim depends only on $\mc L(\mo X,\mo X^i)=\mc L(\mo X)\otimes\mc L(\mo X)$, we may realize $X_k^i = f_k(U\pr)$ for an independent copy $U\pr$ of $U$; with $D_{ii}$, $D_{ik}$ as in \eqref{eq:2fmh:illu:delta}, we have
\begin{equation*}
        D_{ii} D_{ik} = (f_i(U) - f_i(U\pr))(f_k(U) - f_k(U\pr)) \geq 0
        \eqfs
    \end{equation*}
Hence, $\lvert D_{ii} D_{ik}\rvert = D_{ii} D_{ik}$ and the claim follows from \eqref{eq:2fmh:illu:identity} and \eqref{eq:2fmh:illu:bound}.
\end{proof}

The pointwise inequality $(f_i(u)-f_i(v))(f_k(u)-f_k(v))\geq0$ used in \cref{prp:real:monotone-factor} is also the key step in Chebyshev's association inequality $\Covof{f_i(U),f_k(U)}\geq0$; in the terminology of \cite[Section~2, (P3) and (P4)]{EsaryEtAl1967Association-AMStat}, the family $(f_j(U))_{j\in\nnset n}$ is associated.
Conversely, for the independent copy, the coordinatewise bound \eqref{eq:2fmh:illu:bound} is an equality only if $(X_i-X_i^i)(X_k-X_k^i)\geq0$ almost surely for all $i,k\in\nnset n$, as each of its summands is at least the corresponding summand of \eqref{eq:2fmh:illu:identity}.

\subsection{Dependency neighborhoods}\label{ssec:2fmh:strucdep}

Recall the dependency neighborhoods $N_i$ of \cref{def:dep:neighborhood} and their average size $\bar N=n^{-1}\sum_{i=1}^n\#N_i$ from \eqref{eq:dep:L}.

\begin{theorem}[Local dependence]\label{thm:2fmh:neighborhood}
    Let $\mo Y = (Y_1, \dots, Y_n)$ be an $n$-tuple of $\mc Q$-valued random variables with marginals $\mu_{i}\in\mc P^2(\mc Q)$ and dependency neighborhoods $(N_i)_{i\in\nnset n}$, and let $\mbar,\mhat$ be as in \cref{thm:2fmh:dependent}.
    Then
    \begin{equation}\label{eq:2fmh:neighborhood:bound}
    \begin{aligned}
        \EOf{\ol{\mbar}{\mhat}^2}
        &\leq
        \frac{1}{n^2}\sum_{i=1}^n \VmOf{Y_i}
        +
        \frac{\sqrt{2}}{n^2}\sum_{i=1}^n \sqrt{\VmOf{Y_i}} \sum_{k \in N_i\setminus\{i\}} \sqrt{\VfOf{Y_k}}
        \\&\leq
        \sqrt2\,\max_{i\in\nnset n}\VmOf{Y_i}\,\frac{\bar N}n
        \eqfs
    \end{aligned}
    \end{equation}
\end{theorem}

\begin{proof}
    For each $i\in\nnset n$, let $\mo Y^i$ be the exchangeable copy of \cref{prp:dep:neighborhood-copy}.
    \Cref{thm:2fmh:dependent} and the Cauchy--Schwarz inequality yield
    \begin{equation}\label{eq:2fmh:neighborhood:start}
        \EOf{\ol{\mbar}{\mhat}^2} \leq \frac{1}{2n^2}\sum_{i,k=1}^n \EOf{\ol{Y_i}{Y_i^i}^2}^{1/2}\EOf{\ol{Y_k}{Y_k^i}^2}^{1/2}
        \eqfs
    \end{equation}
Since $Y_i^i$ is an independent copy of $Y_i$, we have $\Eof{\ol{Y_i}{Y_i^i}^2} = 2 \Vmof{Y_i}$. This covers the diagonal term $k=i$ exactly. For $k\in N_i\compl$, we have $Y_k^i=Y_k$ almost surely, so $\Eof{\ol{Y_k}{Y_k^i}^2}=0$.

Finally, for $k\in N_i\setminus\{i\}$, the triangle inequality and Minkowski's inequality imply
    \begin{equation*}
        \EOf{\ol{Y_k}{Y_k^i}^2}^{1/2}
        \leq \EOf{\ol{Y_k}{m_k}^2}^{1/2}+\EOf{\ol{Y^i_k}{m_k}^2}^{1/2}
        =  2\sqrt{\VfOf{Y_k}},
    \end{equation*}
  where $m_k:=m_2(\mu_{k})$; the equality holds as $Y_k^i\stackrel{d}{=} Y_k$. Substituting these estimates into \eqref{eq:2fmh:neighborhood:start} gives the first bound in \eqref{eq:2fmh:neighborhood:bound}. The second follows from $\Vfof{Y_k} \leq \Vmof{Y_k}$, see \eqref{eq:2fmh:variances}, from $\Vmof{Y_k}\leq \max_{i\in\nnset n} \Vmof{Y_i}$, and from $1+\sqrt2\,(\#N_i-1)\leq\sqrt2\,\#N_i$.
\end{proof}

The first bound retains the individual variances. For independent observations, we may take $N_i=\{i\}$ for all $i\in\nnset n$, and this bound reduces to $n^{-2}\sum_{i=1}^n\Vmof{Y_i}$, recovering \cref{rem:2fmh:onMainThm}\,(ii).
The second bound expresses the dependence on the largest metric variance and the average neighborhood size.

\begin{remark}\label{rem:2fmh:neighborhood}
    The factor $\sqrt2$ in \cref{thm:2fmh:neighborhood} comes from the bound $\Eof{\ol{Y_k}{Y_k^i}^2}\leq4\Vfof{Y_k}$ for $k\in N_i\setminus\{i\}$.
    In Hilbert spaces, this bound improves to $2\Vfof{Y_k}$, and the factor $\sqrt2$ disappears.
    Indeed, by \cref{prp:dep:neighborhood-copy}\,(iii), $Y_k$ and $Y_k^i$ are conditionally iid given $A:=(Y_j)_{j\in N_i\compl}$, so the improvement follows from the law of total variance.
    The constant $4$ cannot be improved in general Hadamard spaces.
    To see this, consider the metric tree obtained by gluing $L\geq3$ unit intervals at a common endpoint $o$.
    Let $A$ take the values $0$ and $1$ with equal probability.
    Given $A=0$, set $Y_k=Y_k^i=o$; given $A=1$, let $Y_k$ and $Y_k^i$ be independent and uniform on the $L$ outer endpoints.
    Then $o$ is the Fréchet mean of $Y_k$, $\Vfof{Y_k}=1/2$, and $\Eof{\ol{Y_k}{Y_k^i}^2}=2(1-1/L)$, which tends to $2=4\Vfof{Y_k}$ as $L\to\infty$.
    Thus the constant $4$ is optimal in this estimate.
\end{remark}

\subsection{Transport sensitivity}\label{ssec:2fmh:decl}

\begin{theorem}[Transport sensitivity]\label{thm:2fmh:decl}
    Let $\mo Y$ be $\mo b$-sensitive at exponent $1$ (\cref{def:dep:transport}) with marginals $\mu_{i}\in\mc P^2(\mc Q)$, and let $\mbar,\mhat$ be as in \cref{thm:2fmh:dependent}.
    Then
    \begin{equation}\label{eq:2fmh:decl:main-bound}
        \EOf{\ol{\mbar}{\mhat}^2}
        \;\leq\;
        \frac{1}{n^2}\sum_{i=1}^n \VmOf{Y_i}\sum_{k=1}^n b_{ik}
        \eqfs
    \end{equation}
\end{theorem}

\begin{proof}
    For each $i\in\nnset n$, let $\mo Y^i$ be the exchangeable copy of \cref{prp:dep:transport-copy} with $s=1$.
    Since $Y_i^i$ is an independent copy of $Y_i$, we have $\Eof{\ol{Y_i}{Y_i^i}^2}=2\Vmof{Y_i}$, and \eqref{eq:dep:mixed-moment} with $r=s\pr=1$ gives $\Eof{\ol{Y_i}{Y_i^i}\,\ol{Y_k}{Y_k^i}}\leq 2b_{ik}\Vmof{Y_i}$.
    Inserting this into \eqref{eq:2fmh:result_final} gives \eqref{eq:2fmh:decl:main-bound}.
\end{proof}

For stationary Gaussian $\ms{AR}(1)$ data with positive parameter, \eqref{eq:2fmh:decl:main-bound} with the coefficients of \cref{prp:dep:gaussian} is an equality (\cref{ssec:aux:ar1:gaussian}).

\begin{corollary}\label{cor:2fmh:decl:profile}
    In the setting of \cref{thm:2fmh:decl}, suppose $b_{ik}\leq b(\lvert i-k\rvert)$ for a function $b\colon\Nn\to\Rp$ with $b(0)=1$, and set $F_n:=1+2\sum_{h=1}^{n-1}b(h)$.
    Then
    \begin{equation}\label{eq:2fmh:decl:profile}
        \EOf{\ol{\mbar}{\mhat}^2}
        \;\leq\;
        \max_{i\in\nnset n}\VmOf{Y_i}\,\frac{F_n}{n}
        \eqfs
    \end{equation}
    In particular:
    \begin{enumerate}[label=(\roman*)]
        \item if $b$ is summable, then $F_n\leq1+2\sum_{h\geq1}b(h)$, which does not depend on $n$;
        \item if there exists $\rho\in[0,1)$ such that $b(h)=\rho^h$ for all $h\in\Nn$, then $F_n\leq\frac{1+\rho}{1-\rho}$;
        \item if there exist $\theta\in(0,1)$ and $c\geq1$ such that $b(h)\leq c\,(1+h)^{-\theta}$ for all $h\in\Nn$, then $F_n\leq\frac{2c}{1-\theta}\,n^{1-\theta}$;
        \item if there exists $c\geq1$ such that $b(h)\leq c\,(1+h)^{-1}$ for all $h\in\Nn$, then $F_n\leq1+2c\log n$.
    \end{enumerate}
\end{corollary}

\begin{proof}
    Bound $\sum_{k}b_{ik}\leq b(0)+2\sum_{h=1}^{n-1}b(h)=F_n$ in \eqref{eq:2fmh:decl:main-bound}.
    Item (i) is immediate and (ii) follows from $1+2\sum_{h\geq1}\rho^h = \frac{1+\rho}{1-\rho}$.
    For (iii), note first that $c \geq b(0) = 1$, so $\frac{2c}{1-\theta}\geq1$.
    With $\sum_{h=1}^{n-1}(1+h)^{-\theta}\leq\int_0^{n-1}(1+x)^{-\theta}\dl x=\frac{n^{1-\theta}-1}{1-\theta}$, we obtain $F_n\leq\frac{2c}{1-\theta}n^{1-\theta}+1-\frac{2c}{1-\theta}\leq\frac{2c}{1-\theta}n^{1-\theta}$.
    For (iv), $\sum_{h=1}^{n-1}(1+h)^{-1}\leq\int_0^{n-1}(1+x)^{-1}\dl x=\log n$.
\end{proof}

Items (i) and (ii) deal with summable sensitivities.
In this case, the risk is $\landauO(n^{-1})$, as for independent observations, as long as $\max_{i}\Vmof{Y_i}$ and $\sum_{h\geq1}b(h)$ remain bounded as $n$ grows.
A natural example is a reversible Markov chain whose kernel contracts $W_1$ by a factor $\rho\in[0,1)$, started from its stationary law in $\mc P^2(\mc Q)$.
By \cref{prp:dep:kernel} with $s=1$, item (ii) applies, and we obtain the bound \eqref{eq:2fmh:iid} for $n(1-\rho)/(1+\rho)$ independent draws.
In other words, $n(1-\rho)/(1+\rho)$ plays the role of an ``effective sample size'' for the Fréchet mean.
This is in the spirit of the Fréchet correlation effective sample size of Magee et al.~\cite{MageeEtAl2024Phylogenetic-BA}, which is built from Fréchet variances and squared distances for Markov chain Monte Carlo samples of phylogenetic trees.
We note in passing that \cref{prp:dep:kernel} assumes stationarity, so a chain started away from its stationary law requires a separate treatment.

Item (iii) gives a finite-sample bound for nonsummable dependence in an arbitrary Hadamard space.
The risk is then $\landauO(n^{-\theta})$ instead of $\landauO(n^{-1})$, as long as $\theta$ is fixed and $\max_{i}\Vmof{Y_i}$ and $c$ remain bounded.
This rate cannot be improved in general.
Indeed, consider stationary fractional Gaussian noise with Hurst index $H\in(1/2,1)$ and marginal variance $\sigma^2$.
By \cref{prp:dep:gaussian}, we may take $b(h)$ to be the correlation at lag $h$, which is nonnegative and decays like $h^{2H-2}$.
Hence item (iii) applies with $\theta=2-2H$.
Since the correlations are nonnegative, the bound of \cref{thm:2fmh:decl} equals the variance of the sample mean $\overline Y_n$, which is $\sigma^2n^{-\theta}$.

\begin{proposition}[Quadratic risk of the refresh chain]\label{prp:2fmh:refresh}
    Let $\mu\in\mc P^2(\mc Q)$, $\rho\in[0,1)$, and let $\mo Y$ be the refresh chain of \cref{prp:dep:refresh}.
    Then $\mbar=m_2(\mu)$ and
    \begin{equation}\label{eq:2fmh:refresh:bound}
        \EOf{\ol{\mbar}{\mhat}^2} \;\leq\; \frac{\VmOf{Y_1}}{n^2}\sum_{i,k=1}^n\rho^{\abs{i-k}} \;\leq\; \frac{1+\rho}{1-\rho}\cdot\frac{\VmOf{Y_1}}{n}
        \eqfs
    \end{equation}
    If $\mc Q$ is a Hilbert space, the first inequality in \eqref{eq:2fmh:refresh:bound} is an equality.
\end{proposition}

\begin{proof}
    By stationarity, $\mu_i=\mu$ for all $i$, so $\mubar=\mu$, $\mbar=m_2(\mu)$, and $\Vmof{Y_i}=\Vmof{Y_1}$.
    By \cref{prp:dep:refresh} with $s=1$, $\mo Y$ is $\mo b$-sensitive at exponent $1$ with $b_{ik}=\rho^{\lvert i-k\rvert}$, and \cref{thm:2fmh:decl} gives the first inequality in \eqref{eq:2fmh:refresh:bound}.
    The second follows from
    \begin{equation*}
        \sum_{i,k=1}^n\rho^{\abs{i-k}} = n+2\sum_{h=1}^{n-1}(n-h)\rho^h \leq n\,\frac{1+\rho}{1-\rho}
        \eqfs
    \end{equation*}
    For the equality statement, suppose that $\mc Q$ is a Hilbert space, so that $\mbar=\Eof{Y_1}$, $\mhat=n^{-1}\sum_{t=1}^nY_t$, and $\Vmof{Y_1}=\Eof{\normof{Y_1-\mbar}^2}$.
    Iterating \eqref{eq:dep:refresh:kernel} gives $\Eof{Y_k-\mbar\mid Y_1,\dots,Y_i}=\rho^{k-i}(Y_i-\mbar)$ for $i\leq k$, and hence $\Eof{\langle Y_i-\mbar, Y_k-\mbar\rangle}=\rho^{\lvert i-k\rvert}\Vmof{Y_1}$ for all $i,k\in\nnset n$.
    Summing over $i$ and $k$ gives $\Eof{\normof{\mhat-\mbar}^2}=n^{-2}\Vmof{Y_1}\sum_{i,k=1}^n\rho^{\lvert i-k\rvert}$, which is the first bound in \eqref{eq:2fmh:refresh:bound}.
\end{proof}

We define $a_n(\rho):=n^{-2}\sum_{i,k=1}^n\rho^{\lvert i-k\rvert}$.
The first bound in \eqref{eq:2fmh:refresh:bound} is the iid bound \eqref{eq:2fmh:iid} multiplied by $n\,a_n(\rho)$.
In the block representation of \cref{prp:dep:refresh}, this factor is the expected size of the block containing an index chosen uniformly from $\nnset n$, independently of the partition.
Indeed, indices $i$ and $k$ belong to the same block with probability $\rho^{\lvert i-k\rvert}$.
For fixed $\rho\in[0,1)$, this expected size is nondecreasing in $n$ and converges to $(1+\rho)/(1-\rho)$.

In terms of the metric covariance $\covQof{U,V}:=\frac14\Eof{\ol U{V\pr}^2+\ol{U\pr}V^2-2\ol UV^2}$ of \cite[(4)]{DubeyMueller2020TimeVarying-JRSSB}, where $(U\pr,V\pr)$ is an independent copy of $(U,V)$, the refresh chain has $\covQof{Y_i,Y_k}=\rho^{\lvert i-k\rvert}\Vmof{Y_1}$, so the first bound in \eqref{eq:2fmh:refresh:bound} equals $n^{-2}\sum_{i,k=1}^n\covQof{Y_i,Y_k}$.
Such pairwise sums do not determine the risk in general (\cref{rem:2fmh:onMainThm}\,(iii)).

\begin{example}\label{exa:2fmh:decl:refresh-tripod}
    Let $\mc Q$ be the metric tripod obtained by gluing three unit intervals at their common endpoint $o$, and denote their other endpoints by $x_1,x_2,x_3$.
    Equip $\mc Q$ with its path metric, so that $\ol o{x_1}=\ol o{x_2}=\ol o{x_3}=1$ and distinct outer endpoints are at distance $2$.
    As a compact $\R$-tree, $\mc Q$ is a separable Hadamard space \cite[Example~II.1.15(5)]{BridsonHaefliger1999NPC-Springer}, and it cannot be embedded isometrically into a Hilbert space.
    Take
    \[\mu=\tfrac12\delta_{x_1}+\tfrac14\delta_{x_2}+\tfrac14\delta_{x_3}\eqcm \]
    let $\rho\in[0,1)$, and let $\mo Y$ be the refresh chain of \cref{prp:dep:refresh}.
    Then $\mbar=m_2(\mu)=o$, $\Vfof{Y_1}=1$, $\Vmof{Y_1}=5/4$, and, for every $n\in\N$,
    \begin{equation*}
        \frac{1}{2} a_n(\rho) \leq \EOf{\ol{\mhat}o^2} \leq \frac54 a_n(\rho)
        \eqfs
    \end{equation*}
    The upper bound is \cref{prp:2fmh:refresh}. Thus \cref{prp:2fmh:refresh} captures the dependence of the risk on $n$ and $\rho$ up to a factor of at most $5/2$ in a space that is not a Hilbert space. In addition, $\Eof{\ol{\mhat}o^2}/a_n(\rho)\to1/2$ as $n\to\infty$ for fixed $\rho$, so this factor is attained asymptotically. These claims are proved in \cref{ssec:supp:2fmh:dep}.
\end{example}

\begin{remark}[An $\ms{AR}(1)$ comparison]\label{rem:2fmh:decl:limits}
    For the $\ms{AR}(1)$ process with Rademacher innovations and $0<\rho<1/2$ of \cref{rem:dep:ar1}, \cref{thm:2fmh:decl} gives no vanishing bound, since its admissible coefficients grow geometrically (\cref{rem:dep:ar1}).  The best bounds from \eqref{eq:2fmh:result_intermediate} and \eqref{eq:2fmh:result_final}, optimized over all exchangeable copies with $i$-independence, have orders $n^{-1/2}$ and $1$, respectively.
    With iid $N(0,1)$ innovations and $\rho\in(0,1)$, both bounds are exact. The exact error has order $n^{-1}$ in both models. The constructions and calculations are given in \cref{sec:aux:ar1}.
\end{remark}

\section{Quadratic Fréchet means in spaces with a positive curvature upper bound}\label{sec:pos}

In this section, $(\mc Q,d)$ is a separable and complete $\CAT(\kappa)$ space with $\kappa>0$, and $\alpha=2$ in the sampling notation \eqref{eq:pre:sampling}.
Restricting the data to a closed convex set of small circumradius retains uniqueness of the mean as well as the quadruple, variance, and contraction inequalities, up to explicit constants (\cref{ssec:pos:pre}).
With these, the coupling bound (\cref{ssec:pos:main}) and its consequences under dependency neighborhoods (\cref{ssec:pos:strucdep}) and under transport sensitivity (\cref{ssec:pos:decl}) carry over from \cref{sec:2fmh} with curvature-dependent factors.
Lipschitz images of Gaussian vectors are treated with the coupling bound, and a spherical three-point law observed along a refresh chain shows that the transport-sensitivity bound is sharp up to these factors.

\subsection{Preliminaries}\label{ssec:pos:pre}

We use the setting and notation of \cite[Section~4.1]{Schoetz2026Weighted-arxiv}.
Let $\kappa>0$ and let $(\mc Q, d)$ be a separable and complete $\CAT(\kappa)$ space.
Let $\setC \subseteq \mc Q$ be a nonempty, bounded, closed, and convex set with circumradius $\varrho := \inf_{o\in\mc Q}\sup_{y\in\setC}\ol oy$, and assume $\varrho<\pi/(2\sqrt\kappa)$.
Set
\begin{equation}\label{eq:pos:Delta}
    \Delta := 2\varrho\sqrt\kappa \in [0,\pi)
    \eqcm\quad
    C_\Delta := \frac{\Delta}{\sin(\Delta)} \in[1,\infty)
    \eqcm\quad
    C_0:=1
    \eqfs
\end{equation}
Let $\mc P(\setC)$ be the set of Borel probability measures on $\mc Q$ concentrated on $\setC$.
By \cite[Theorem~B]{Yokota2016Barycenter-JMSJ}, every $\mu \in \mc P(\setC)$ has a unique $2$-Fréchet mean $m_2(\mu)=\argmin_{q\in\mc Q}\int\ol yq^2\dl\mu(y)$, and $m_2(\mu)\in\setC$, so that the minimization may equivalently be restricted to $\setC$.
We use the following counterparts of \hyperref[sqr:QI]{(QI)}, \hyperref[sqr:VI]{(VI)}, and \hyperref[sqr:CP]{(CP)}.

\begin{enumerate}
    \item[(QI+)] \phantomsection\label{pos:QI}
    \textbf{Quadruple Inequality.}
    For $q,p,y,z\in \setC$, we have \cite[Theorem~1.4]{Schoetz2026Trapezoid-arxiv},
    \begin{equation*}
        \ol yq^2 - \ol yp^2 - \ol zq^2 + \ol zp^2 \leq 2C_\Delta\, \ol qp\, \ol yz
        \eqfs
    \end{equation*}
    \item[(VI+)] \phantomsection\label{pos:VI}
    \textbf{Variance Inequality.}
    Let $\mu \in \mc P(\setC)$ and $m:=m_2(\mu)$.
    Then for all $q\in \setC$, we have \cite[Theorem~A.2]{Schoetz2026Weighted-arxiv}
    \begin{equation*}
        \ol {m}q^2 \leq C_\Delta\int \brOf{\ol yq^2 - \ol y{m}^2} \dl \mu(y)
        \eqfs
    \end{equation*}
    \item[(CP+)]
    \phantomsection\label{pos:CP}
    \textbf{Contraction Property.}
    Let $\mu,\nu \in \mc P(\setC)$.
    Then we have \cite[Theorem~A.12]{Schoetz2026Weighted-arxiv}
    \begin{equation*}
        \ol {m_2(\mu)}{m_2(\nu)} \leq C_\Delta^2\, W_1(\mu, \nu)
        \eqfs
    \end{equation*}
\end{enumerate}

Moreover, for $\mc L(Y)\in\mc P(\setC)$, the variances are comparable \cite[(27) and Lemma~A.10]{Schoetz2026Weighted-arxiv}:
\begin{equation}\label{eq:pos:variances}
    \VfOf{Y}\leq\frac{2C_\Delta}{C_\Delta+1}\VmOf{Y}
    \quad\text{and}\quad
    \VmOf{Y}\leq2\VfOf{Y}
    \eqfs
\end{equation}
As $\Delta\to0$, we have $C_\Delta\to1$, and \hyperref[pos:QI]{(QI+)}, \hyperref[pos:VI]{(VI+)}, \hyperref[pos:CP]{(CP+)}, and \eqref{eq:pos:variances} turn into their Hadamard counterparts of \cref{ssec:2fmh:pre}.
Sharper constants are discussed in \cite[Remark~4.2]{Schoetz2026Weighted-arxiv}.

\subsection{The coupling bound}\label{ssec:pos:main}

\begin{theorem}[Coupling bound]\label{thm:pos:dependent}
    Let $\mo Y = (Y_1, \dots, Y_n)$ be an $n$-tuple of $\mc Q$-valued random variables with arbitrary joint distribution and marginals $\mu_{i} \in \mc P(\setC)$, and let $\mbar$ and $\mhat$ be as in \eqref{eq:pre:sampling} with $\alpha=2$.
    For $i \in \nnset n$, let $\mo Y^i$ be an exchangeable copy of $\mo Y$ with $i$-independence, and let $\mhat^i := m_2(\muhat^i)$.
    Then
    \begin{equation}\label{eq:pos:result_intermediate}
        \EOf{\ol{\mbar}{\mhat}^2}
        \leq
        \frac{C_\Delta^2}{2n}\sum_{i=1}^n \EOf{\ol{Y_i}{Y_i^i}\, \ol{\mhat}{\mhat^i}},
    \end{equation}
    and 
    \begin{equation}\label{eq:pos:result_final}
        \EOf{\ol{\mbar}{\mhat}^2}
        \leq
        \frac{C_\Delta^4}{2n^2}\, \sum_{i,k=1}^n \EOf{\ol{Y_i}{Y_i^i}\, \ol{Y_k}{Y_k^i}}
        \eqfs
    \end{equation}
\end{theorem}

\begin{proof}
    Apply \cref{lmm:supp:loo}\,(b), so that $C_{\mathrm{QI}}=2C_\Delta$.
    Almost surely, $\muhat\in\mc P(\setC)$, so $\mhat\in\setC$, and also $\mbar\in\setC$.
    By \hyperref[pos:VI]{(VI+)} for $\mubar$ at $q=\mhat$ and for $\muhat$ at $q=\mbar$, each of the two terms on the left-hand side of \eqref{eq:supp:loo} is at least $C_\Delta^{-1}\Eof{\ol{\mbar}{\mhat}^2}$.
    Hence
    \begin{equation*}
        \frac{2}{C_\Delta}\,\EOf{\ol{\mbar}{\mhat}^2}
        \leq
        \frac{C_\Delta}{n} \sum_{i=1}^n \EOf{\ol{Y_i}{Y_i^i} \, \ol{\mhat}{\mhat^i}}
        \eqcm
    \end{equation*}
    which is \eqref{eq:pos:result_intermediate}.
    By \hyperref[pos:CP]{(CP+)} and the index-wise coupling of the two empirical measures, $\ol{\mhat}{\mhat^i} \leq C_\Delta^2\,W_1(\muhat, \muhat^i) \leq C_\Delta^2 n^{-1} \sum_{k=1}^n \ol{Y_k}{Y_k^i}$, which gives \eqref{eq:pos:result_final}.
\end{proof}

As in \cref{thm:2fmh:dependent}, we call \eqref{eq:pos:result_intermediate} the intermediate bound and \eqref{eq:pos:result_final} the coordinatewise bound.
The proof is that of \cref{thm:2fmh:dependent}, with the inequalities of \cref{ssec:pos:pre} in place of those of \cref{ssec:2fmh:pre}; the factor $C_\Delta^4$ is the product of the curvature factors $C_\Delta$, $C_\Delta$, and $C_\Delta^2$ by which \hyperref[pos:VI]{(VI+)}, \hyperref[pos:QI]{(QI+)}, and \hyperref[pos:CP]{(CP+)} differ from their Hadamard counterparts.

\begin{remark}\label{rem:pos:onMainThm}
    As in \cref{cor:2fmh:dependent}, the terms (i)--(iv) there are upper bounds for $\Eof{\ol{\mbar}{\mhat}^2}$ if (i) is multiplied by $C_\Delta^2$ and (ii)--(iv) by $C_\Delta^4$.
    For independent observations, the coordinate-replacement copies \eqref{eq:dep:replacement} give $\Eof{\ol{\mbar}{\mhat}^2}\leq C_\Delta^4 n^{-2}\sum_{i=1}^n \Vmof{Y_i}$, which is \cite[Theorem~4.3]{Schoetz2026Weighted-arxiv} with equal weights.
\end{remark}

The coupling bound also yields risk bounds for Lipschitz images of Gaussian vectors, without requiring transport sensitivity of the transformed observations. Indeed, the bi-Lipschitz transfer result of \cref{prp:dep:link} cannot be used for nonconstant Gaussian inputs here: their unbounded support cannot be mapped into the bounded set $\setC$ by a map with a positive lower Lipschitz bound. Instead, we apply the map coordinatewise to the Gaussian copies of \cref{prp:2fmh:gaussian-1d} and use \cref{thm:pos:dependent}.

\begin{proposition}[Lipschitz images of Gaussian vectors]\label{prp:pos:gaussian-link}
    Let $\mo Z=(Z_1,\dots,Z_n)$ be jointly Gaussian in $\R^n$ with $\Vof{Z_i}>0$ for all $i\in\nnset n$.
    Let $g\colon\R\to\mc Q$ take values in $\setC$ and satisfy $\ol{g(z)}{g(z\pr)}\leq\Lambda\lvert z-z\pr\rvert$ for all $z,z\pr\in\R$ and some $\Lambda\in\Rpp$, and let $Y_i:=g(Z_i)$ for $i\in\nnset n$.
    Then, with $\mbar$ and $\mhat$ as in \cref{thm:pos:dependent},
    \begin{equation}\label{eq:pos:gaussian-link}
        \EOf{\ol{\mbar}{\mhat}^2}
        \leq
        \frac{C_\Delta^4\,\Lambda^2}{n^2}\sum_{i,k=1}^n\abs{\CovOf{Z_i,Z_k}}
        \eqfs
    \end{equation}
\end{proposition}
\begin{proof}
    For $i\in\nnset n$, let $(\mo Z,\mo Z^i)$ be the Gaussian exchangeable pair with $i$-independence of \cref{prp:2fmh:gaussian-1d}.
    By \eqref{eq:2fmh:illu:regression}, $Z_k-Z_k^i=\beta_{ik}(Z_i-Z_i^i)$ with $\beta_{ik}=\Covof{Z_i,Z_k}/\Vof{Z_i}$, and $\Eof{(Z_i-Z_i^i)^2}=2\Vof{Z_i}$.
    Let $Y_k^i:=g(Z_k^i)$ for $k\in\nnset n$.
    Applying the continuous map $g$ coordinatewise preserves exchangeability and $i$-independence, so $\mo Y^i$ is an exchangeable copy of $\mo Y$ with $i$-independence, and all marginals lie in $\mc P(\setC)$.
    By the Lipschitz property,
    \begin{equation*}
        \EOf{\ol{Y_i}{Y_i^i}\,\ol{Y_k}{Y_k^i}}
        \leq\Lambda^2\abs{\beta_{ik}}\,\EOf{(Z_i-Z_i^i)^2}
        =2\Lambda^2\abs{\CovOf{Z_i,Z_k}}
        \eqcm
    \end{equation*}
    and inserting this into \eqref{eq:pos:result_final} gives \eqref{eq:pos:gaussian-link}.
\end{proof}

For instance, if $\mo Z$ is a stationary Gaussian $\ms{AR}(1)$ sequence with marginal variance $\sigma^2$ and parameter $\rho\in(-1,1)$, the right-hand side of \eqref{eq:pos:gaussian-link} equals $C_\Delta^4\Lambda^2\sigma^2a_n(\lvert \rho\rvert)$, with $a_n$ as after \cref{prp:2fmh:refresh}.
This holds even if $g$ winds $\R$ around a circle in $\setC$, in which case $\mo Y$ is in general not a Markov chain even though $\mo Z$ is.
With $C_\Delta=1$ and \cref{thm:2fmh:dependent} in place of \cref{thm:pos:dependent}, the same bound holds for Lipschitz maps into a Hadamard space.

\subsection{Dependency neighborhoods}\label{ssec:pos:strucdep}

Recall the dependency neighborhoods $N_i$ of \cref{def:dep:neighborhood} and their average size $\bar N=n^{-1}\sum_{i=1}^n\#N_i$ from \eqref{eq:dep:L}.

\begin{theorem}[Local dependence]\label{thm:pos:neighborhood}
    Let $\mo Y = (Y_1, \dots, Y_n)$ be an $n$-tuple of $\mc Q$-valued random variables with marginals in $\mc P(\setC)$ and dependency neighborhoods $(N_i)_{i\in\nnset n}$, and let $\mbar,\mhat$ be as in \cref{thm:pos:dependent}.
    Then
    \begin{equation*}
    \begin{aligned}
        \EOf{\ol{\mbar}{\mhat}^2}
        &\leq
        \frac{C_\Delta^4}{n^2}\sum_{i=1}^n \VmOf{Y_i}
        +
        \frac{\sqrt{2}\,C_\Delta^4}{n^2}\sum_{i=1}^n \sqrt{\VmOf{Y_i}} \sum_{k \in N_i\setminus\{i\}} \sqrt{\VfOf{Y_k}}
        \\&\leq
        2\sqrt{\frac{C_\Delta}{C_\Delta+1}}\;C_\Delta^4\max_{i\in\nnset n}\VmOf{Y_i}\,\frac{\bar N}n
        \eqfs
    \end{aligned}
    \end{equation*}
\end{theorem}

\begin{proof}
    Choose the copies of \cref{prp:dep:neighborhood-copy}. The estimates of the proof of \cref{thm:2fmh:neighborhood}, inserted into \eqref{eq:pos:result_final} together with the Cauchy--Schwarz inequality, give the first bound. For the second, let $c:=2\sqrt{C_\Delta/(C_\Delta+1)}\geq1$; then \eqref{eq:pos:variances} gives $\sqrt2\,\Vfof{Y_k}^{1/2}\leq c\,(\max_{i}\Vmof{Y_i})^{1/2}$, and $1+c\,(\#N_i-1)\leq c\,\#N_i$.
\end{proof}

As $\Delta\to0$, the constant
$2\sqrt{C_\Delta/(C_\Delta+1)}$ converges to the constant
$\sqrt2$ of \cref{thm:2fmh:neighborhood}.

\subsection{Transport sensitivity}\label{ssec:pos:decl}

The notion of transport sensitivity (\cref{def:dep:transport}) applies without change, and so do the verifications of \cref{ssec:dep:examples}, which are metric constructions that do not use the curvature of $\mc Q$: \cref{prp:dep:kernel} applies to reversible Markov chains whose kernel maps $\setC$ into $\mc P(\setC)$ and satisfies \eqref{eq:dep:kernel-contraction} on $\setC$, and \cref{prp:dep:link} applies whenever the latent process takes values in a closed set $\mc K\subseteq\mc H$ with $g(\mc K)\subseteq\setC$, which forces $\mc K$ to be bounded.

\begin{theorem}[Transport sensitivity]\label{thm:pos:decl}
    Let $\mo Y$ be $\mo b$-sensitive at exponent $1$ (\cref{def:dep:transport}) with marginals in $\mc P(\setC)$, and let $\mbar, \mhat$ be as in \cref{thm:pos:dependent}.
    Then
    \begin{equation}\label{eq:pos:decl:main-bound}
        \EOf{\ol{\mbar}{\mhat}^2}
        \;\leq\;
        \frac{C_\Delta^4}{n^2}\sum_{i=1}^n \VmOf{Y_i}\sum_{k=1}^n b_{ik}
        \eqfs
    \end{equation}
\end{theorem}
\begin{proof}
    Choose the copies of \cref{prp:dep:transport-copy} with $s=1$. Since $Y_i^i$ is an independent copy of $Y_i$, we have $\Eof{\ol{Y_i}{Y_i^i}^2}=2\Vmof{Y_i}$; now insert \eqref{eq:dep:mixed-moment} with $r=s\pr=1$ into \eqref{eq:pos:result_final}.
\end{proof}

\begin{corollary}\label{cor:pos:decl:profile}
    In the setting of \cref{thm:pos:decl}, suppose $b_{ik}\leq b(\lvert i-k\rvert)$ for a function $b\colon\Nn\to\Rp$ with $b(0)=1$.
    Then \eqref{eq:2fmh:decl:profile} holds with its right-hand side multiplied by $C_\Delta^4$, and $F_n$ as in \cref{cor:2fmh:decl:profile}.
\end{corollary}
\begin{proof}
    The proof of \cref{cor:2fmh:decl:profile} only manipulates $\sum_k b_{ik}$ and leaves the prefactor of \eqref{eq:pos:decl:main-bound} untouched.
\end{proof}

For the refresh chain of \cref{prp:dep:refresh} with $\mu\in\mc P(\setC)$, which has $b_{ik}=\rho^{\lvert i-k\rvert}$ at exponent $1$, \cref{thm:pos:decl} gives
\begin{equation}\label{eq:pos:refresh}
    \EOf{\ol{\mbar}{\mhat}^2}\leq \frac{1}{n^2}C_\Delta^4\,\VmOf{Y_1}\,\sum_{i,k=1}^n\rho^{\lvert i-k\rvert}.
\end{equation}

\begin{example}[Three points on a sphere]\label{exa:pos:sphere}
    Let $\mc Q$ be the $2$-sphere of radius $1/\sqrt\kappa$ with its intrinsic distance. Fix $o\in\mc Q$ and $r\in(0,\pi/(4\sqrt\kappa))$, and let $\setC$ be the closed ball of radius $r$ about $o$, so that $\varrho=r$ and $\Delta=2\sqrt\kappa\,r<\pi/2$ in the notation of \cref{ssec:pos:pre}.
    
    Consider three points $x_1,x_2,x_3\in\setC$ at distance $r$ from $o$ such that the geodesics from $o$ to these points make pairwise angles $2\pi/3$ at $o$. By the spherical law of cosines, the distance between any two of these points is $\ell_r:=2\kappa^{-1/2}\arcsin(2^{-1}\sqrt3\sin(\sqrt\kappa\,r))$. Let $\mu$ be the uniform distribution on $\{x_1,x_2,x_3\}$, and let $\mo Y$ be the refresh chain of \cref{prp:dep:refresh} with marginal law $\mu$ and parameter $\rho\in[0,1)$. Then $\mbar=o$, $\Vfof{Y_1}=r^2$, and $\Vmof{Y_1}=\ell_r^2/3<r^2$.
    In addition, for all $n\in\N$,
    \begin{equation}\label{eq:pos:sphere}
        r^2\,a_n(\rho)
        \leq \EOf{\ol{\mbar}{\mhat}^2}
        \leq C_\Delta^4\,\frac{\ell_r^2}{3}\,a_n(\rho)
        \leq C_\Delta^4\,r^2\,a_n(\rho)
        \eqfs
    \end{equation}
\end{example}

Thus \cref{thm:pos:decl} captures the dependence of the error on $n$ and $\rho$ up to the factor $C_\Delta^4$, uniformly in $n$ and $\rho$, in a space of constant positive curvature.
Moreover, as $\Vmof{Y_1}=\ell_r^2/3<r^2$, the lower bound in \eqref{eq:pos:sphere} exceeds $\Vmof{Y_1}\,a_n(\rho)$: the Hadamard-space bound of \cref{prp:2fmh:refresh} fails on the sphere, and a curvature factor of at least $3r^2/\ell_r^2>1$ is necessary. These claims are proved in \cref{sec:supp:pos}.

\section{Power Fréchet means in Hadamard spaces}\label{sec:power}

In this section, $(\mc Q,d)$ is a separable Hadamard space and $\alpha\in(1,2)$ in the sampling notation \eqref{eq:pre:sampling}.
We first collect the geometric tools and the standing notation (\cref{ssec:power:prelim}), then introduce the adapted loss and its conversion to the $L^\alpha$ risk (\cref{ssec:supp:power:loss}). We subsequently state the general coupling bound (\cref{ssec:power:general}) and its consequences under dependency neighborhoods (\cref{ssec:power:strucdep}) and transport sensitivity (\cref{ssec:power:decl}).

\subsection{Preliminaries}\label{ssec:power:prelim}

Let $(\mc Q,d)$ be a separable Hadamard space, and fix
\begin{equation}\label{eq:power:alpha-phi}
\alpha\in(1,2)\eqcm\quad \pwa:=\frac{2-\alpha}{\alpha-1}
\eqfs
\end{equation}
Thus $\pwa>0$; the excluded case $\alpha=2$, in which $\pwa=0$, is the subject of \cref{sec:2fmh}. For $\mu\in \mc P^{\alpha-1}(\mc Q)$, the $\alpha$-Fréchet mean $m_\alpha(\mu)$ of \eqref{eq:pre:mean} exists and is unique \cite[Proposition~5.2(ii) and Corollary~5.8]{Schoetz2025VarianceInequalities-EJP}.
For $s\in\Rpp$, we call
\begin{equation}\label{eq:power:sigma-def}
\sigma_s(\mu) := \int \ol y{m_\alpha(\mu)}^{s}\,\dl\mu(y) \in [0,\infty]
\end{equation}
the \emph{moment of order $s$} of $\mu$ about its $\alpha$-Fréchet mean. We use the following inequalities.
\begin{enumerate}
    \item[(QI$_\alpha$)] \phantomsection\label{power:QI}
    \textbf{Quadruple Inequality.}
    For $q,p,y,z\in \mc Q$, we have \cite[Theorem~3]{Schoetz2019ConvergenceRates-EJS}
    \begin{equation*}
        \ol yq^\alpha-\ol yp^\alpha-\ol zq^\alpha+\ol zp^\alpha
        \leq
        2^{2-\alpha}\alpha\,\ol qp\,\ol yz^{\alpha-1}
    \eqfs
    \end{equation*}

    \item[(VI$_\alpha$)] \phantomsection\label{power:VI}
    \textbf{Variance Inequality.}
    Let $\mu\in \mc P^{\alpha-1}(\mc Q)$ and $m:=m_\alpha(\mu)$.
    Then for all  $q\in \mc Q$, we have
    \begin{equation*}
        \frac{1}{2}\alpha(\alpha-1)\ol qm^2
        \bigl(\sigma_{\alpha-1}(\mu)+\ol qm^{\alpha-1}\bigr)^{-\pwa}
        \leq
        \int\brOf{\ol yq^\alpha-\ol ym^\alpha}\dl\mu(y)
    \eqfs
    \end{equation*}
    This is \cite[Supplement, Lemma~S7.8]{Schoetz2026Transformed-arxiv} and \cite[(107)]{Schoetz2026Weighted-arxiv}, both based on \cite[Theorem~5.4]{Schoetz2025VarianceInequalities-EJP}.

    \item[(CP$_\alpha$)] \phantomsection\label{power:CP}
    \textbf{Contraction Property.}
    Let $\mu,\nu\in \mc P^{\alpha-1}(\mc Q)$ and $\delta:=\ol{m_\alpha(\mu)}{m_\alpha(\nu)}$.
    Then
    \begin{equation*}
        \frac{\delta}{\bigl(\sigma_{\alpha-1}(\mu)+\delta^{\alpha-1}\bigr)^\pwa
        + \bigl(\sigma_{\alpha-1}(\nu)+\delta^{\alpha-1}\bigr)^\pwa}
        \leq \frac{2^{1-\alpha}}{\alpha-1}W_{\alpha-1}(\mu,\nu)
    \eqfs
    \end{equation*}
    The left-hand side is read as $0$ if $\delta=0$. This is \cite[(108)]{Schoetz2026Weighted-arxiv}, which generalizes \cite[Supplement, Lemma~S7.2]{Schoetz2026Transformed-arxiv}.
\end{enumerate}
In the limiting case $\alpha=2$, the inequalities \hyperref[power:QI]{(QI$_\alpha$)}, \hyperref[power:VI]{(VI$_\alpha$)}, and \hyperref[power:CP]{(CP$_\alpha$)} reduce to \hyperref[sqr:QI]{(QI)}, \hyperref[sqr:VI]{(VI)}, and \hyperref[sqr:CP]{(CP)} of \cref{ssec:2fmh:pre}.
For $\alpha<2$, the moduli of \hyperref[power:VI]{(VI$_\alpha$)} and \hyperref[power:CP]{(CP$_\alpha$)} depend on the distributions through $\sigma_{\alpha-1}$.
\hyperref[power:VI]{(VI$_\alpha$)} is quadratic in $\ol qm$ only below the scale $\sigma_{\alpha-1}(\mu)^{1/(\alpha-1)}$ and of order $\ol qm^\alpha$ above it, and \hyperref[power:CP]{(CP$_\alpha$)} bounds $\delta$ by $W_{\alpha-1}(\mu,\nu)$ only up to a factor that grows with $\sigma_{\alpha-1}(\mu)$, $\sigma_{\alpha-1}(\nu)$, and $\delta$. As in \cref{sec:2fmh,sec:pos}, the proofs combine these three inequalities by the compensated replace-one stability technique (\cref{lmm:supp:loo}).

\begin{notation}[Standing notation]\label{not:power:standing}
Let $\mo Y=(Y_1,\dots,Y_n)$ be an $n$-tuple of $\mc Q$-valued random variables with arbitrary joint distribution and marginals $\mu_i\in \mc P^\alpha(\mc Q)$, and let $\mubar$, $\mbar=m_\alpha(\mubar)$, $\muhat$, and $\mhat=m_\alpha(\muhat)$ be as in \eqref{eq:pre:sampling}. For $s\in(0,\alpha]$, define
\begin{equation}\label{eq:power:standing-moments}
\sigma_s:=\sigma_s(\mubar)=\frac{1}{n}\sum_{j=1}^n\EOf{\ol{Y_j}{\mbar}^{s}}\eqcm\,\,
\sigmax s:=\max_{j\in\nnset n}\EOf{\ol{Y_j}{\mbar}^s}\eqcm\,\,
\hat\sigma_{\alpha-1}:=\frac{1}{n}\sum_{j=1}^n\ol{Y_j}{\mbar}^{\alpha-1}
\eqfs
\end{equation}
We also define $V^{\max}_s:=\max_{j\in\nnset n}\Vm^{s}[Y_j]$ for the maximal metric $s$-variance. It satisfies $V^{\max}_s\leq2\sigmax s$. Unlike $\sigmax s$, it vanishes when all observations are deterministic, even if their values differ.
\end{notation}

When exchangeable copies $\mo Y^i$ have been chosen, we abbreviate
\begin{equation*}
\begin{gathered}
    \eta_{ik}:=\ol{Y_k}{Y_k^i}^{\alpha-1}, \,\,
    \eta_i:=\eta_{ii}, \,\,
    T_i:=W_{\alpha-1}(\muhat,\muhat^i), \,\,
    \delta_i:=\ol{\mhat}{\mhat^i}, \,\,
    \hat\sigma_{\alpha-1}^i :=\frac{1}{n}\sum_{j=1}^n\ol{Y_j^i}{\mbar}^{\alpha-1}.
\end{gathered}
\end{equation*}
We also use the exponents
\begin{equation}\label{eq:power:structured-exponents}
    \beta:=\frac{\alpha}{\alpha-1}, \quad
    \alpha^\ast:=\frac{\alpha}{2-\alpha}, \quad
    \beta^\ast:=\frac{\alpha}{2\alpha-2},
\end{equation}
for which $\beta\geq2$, $\beta^\ast\geq1$, and $(\alpha^\ast)^{-1}+(\beta^\ast)^{-1}=1$.

\subsection{The adapted loss and the conversion to the
\texorpdfstring{$L^\alpha$}{L-alpha} risk}
\label{ssec:supp:power:loss}

The variance inequality \hyperref[power:VI]{(VI$_\alpha$)} suggests the adapted loss:
\begin{equation}\label{eq:power:loss}
    \psi_\alpha(r)
    :=r^2\brOf{\sigma_{\alpha-1}+r^{\alpha-1}}^{-\pwa}
    \eqcm \,\, r>0
\end{equation}
with $\psi_\alpha(0):=0$.
Up to the factor $\frac12\alpha(\alpha-1)$, $\psi_\alpha(\ol{\mbar}{\mhat})$ is the lower bound in \hyperref[power:VI]{(VI$_\alpha$)} for $\mubar$ at $q=\mhat$. This loss is used in \cite[Theorem~3.3 and Remark~3.4]{Schoetz2026Transformed-arxiv}; see also \cite[(128)]{Schoetz2026Weighted-arxiv}.

For $\sigma_{\alpha-1}>0$, it satisfies
\begin{equation*}
    2^{-\pwa}
    \min\{\sigma_{\alpha-1}^{-\pwa}r^2,r^\alpha\}
    \leq \psi_\alpha(r)
    \leq
    \min\{\sigma_{\alpha-1}^{-\pwa}r^2,r^\alpha\}
    \eqcm
\end{equation*}
see \cite[Section~C.4.2]{Schoetz2026Weighted-arxiv}. Thus the loss is quadratic below the scale $\sigma_{\alpha-1}^{1/(\alpha-1)}$ and of order $r^\alpha$ above it. If $\sigma_{\alpha-1}=0$, then $\psi_\alpha(r)=r^\alpha$. In the limiting case $\alpha=2$, the weight equals $1$, and the loss reduces to $\psi_2(r)=r^2$, as in \cref{sec:2fmh}.

The following lemma converts a bound on the expected adapted loss into an $L^\alpha$ risk bound.

\begin{lemma} [Proof in \cref{proof_lmm:supp:power:conversion}]\label{lmm:supp:power:conversion}
   
    Let $v\in\Rpp$, let $X$ be a nonnegative random variable,
    and let $R\in\Rp$ satisfy
    $\Eof{X^2(v+X^{\alpha-1})^{-\pwa}}\leq R$.
    Then
    \begin{equation*}
        \EOf{X^\alpha}^{1/\alpha}
        \leq
        c_\alpha\br{
            v^{\pwa/2}R^{1/2}+R^{1/\alpha}
        }
        \eqfs
    \end{equation*}
\end{lemma}

If $\sigma_{\alpha-1}=0$, then $Y_j=\mbar$ almost surely for every $j$, hence $\mhat=\mbar$ almost surely, and all bounds of \cref{sec:power} hold trivially. We therefore assume $\sigma_{\alpha-1}>0$ below.

We apply \cref{lmm:supp:power:conversion} with $v=\sigma_{\alpha-1}$ and $X=\ol{\mbar}{\mhat}$, and define
\begin{equation*}
    \Phi(R)
    :=\sigma_{\alpha-1}^{\pwa/2}R^{1/2}+R^{1/\alpha}
    \eqcm \,\, R\geq0
    \eqfs
\end{equation*}
The map $\Phi$ is nondecreasing and subadditive on $\Rp$, and $\Phi(cR) \leq \max\{c^{1/2},c^{1/\alpha}\}\Phi(R)$, $c\geq0.$ Hence an adapted-loss bound
$\Eof{\psi_\alpha(\ol{\mbar}{\mhat})}
\leq c_\alpha(R_1+R_2)$ yields
\begin{equation*}
    \EOf{\ol{\mbar}{\mhat}^\alpha}^{1/\alpha}
    \leq
    c_\alpha\br{\Phi(R_1)+\Phi(R_2)}
    \eqfs
\end{equation*}
By Jensen's inequality, the moments of
\eqref{eq:power:standing-moments} satisfy
\begin{equation}\label{eq:supp:power:jensen}
    \sigma_{\alpha-1}
    \leq \sigma_\alpha^{\frac{\alpha-1}\alpha}
    \eqfs
\end{equation}
For the leading contribution in our bounds, the term $R_1^{1/\alpha}$ in $\Phi(R_1)$ can be absorbed into $\sigma_{\alpha-1}^{\pwa/2}R_1^{1/2}$ and a term with displacements of order $\alpha$, which avoids moments of order $\alpha$ about $\mbar$. This is the absorption step of \cite[proof of Theorem~5.1, Step~2]{Schoetz2026Weighted-arxiv}.

\subsection{The coupling bound}\label{ssec:power:general}

We now extend the coupling bound of \cref{thm:2fmh:dependent} to power Fréchet means. For each $i\in\nnset n$, let $\mo Y^i$ be an exchangeable copy of $\mo Y$ with $i$-independence, and let $\mhat^i:=m_\alpha(\muhat^i)$. We summarize the displacements of these copies by $D_n$ and $D_n^\ast$, the fluctuation of the empirical moment $\hat\sigma_{\alpha-1}$ by $\omega_n$, and recall the conversion map $\Phi$:
\begin{equation}\label{eq:power:dep:notation}
\begin{gathered}
    D_n
    :=\frac{1}{n^2}\sum_{i,k=1}^n
    \EOf{
        \ol{Y_i}{Y_i^i}^{\alpha-1}
        \,\ol{Y_k}{Y_k^i}^{\alpha-1}}, \quad
    D_n^\ast
    :=\frac{1}{n^2}\sum_{i,k=1}^n
    \EOf{
        \ol{Y_i}{Y_i^i}^{\alpha/2}
        \,\ol{Y_k}{Y_k^i}^{\alpha/2}
    }^{\frac{2\alpha-2}\alpha}
    \eqcm
    \\
    \omega_n
    :=\EOf{
        \abs{\hat\sigma_{\alpha-1}-\sigma_{\alpha-1}}
        ^{\frac\alpha{\alpha-1}}
    }^{\frac{\alpha-1}\alpha}, \quad  \Phi(R)
    =\sigma_{\alpha-1}^{\pwa/2}R^{1/2}+R^{1/\alpha}, \,\, R\geq0.
\end{gathered}
\end{equation}
The quantities $D_n$ and $D_n^\ast$ depend on the choice
of copies.

\begin{lemma}\label{lmm:supp:power:single}
    Let $R\in\Rp$ satisfy
    $\Eof{\psi_\alpha(\ol{\mbar}{\mhat})}\leq R$.
    Then
    \begin{equation*}
        \EOf{\ol{\mbar}{\mhat}^\alpha}^{1/\alpha}
        \leq
        c_\alpha
        \sigma_\alpha^{(2-\alpha)/(2\alpha)}
        R^{1/2}
        \eqfs
    \end{equation*}
\end{lemma}

\begin{proof}\label{proof_lmm:supp:power:single}
    If $R\leq\sigma_\alpha$, then
    \begin{equation*}
        R^{1/\alpha}
        =R^{1/2}R^{(2-\alpha)/(2\alpha)}
        \leq
        \sigma_\alpha^{(2-\alpha)/(2\alpha)}R^{1/2}
        \eqcm
    \end{equation*}
    and
    $\sigma_{\alpha-1}^{\pwa/2}
    \leq\sigma_\alpha^{(2-\alpha)/(2\alpha)}$
    by \eqref{eq:supp:power:jensen}.
    Thus the claim follows from
    \cref{lmm:supp:power:conversion} with
    $v=\sigma_{\alpha-1}$.

    If $R>\sigma_\alpha$, the integrability step of the proof of \cref{lmm:supp:loo} gives
    \begin{equation*}
        \EOf{\ol{\mhat}{\mbar}^\alpha}
        \leq c_\alpha\frac{1}{n}\sum_j \EOf{\ol{Y_j}{\mbar}^\alpha}
        \eqfs
    \end{equation*}
    Hence
    \begin{align*}
        \EOf{\ol{\mbar}{\mhat}^\alpha}^{1/\alpha} \leq c_\alpha\sigma_\alpha^{1/\alpha} =c_\alpha
        \sigma_\alpha^{(2-\alpha)/(2\alpha)}
        \sigma_\alpha^{1/2}  < c_\alpha \sigma_\alpha^{(2-\alpha)/(2\alpha)} R^{1/2}.
    \end{align*}
\end{proof}

\begin{proposition}[Proof in \cref{proof_prp:supp:power:dep}]\label{prp:supp:power:dep}
 In the setting of \cref{not:power:standing}, let  $\mo Y^i$, $i\in\nnset n$, be exchangeable copies of $\mo Y$ with $i$-independence, let $D_n$,  $D_n^\ast$, $\omega_n$, and $\Phi$ be as in \eqref{eq:power:dep:notation}. Set $R_n:=n^{-1}\sum_{i=1}^n\Eof{\hat\sigma_{\alpha-1}^\pwa\eta_iT_i}$. Then, 

    \begin{equation}\label{eq:supp:power:dep}
        \EOf{\psi_\alpha(\ol{\mbar}{\mhat})}
        \leq
        \frac{2^{2-\alpha}}{\alpha-1}\,\frac{1}{n}\sum_{i=1}^n\EOf{\eta_i\delta_i}
        \eqcm\quad
        \EOf{\psi_\alpha(\ol{\mbar}{\mhat})}
        \leq
        c_\alpha R_n
        \eqfs
    \end{equation}
    Furthermore,
    \begin{equation}\label{eq:supp:power:dep:final}
        R_n\leq c_\alpha\brOf{\sigma_{\alpha-1}^\pwa D_n+\omega_n^\pwa D_n^\ast}
        \eqcm\quad
        R_n\leq\sigma_\alpha^{(2-\alpha)/\alpha}D_n^\ast
        \eqfs
    \end{equation}
\end{proposition}

\begin{lemma}[Absorption, proof in \cref{proof:lmm:supp:power:absorption} ]\label{lmm:supp:power:absorption}
In the setting of \cref{prp:supp:power:dep}, let $x,y\in\Rp$ and $C>0$ satisfy $x\leq C\sigma_{\alpha-1}^{\pwa\alpha}y^{\alpha-1}$. 
Then
    \begin{equation}\label{eq:supp:power:absorption}
        \Phi(x) \leq 2\sigma_{\alpha-1}^{\pwa/2}x^{1/2} + C^{\frac1{\alpha(\alpha-1)}}y^{1/\alpha}
        \eqfs
    \end{equation}
    Moreover, $\sigma_{\alpha-1}^\pwa D_n\leq2^{2-\alpha}\sigma_{\alpha-1}^{\pwa\alpha}\widetilde D_n^{\alpha-1}$, where $\widetilde D_n := n^{-2}\sum_{i,k=1}^n \EOf{ \ol{Y_i}{Y_i^i}^{\alpha/2} \,\ol{Y_k}{Y_k^i}^{\alpha/2}}.$ 
\end{lemma}

\begin{lemma}[Proof in \cref{proof_lmm:supp:power:omega}]\label{lmm:supp:power:omega}
    In the setting of \cref{prp:supp:power:dep}, we have 
\begin{equation}\label{eq:power:dep:omega}
        \omega_n^{\frac\alpha{\alpha-1}}
        \leq
        c_\alpha\inf_{M\geq0}\br{
            M^{2-\alpha}D_n
            +\max_{j\in\nnset n}
            \Eof{ \ol{Y_j}{\mbar}^\alpha
                \indOf{\{\ol{Y_j}{\mbar}>M\}} } },
    \end{equation}
and $\omega_n^\beta\leq2V^{\max}_\alpha$.
\end{lemma}

\begin{theorem}[Coupling bound]\label{thm:power:dep}
    In the setting of \cref{prp:supp:power:dep}, let $\mhat^i:=m_\alpha(\muhat^i)$. Then
    \begin{equation}\label{eq:power:dep:intermediate}
        \EOf{\ol{\mbar}{\mhat}^\alpha}^{1/\alpha} \leq
        c_\alpha\,\Phi\br{\frac{1}{n}\sum_{i=1}^n \EOf{ \ol{Y_i}{Y_i^i}^{\alpha-1} \,\ol{\mhat}{\mhat^i}} }
    \end{equation}
    and
    \begin{equation}\label{eq:power:dep:final}
        \EOf{\ol{\mbar}{\mhat}^\alpha}^{1/\alpha}
        \leq c_\alpha\,\Phi\brOf{\sigma_{\alpha-1}^\pwa D_n} +c_\alpha\,\Phi\brOf{\omega_n^\pwa D_n^\ast}
        \eqfs
    \end{equation}
Furthermore, 
    \begin{equation}\label{eq:power:dep:simple}
        \EOf{\ol{\mbar}{\mhat}^\alpha}^{1/\alpha} \leq
        c_\alpha\,\sigma_\alpha^{(2-\alpha)/\alpha} \brOf{D_n^\ast}^{1/2}.
    \end{equation}
\end{theorem}

\begin{proof}
We apply \cref{prp:supp:power:dep}, with $R_n$ as defined there. The first bound in \eqref{eq:supp:power:dep} and
    \cref{lmm:supp:power:conversion} with
    $v=\sigma_{\alpha-1}$ give
    \eqref{eq:power:dep:intermediate}, since
    $\eta_i\delta_i
    =\ol{Y_i}{Y_i^i}^{\alpha-1}\ol{\mhat}{\mhat^i}$.

Next, the second bound in \eqref{eq:supp:power:dep} and the first bound in \eqref{eq:supp:power:dep:final} give
    \begin{equation*}
        \EOf{\psi_\alpha(\ol{\mbar}{\mhat})}
        \leq
        c_\alpha\br{
            \sigma_{\alpha-1}^\pwa D_n
            +\omega_n^\pwa D_n^\ast
        }
        \eqfs
    \end{equation*}
Applying the conversion in the form $\EOf{\ol{\mbar}{\mhat}^\alpha}^{1/\alpha} \leq c_\alpha\br{\Phi(R_1)+\Phi(R_2)}$ of \cref{ssec:supp:power:loss} gives \eqref{eq:power:dep:final}.

Finally, the second bounds in \eqref{eq:supp:power:dep} and \eqref{eq:supp:power:dep:final} give $\EOf{\psi_\alpha(\ol{\mbar}{\mhat})}\leq c_\alpha\sigma_\alpha^{(2-\alpha)/\alpha}D_n^\ast$. Then \cref{lmm:supp:power:single} implies
    \[ \EOf{\ol{\mbar}{\mhat}^\alpha}^{1/\alpha}\leq c_\alpha
        \sigma_\alpha^{(2-\alpha)/(2\alpha)}
        \br{\sigma_\alpha^{(2-\alpha)/\alpha}D_n^\ast
        }^{1/2}
        =c_\alpha\sigma_\alpha^{(2-\alpha)/\alpha} \brOf{D_n^\ast}^{1/2}.
    \]
\end{proof}

As in \cref{thm:2fmh:dependent}, we call \eqref{eq:power:dep:intermediate} the intermediate bound and \eqref{eq:power:dep:final} the coordinatewise bound. At $\alpha=2$, we have $\pwa=0$ and $D_n=D_n^\ast$, so these bounds have the same form, up to constants, as the corresponding quadratic bounds \eqref{eq:2fmh:result_intermediate} and \eqref{eq:2fmh:result_final}.

The coordinatewise bound distinguishes two moment scales. The term $\sigma_{\alpha-1}^\pwa D_n^{1/2}$ involves moments of order $\alpha-1$ about $\mbar$ and of order $2\alpha-2$ of the displacements. The higher-order displacement quantity $D_n^\ast$ appears in the second contribution, with the factor $\omega_n^\pwa$ inside $\Phi$. By \eqref{eq:power:dep:omega}, $\omega_n\to0$ whenever $D_n\to0$ and the $\alpha$-th moments about $\mbar$ are uniformly integrable, that is, the tail term in \eqref{eq:power:dep:omega} tends to zero as $M\to\infty$,
uniformly in $n$. We also have $\omega_n^{\alpha/(\alpha-1)}\leq2V_\alpha^{\max}$ by \cref{lmm:supp:power:omega}.

The proof first bounds the expected adapted loss in \cref{prp:supp:power:dep}. These bounds are then converted into $L^\alpha$ bounds by \cref{lmm:supp:power:conversion}; this is the role of $\Phi$. For independent observations, the coordinate-replacement copies of \eqref{eq:dep:replacement} make the off-diagonal terms of $D_n$ and $D_n^\ast$ vanish. 
We compare the resulting bound with known results for independent data in \cref{rem:power:scope}.

\subsection{Dependency neighborhoods}\label{ssec:power:strucdep}

Recall the dependency neighborhoods $N_i$ of \cref{def:dep:neighborhood}, their maximal size $N_{\max}=\max_{i\in\nnset n}\#N_i$, and their average size $\bar N=n^{-1}\sum_{i=1}^n\#N_i$ from \eqref{eq:dep:L}.
The bound also involves a partition $\mc B_1,\dots,\mc B_\chi$ of $\nnset n$ such that, for every $t\in\nnset\chi$, the family $(Y_j)_{j\in\mc B_t}$ is mutually independent.
Such a partition with $\chi\leq N_{\max}$ always exists (\cref{lmm:power:locdep:partition}), so the bound below also holds with $\chi$ replaced by $N_{\max}$.

\begin{lemma}[Proof in \cref{proof_lmm:power:locdep:Xbar}]\label{lmm:power:locdep:Xbar}
    Let $\mc B_1,\dots,\mc B_\chi$ be a partition of $\nnset n$ such that $(Y_j)_{j\in\mc B_t}$ is mutually independent for every $t\in\nnset\chi$, and let $X_j:=\ol{Y_j}{\mbar}^{\alpha-1}-\Eof{\ol{Y_j}{\mbar}^{\alpha-1}}$ and $\overline{X}_n:=\frac{1}{n}\sum_{j=1}^nX_j=\hat\sigma_{\alpha-1}-\sigma_{\alpha-1}$.
    Then
    \begin{equation}\label{eq:power:locdep:Xbar}
        \EOf{\abs{\overline X_n}^\beta} \leq c_\alpha \br{\frac{\chi}{n}}^{\beta/2}V^{\max}_\alpha
        \eqfs
    \end{equation}
\end{lemma}

\begin{proposition}[Adapted-loss bound under local dependence]\label{prp:supp:power:structured}
Assume the setting of \cref{not:power:standing}, let $(N_i)_{i\in\nnset n}$ be dependency neighborhoods, and let $\mc B_1,\dots,\mc B_\chi$ be a partition of $\nnset n$ into mutually independent families as above. Then
    \begin{equation}\label{eq:supp:power:structured}
    \begin{aligned}
        \EOf{\psi_\alpha(\ol{\mbar}{\mhat})} &\leq c_\alpha\frac {\bar N}n \left( \sigma_{\alpha-1}^\pwa V^{\max}_{2\alpha-2} + \br{\frac\chi n}^{\pwa/2}V^{\max}_\alpha \right)
        \eqfs
    \end{aligned}
    \end{equation}
\end{proposition}

\begin{proof}
    For each $i$, choose the exchangeable copy $\mo Y^i$ of \cref{prp:dep:neighborhood-copy}; as $\mo Y^i$ and $\mo Y$ agree outside $N_i$, $\eta_{ik}=0$ for $k\notin N_i$.
    Let $r\in\{2\alpha-2,\alpha\}$ and $q\in\mc Q$.
    The triangle inequality, \eqref{eq:aux:subadd}, and $Y_k^i\stackrel d=Y_k$ give $\Eof{\ol{Y_k}{Y_k^i}^r}\leq c_\alpha\Eof{\ol{Y_k}q^r}$, and integrating over $q$ with respect to $\mu_k$ yields $\Eof{\ol{Y_k}{Y_k^i}^r}\leq c_\alpha\Vm^{r}[Y_k]\leq c_\alpha V^{\max}_r$ for every $k$.
    By the Cauchy--Schwarz inequality, $\Eof{\eta_i\eta_{ik}}\leq c_\alpha V^{\max}_{2\alpha-2}$ and $\Eof{\ol{Y_i}{Y_i^i}^{\alpha/2}\ol{Y_k}{Y_k^i}^{\alpha/2}}\leq c_\alpha V^{\max}_\alpha$.
    As there are $n\bar N$ pairs $(i,k)$ with $k\in N_i$, it follows that $D_n\leq c_\alpha(\bar N/n)V^{\max}_{2\alpha-2}$, $D_n^\ast\leq c_\alpha(\bar N/n)(V^{\max}_\alpha)^{(2\alpha-2)/\alpha}$, and, with $\widetilde D_n$ as in \cref{lmm:supp:power:absorption}, $\widetilde D_n\leq c_\alpha(\bar N/n)V^{\max}_\alpha$.
    \cref{lmm:power:locdep:Xbar} gives $\omega_n^\pwa\leq c_\alpha(\chi/n)^{\pwa/2}(V^{\max}_\alpha)^{(2-\alpha)/\alpha}$, since $\pwa/\beta=(2-\alpha)/\alpha$.
    \Cref{prp:supp:power:dep} gives $\Eof{\psi_\alpha(\ol{\mbar}{\mhat})}\leq c_\alpha R_n\leq c_\alpha(\sigma_{\alpha-1}^\pwa D_n+\omega_n^\pwa D_n^\ast)$, and inserting these bounds proves \eqref{eq:supp:power:structured}.
\end{proof}

\begin{theorem}[Local dependence]\label{thm:power:structured}
Assume the setting of \cref{prp:supp:power:structured}. Then
\begin{eqnarray}\label{eq:power:structured}
    \EOf{\ol{\mbar}{\mhat}^\alpha}^{1/\alpha}
    \leq c_\alpha\br{\frac{\bar N}n}^{\frac{1}{2}}\br{\sigma_{\alpha-1}^{\pwa}\,\brOf{V^{\max}_{2\alpha-2}}^{\frac{1}{2}} + \sigma_{\alpha-1}^{\frac\pwa2}\,\brOf{V^{\max}_\alpha}^{\frac{1}{2}}\br{\frac\chi n}^{\frac\pwa4}} 
    + c_\alpha\br{\frac{\bar N\,V^{\max}_\alpha}{n}}^{\frac1\alpha}.
\end{eqnarray}
In addition,
\begin{equation}\label{eq:power:structured:simple}
    \EOf{\ol{\mbar}{\mhat}^\alpha}^{1/\alpha}
    \leq
    c_\alpha\,\sigma_\alpha^{(2-\alpha)/\alpha}\brOf{V^{\max}_\alpha}^{(\alpha-1)/\alpha}\br{\frac{\bar N}n}^{1/2}
    \eqfs
\end{equation}
\end{theorem}

\begin{proof}
For each $i\in\nnset n$, choose the exchangeable copy $\mo Y^i$ of \cref{prp:dep:neighborhood-copy}. We first prove \eqref{eq:power:structured}.
The coordinatewise bound \eqref{eq:power:dep:final} gives \[\Eof{\ol{\mbar}{\mhat}^\alpha}^{1/\alpha}\leq c_\alpha(\Phi(\sigma_{\alpha-1}^\pwa D_n)+\Phi(\omega_n^\pwa D_n^\ast)).\] 
Applying \cref{lmm:supp:power:absorption} to the first term and expanding $\Phi$ in the second, we obtain
\[
\EOf{\ol{\mbar}{\mhat}^\alpha}^{1/\alpha} \leq c_\alpha\br{\sigma_{\alpha-1}^\pwa D_n^{1/2} +\sigma_{\alpha-1}^{\pwa/2}
            \brOf{\omega_n^\pwa D_n^\ast}^{1/2}}+c_\alpha\br{\widetilde D_n^{1/\alpha}+\brOf{\omega_n^\pwa D_n^\ast}^{1/\alpha}}.\]

We now bound these four terms. The estimates from the proof of
\cref{prp:supp:power:structured} are $D_n\leq c_\alpha(\bar N/n)V^{\max}_{2\alpha-2}$, $\widetilde D_n\leq c_\alpha(\bar N/n)V^{\max}_\alpha$, and $\omega_n^\pwa D_n^\ast\leq c_\alpha(\bar N/n)(\chi/n)^{\pwa/2}V^{\max}_\alpha$.
For the two square-root terms, these give 
\[\sigma_{\alpha-1}^\pwa D_n^{1/2} + \sigma_{\alpha-1}^{\pwa/2}
\brOf{\omega_n^\pwa D_n^\ast}^{1/2}\leq c_\alpha\br{\frac{\bar N}{n}}^{1/2} \br{\sigma_{\alpha-1}^\pwa\brOf{V^{\max}_{2\alpha-2}}^{1/2}  +\sigma_{\alpha-1}^{\pwa/2}\brOf{V^{\max}_\alpha}^{1/2} \br{\frac\chi n}^{\pwa/4}}.\]
For the remaining terms, we use $\chi\leq n$ to obtain
\[\widetilde D_n^{1/\alpha}+ \brOf{\omega_n^\pwa D_n^\ast}^{1/\alpha}\leq c_\alpha\br{\frac{\bar N\,V^{\max}_\alpha}{n}}^{1/\alpha}.\] 
Substituting this into \eqref{eq:power:dep:simple} proves the second bound.
\end{proof}

The bound \eqref{eq:power:structured} consists of three terms. The first one is the leading term. It is of order $(\bar N/n)^{1/2}$ and only involves moments of order $\alpha-1$ and $2\alpha-2$. The $\alpha$-th moment only enters through the other two terms. In the second term, it comes with the additional factor $(\chi/n)^{\pwa/4}$. The third term is of order $(\bar N/n)^{1/\alpha}$, which is smaller because $\alpha<2$.

Two kinds of moments appear in \eqref{eq:power:structured} and \eqref{eq:power:structured:simple}. The quantities $\sigma_{\alpha-1}$ and $\sigma_\alpha$ measure distances to the common target $\mbar$.
All other moments are maximal metric variances, which measure \emph{fluctuations}. Since $\sigma_{\alpha-1}$ and $\sigma_\alpha$ only appear multiplied by metric variances, both bounds vanish for deterministic observations, even at different locations.

As an example, consider $\ell$-dependent data with $\ell\geq1$. By \cref{rem:dep:neighborhood-related}, the sets $N_i=\setByEleInText{j\in\nnset n}{\lvert i-j\rvert\leq\ell}$ are dependency neighborhoods, so that $\bar N\leq N_{\max}\leq\min\{n,2\ell+1\}$. For the partition, we group the indices by their residue modulo $\ell+1$. Two distinct indices in the same class are more than $\ell$ apart, so the observations in each class are mutually independent. This gives $\chi=\min\{n,\ell+1\}$ classes In particular, $\bar N\leq3\ell$ and $\chi\leq2\ell$. The right-hand sides of \eqref{eq:power:structured} and \eqref{eq:power:structured:simple} are nondecreasing in $\bar N$ and $\chi$.
Therefore, both bounds hold with $\bar N$ and $\chi$ replaced by $\ell$, after adjusting the constant $c_\alpha$.

\begin{remark}\label{rem:power:scope}
    For independent observations, the coordinate-replacement copies \eqref{eq:dep:replacement} move only the $i$-th coordinate, so the off-diagonal terms of $D_n$ and $D_n^\ast$ vanish.
    With the moment bound of \cref{lmm:power:locdep:Xbar} for $\omega_n$ and the absorption step of \cref{lmm:supp:power:absorption}, \eqref{eq:power:dep:final} then gives
    \begin{equation*}
        \EOf{\ol{\mbar}{\mhat}^\alpha}^{1/\alpha} \leq c_\alpha\,\sigma_{\alpha-1}^{\pwa}\br{\frac{V^{\max}_{2\alpha-2}}n}^{\frac{1}{2}} + c_\alpha\,\sigma_{\alpha-1}^{\frac\pwa2}\brOf{V^{\max}_\alpha}^{\frac{1}{2}}n^{-\frac{1}{2}-\frac\pwa4} + c_\alpha\br{\frac{V^{\max}_\alpha}{n}}^{\frac1\alpha}
        \eqcm
    \end{equation*}
    which is \cref{thm:power:structured} with $N_i=\{i\}$ and a single block, $\chi=1$.
For iid data, the leading term coincides, up to a constant, with that of \cite[Corollary~3.5]{Schoetz2026Transformed-arxiv}, while the lower-order terms are weaker, of orders $n^{-1/2-\pwa/4}$ and $n^{-1/\alpha}$ against $n^{-1/2-\pwa/2}$ and $n^{-1/(\alpha(\alpha-1))}$ there.  A similar comparison holds for independent observations with different laws in \cite[Corollary~5.2]{Schoetz2026Weighted-arxiv}. 
    
The weaker middle term comes from the random modulus $\hat\sigma_{\alpha-1}$, whose fluctuation we control by H\"older's inequality (\cref{prp:supp:power:dep}); for independent observations, \cite[proof of Proposition~C.3]{Schoetz2026Weighted-arxiv} separates the resampled observation from this modulus, so that the expectation factorizes, which is not possible when the copy moves all coordinates.

The weaker last term comes from the absorption step (\cref{lmm:supp:power:absorption}), which does not use the decay in $n$ of a displacement moment of lower order. For iid data, lower bounds that attain this leading term for suitable laws are given in \cite[Corollary~5.10 and Proposition~5.11]{Schoetz2026Transformed-arxiv}. We do not pursue lower bounds under dependence. Unlike in \cref{ssec:2fmh:illu}, there is no exact representation of the error to compare the bounds against for $\alpha<2$, as the $\alpha$-Fréchet mean is not linear even on the real line, and the constants $c_\alpha$ are not tracked; we do not claim that the dependence of the bounds in \cref{thm:power:structured,thm:power:decl} on the moments and on the strength of dependence is optimal.
    
\end{remark}

\subsection{Transport sensitivity}\label{ssec:power:decl}
We now apply the coupling bound under transport sensitivity.
For a sensitivity matrix $\mo b=(b_{ik})$, set $r_n:=n^{-1}\max_{i}\sum_{k}b_{ik}^{\alpha-1}$.
We use the notation $\omega_n$ for the fluctuation term defined in \eqref{eq:power:dep:notation}.

\begin{lemma}\label{lmm:power:decl:hatsigma}
    We have $\Eof{\hat\sigma_{\alpha-1}^{\,\beta}} \leq \sigma_\alpha\leq\sigmax\alpha$.
\end{lemma}

\begin{proof}
    As $\hat\sigma_{\alpha-1}$ is an average of nonnegative terms and $\beta\geq1$, Jensen's inequality gives $\hat\sigma_{\alpha-1}^{\,\beta} \leq \frac{1}{n}\sum_{j=1}^n\ol{Y_j}{\mbar}^{(\alpha-1)\beta} = \frac{1}{n}\sum_{j=1}^n\ol{Y_j}{\mbar}^{\alpha}$.
    Taking expectations proves the claim.
\end{proof}

\begin{proposition}[Adapted-loss bound under transport sensitivity]\label{prp:supp:power:decl}
   Assume the setting of \cref{not:power:standing}, and let $\mo Y$ be $\mo b$-sensitive at exponent $\alpha/2$ (\cref{def:dep:transport}).
    Then
    \begin{equation}\label{eq:supp:power:decl}
    \begin{aligned}
        \EOf{\psi_\alpha(\ol{\mbar}{\mhat})}
        \leq
        \frac{c_\alpha\sigma_\alpha^{(2-\alpha)/\alpha}}{n^2}\sum_{i,k=1}^nb_{ik}^{\alpha-1}\brOf{\VmsOf{\alpha}{Y_i}}^{(2\alpha-2)/\alpha}
 \leq
        c_\alpha\sigma_\alpha^{(2-\alpha)/\alpha}\brOf{V^{\max}_\alpha}^{(2\alpha-2)/\alpha}r_n
        \eqfs
    \end{aligned}
    \end{equation}
\end{proposition}

\begin{proof}
    For each $i$, choose the exchangeable copy $\mo Y^i$ of \cref{prp:dep:transport-copy} with $s=\alpha/2$.
    As $Y_i^i$ is an independent copy of $Y_i$, \eqref{eq:dep:mixed-moment} with $r=s\pr=\alpha/2$ gives
    \begin{equation}\label{eq:supp:power:decl:mixed}
        \EOf{\ol{Y_i}{Y_i^i}^{\alpha/2}\,\ol{Y_k}{Y_k^i}^{\alpha/2}}
        \leq b_{ik}^{\alpha/2}\,\EOf{\ol{Y_i}{Y_i^i}^{\alpha}}
        = 2b_{ik}^{\alpha/2}\,\VmsOf{\alpha}{Y_i}
        \eqfs
    \end{equation}
    H\"older's inequality with the exponents \eqref{eq:power:structured-exponents}, \cref{lmm:power:decl:hatsigma}, the index-wise coupling $T_i\leq n^{-1}\sum_{k=1}^n\eta_{ik}$, and Minkowski's inequality in $L^{\beta^\ast}$ give
    \begin{equation*}
        \EOf{\hat\sigma_{\alpha-1}^\pwa\,\eta_iT_i}
        \leq
        \EOf{\hat\sigma_{\alpha-1}^{\,\beta}}^{1/\alpha^\ast}\,\normOf{\eta_iT_i}_{\beta^\ast}
        \leq
        \sigma_\alpha^{(2-\alpha)/\alpha}\,\frac{1}{n}\sum_{k=1}^n\normOf{\eta_i\eta_{ik}}_{\beta^\ast}
        \eqfs
    \end{equation*}
    Since $(\eta_i\eta_{ik})^{\beta^\ast}=\ol{Y_i}{Y_i^i}^{\alpha/2}\,\ol{Y_k}{Y_k^i}^{\alpha/2}$ and $\frac{\alpha/2}{\beta^\ast}=\alpha-1$, \eqref{eq:supp:power:decl:mixed} gives
    \begin{equation*}
        \normOf{\eta_i\eta_{ik}}_{\beta^\ast}\leq c_\alpha\, b_{ik}^{\alpha-1}\br{\Vm^{\alpha}[Y_i]}^{(2\alpha-2)/\alpha}
        \eqfs
    \end{equation*}
    Hence
    \begin{equation*}
        \EOf{\hat\sigma_{\alpha-1}^\pwa\,\eta_iT_i}\leq c_\alpha\sigma_\alpha^{(2-\alpha)/\alpha}\brOf{\VmsOf{\alpha}{Y_i}}^{(2\alpha-2)/\alpha}\frac{1}{n}\sum_{k=1}^nb_{ik}^{\alpha-1}
        \eqcm
    \end{equation*}
    and averaging over $i$ and applying the second bound in \eqref{eq:supp:power:dep} proves the first bound in \eqref{eq:supp:power:decl}.
    The second follows from $\Vm^{\alpha}[Y_i]\leq V^{\max}_\alpha$ and $n^{-2}\sum_{i,k=1}^nb_{ik}^{\alpha-1}\leq r_n$.
\end{proof}
\begin{theorem}[Transport sensitivity]\label{thm:power:decl}
   Assume the setting of \cref{not:power:standing}, and let $\mo Y$ be $\mo b$-sensitive at exponent $\alpha/2$ (\cref{def:dep:transport}).
    Then
    \begin{equation}\label{eq:power:decl:simple}
        \EOf{\ol{\mbar}{\mhat}^\alpha}^{1/\alpha}
        \leq
        c_\alpha\,\sigma_\alpha^{(2-\alpha)/\alpha}\brOf{V^{\max}_\alpha}^{(\alpha-1)/\alpha}r_n^{1/2}
        \eqfs
    \end{equation}
    If, in addition, $r_n\leq1$, then
    \begin{equation}\label{eq:power:decl:main}
        \EOf{\ol{\mbar}{\mhat}^\alpha}^{1/\alpha}
        \leq
        c_\alpha\,r_n^{\frac{1}{2}}\br{\sigma_{\alpha-1}^{\pwa}\,\brOf{V^{\max}_{2\alpha-2}}^{\frac{1}{2}}+\sigma_{\alpha-1}^{\frac\pwa2}\,\omega_n^{\frac\pwa2}\brOf{V^{\max}_\alpha}^{\frac{\alpha-1}\alpha}}+c_\alpha\brOf{r_nV^{\max}_\alpha}^{\frac1\alpha}
        \eqcm
    \end{equation}
    where the fluctuation term $\omega_n$ satisfies
    \begin{equation}\label{eq:power:decl:omega}
        \omega_n^{\frac\alpha{\alpha-1}}
        \leq
        c_\alpha\inf_{M\geq0}\br{M^{2-\alpha}\,r_nV^{\max}_{2\alpha-2}+\max_{j\in\nnset n}\Eof{\ol{Y_j}{\mbar}^\alpha\indOf{\{\ol{Y_j}{\mbar}>M\}}}}
        \eqfs
    \end{equation}
\end{theorem}

\begin{proof}
For each $i\in\nnset n$, choose the exchangeable copy $\mo Y^i$ of \cref{prp:dep:transport-copy} with $s=\alpha/2$. Since $Y_i^i$ is an independent copy of $Y_i$, \eqref{eq:dep:mixed-moment}, applied with
$r=s\pr=\alpha-1$ and with $r=s\pr=\alpha/2$, gives
    \[
        \EOf{\eta_i\eta_{ik}}\leq2b_{ik}^{\alpha-1}\VmsOf{2\alpha-2}{Y_i}
        \eqcm\quad
        \EOf{\ol{Y_i}{Y_i^i}^{\alpha/2}\ol{Y_k}{Y_k^i}^{\alpha/2}}\leq2b_{ik}^{\alpha/2}\VmsOf{\alpha}{Y_i}.
    \]
Then $D_n\leq2r_nV^{\max}_{2\alpha-2}$ and $D_n^\ast\leq c_\alpha r_n(V^{\max}_\alpha)^{(2\alpha-2)/\alpha}$. Substituting the second estimate into \eqref{eq:power:dep:simple} proves \eqref{eq:power:decl:simple}. The first estimate, combined with \eqref{eq:power:dep:omega}, gives
\eqref{eq:power:decl:omega}. It remains to prove \eqref{eq:power:decl:main}. Assume now that $r_n\leq1$. The coordinatewise bound \eqref{eq:power:dep:final} gives
\[
    \EOf{\ol{\mbar}{\mhat}^\alpha}^{1/\alpha}
    \leq
    c_\alpha\Phi\brOf{\sigma_{\alpha-1}^\pwa D_n}
    +
    c_\alpha\Phi\brOf{\omega_n^\pwa D_n^\ast}
    \eqfs
\]
We bound these two contributions separately. For the first, we need an interpolation estimate for $D_n$. Write $s_i:=\Eof{\ol{Y_i}{\mbar}^{\alpha-1}}$, so that $n^{-1}\sum_i s_i=\sigma_{\alpha-1}$.
H\"older's inequality and the triangle inequality give $\Vm^{2\alpha-2}[Y_i] \leq\brOf{\Vm^{\alpha-1}[Y_i]}^{2-\alpha}
        \brOf{\Vm^\alpha[Y_i]}^{\alpha-1} \leq s_i^{2-\alpha} \brOf{\Vm^\alpha[Y_i]}^{\alpha-1}.$
Using the definition of $r_n$ and applying H\"older inequality to the sum over $i$, we obtain
\[D_n \leq \frac2{n^2}\sum_{i,k=1}^n b_{ik}^{\alpha-1}\Vm^{2\alpha-2}[Y_i] \leq \frac{2r_n}{n}\sum_{i=1}^ns_i^{2-\alpha} \brOf{\Vm^\alpha[Y_i]}^{\alpha-1} 
\leq2r_n\sigma_{\alpha-1}^{2-\alpha} \brOf{V^{\max}_\alpha}^{\alpha-1}.\]
        
Since $\pwa+2-\alpha=\pwa\alpha$ and $(1+\pwa)(\alpha-1)=1$, this implies
$ \sigma_{\alpha-1}^\pwa D_n \leq 2\sigma_{\alpha-1}^{\pwa\alpha}
    \brOf{r_n^{1+\pwa}V^{\max}_\alpha}^{\alpha-1} \eqfs$ We then apply \eqref{eq:supp:power:absorption} with $x=\sigma_{\alpha-1}^\pwa D_n$, $y=r_n^{1+\pwa}V^{\max}_\alpha$, and $C=2$. This gives
 $\Phi(\sigma_{\alpha-1}^\pwa D_n)\leq2\sigma_{\alpha-1}^\pwa D_n^{1/2}+c_\alpha(r_nV^{\max}_\alpha)^{1/\alpha}$, since $r_n\leq1$. For the second contribution, \cref{lmm:supp:power:omega} gives $(\omega_n^\pwa D_n^\ast)^{1/\alpha}\leq c_\alpha(r_nV^{\max}_\alpha)^{1/\alpha}$. Together with the bound on $D_n^\ast$, this yields $\omega_n^\pwa D_n^\ast \leq c_\alpha r_nV^{\max}_\alpha$. Expanding $\Phi$, we conclude that
$\Phi\brOf{\omega_n^\pwa D_n^\ast}
    \leq \sigma_{\alpha-1}^{\pwa/2} \brOf{\omega_n^\pwa D_n^\ast}^{1/2}
 + c_\alpha\brOf{r_nV^{\max}_\alpha}^{1/\alpha}$ Combining the two contributions gives
\[\EOf{\ol{\mbar}{\mhat}^\alpha}^{1/\alpha} \leq c_\alpha\br{ \sigma_{\alpha-1}^\pwa D_n^{1/2}  + \sigma_{\alpha-1}^{\pwa/2}
\omega_n^{\pwa/2}(D_n^\ast)^{1/2}} + c_\alpha\brOf{r_nV^{\max}_\alpha}^{1/\alpha}.
\]
Finally, substituting the bounds on $D_n$ and $D_n^\ast$
from the beginning of the proof gives
\eqref{eq:power:decl:main}.
\end{proof}

The adapted-loss version of \eqref{eq:power:decl:simple} is given in \cref{prp:supp:power:decl}.
The first contribution in \eqref{eq:power:decl:main} involves only $\sigma_{\alpha-1}$ and $V^{\max}_{2\alpha-2}$.
The higher-order variance $V^{\max}_\alpha$ appears in the other two terms, with the additional factor $\omega_n^{\pwa/2}$ or the power $r_n^{1/\alpha}$.
The estimate \eqref{eq:power:decl:omega} shows that
 $\omega_n\to0$ whenever $r_n\to0$ and the variables $\ol{Y_j}{\mbar}^\alpha$ are uniformly integrable over $n$ and $j\in\nnset n$.

\begin{corollary}\label{cor:power:decl:profile}
    In the setting of \cref{thm:power:decl}, suppose $b_{ik}\leq b(\lvert i-k\rvert)$ for a function $b\colon\Nn\to\Rp$ with $b(0)=1$, and set $F_n:=1+2\sum_{h=1}^{n-1} b(h)^{\alpha-1}$.
    Then
    \begin{equation*}
        \EOf{\ol{\mbar}{\mhat}^\alpha}^{1/\alpha}
        \;\leq\;
        c_\alpha\,\sigma_\alpha^{(2-\alpha)/\alpha}\brOf{V^{\max}_\alpha}^{(\alpha-1)/\alpha}\br{\frac{F_n}n}^{1/2}
        \eqfs
    \end{equation*}
    Furthermore, the following bounds on $F_n$ hold:
    \begin{enumerate}[label=(\roman*)]
        \item if $\sum_{h\geq1}b(h)^{\alpha-1}<\infty$, then $F_n\leq1+2\sum_{h\geq1}b(h)^{\alpha-1}$, which does not depend on $n$;
        \item if there exist $\theta>0$ and $c\geq1$ such that $b(h)\leq c\,(1+h)^{-\theta}$ for all $h\in\Nn$, then, with $\gamma:=\theta(\alpha-1)$,
        \begin{equation*}
            F_n\leq
            \begin{cases}
                1+\frac{2c^{\alpha-1}}{\gamma-1}, & \gamma>1,\\
                1+2c^{\alpha-1}\log n, & \gamma=1,\\
                \frac{2c^{\alpha-1}}{1-\gamma}\,n^{1-\gamma}, & \gamma\in(0,1).
            \end{cases}
        \end{equation*}
    \end{enumerate}
\end{corollary}

\begin{proof}
For each $i$, there are at most two indices at any given positive distance from $i$. Then  $\sum_k b_{ik}^{\alpha-1}\leq b(0)^{\alpha-1}+2\sum_{h=1}^{n-1}b(h)^{\alpha-1}=F_n.$ Thus $r_n\leq F_n/n$, and the risk bound follows from
\eqref{eq:power:decl:simple}. Item (i) follows by extending the nonnegative sum to all $h\geq1$.

For (ii), the assumption gives $b(h)^{\alpha-1}\leq c^{\alpha-1}(1+h)^{-\gamma}$, then
$\sum_{h=1}^{n-1}(1+h)^{-\gamma}\leq\int_0^{n-1}(1+x)^{-\gamma}\dl x$. If $\gamma>1$, the integral is at most $1/(\gamma-1)$; if $\gamma=1$, it equals $\log n$. These give the first two cases.
For $0<\gamma<1$, evaluating the integral gives $F_n \leq 1+\frac{2c^{\alpha-1}}{1-\gamma}  \br{n^{1-\gamma}-1} \leq
    \frac{2c^{\alpha-1}}{1-\gamma}\,n^{1-\gamma},$ where the last inequality uses $2c^{\alpha-1}/(1-\gamma)\geq1$. This proves the remaining case.
\end{proof}

The resulting rates are $\landauO(n^{-1/2})$ in case (i) and in case (ii) with $\gamma>1$, $\landauO((n^{-1}\log n)^{1/2})$ for $\gamma=1$, and $\landauO(n^{-\gamma/2})$ for $\gamma\in(0,1)$.
These rates require $\sigma_\alpha$ and $V^{\max}_\alpha$ to be bounded in $n$, as well as $\sum_{h\geq1}b(h)^{\alpha-1}$ in case (i) and $c$ in case (ii), with $\theta$ fixed. In case (i), Markov's inequality then gives $\ol{\mbar}{\mhat}=\Op(n^{-1/2})$. The refresh chain of \cref{prp:dep:refresh} gives an example with infinite variance.

\begin{example}\label{exa:power:decl:refresh}
Let $\mu\in\mc P^\alpha(\mc Q)$, possibly with $\mu\notin\mc P^2(\mc Q)$, let $\rho\in[0,1)$, and let $\mo Y$ be the refresh chain of \cref{prp:dep:refresh} with marginal law $\mu$. By \cref{prp:dep:refresh} with $s=\alpha/2$, $\mo Y$ is $\mo b$-sensitive at exponent $\alpha/2$ with $b_{ik}=\rho^{2\lvert i-k\rvert/\alpha}$. As $b_{ik}^{\alpha-1}=\rho^{2(\alpha-1)\lvert i-k\rvert/\alpha}$ and $1-\rho^t\geq t(1-\rho)$ for $t\in(0,1]$,
 \begin{equation*}
nr_n\leq1+2\sum_{h\geq1}\rho^{2(\alpha-1)h/\alpha}=\frac{1+\rho^{2(\alpha-1)/\alpha}}{1-\rho^{2(\alpha-1)/\alpha}}\leq\frac{\alpha}{(\alpha-1)(1-\rho)}
\eqfs
\end{equation*}
As $\mbar=m_\alpha(\mu)$, $\sigma_\alpha=\sigmax\alpha=\sigma_\alpha(\mu)$ by stationarity, and $V^{\max}_\alpha\leq2\sigmax\alpha$, \eqref{eq:power:decl:simple} gives
    \begin{equation*}
        \EOf{\ol{m_\alpha(\mu)}{\mhat}^\alpha}^{1/\alpha}
        \leq c_\alpha\,\sigma_\alpha(\mu)^{1/\alpha}\br{\frac{1+\rho}{(1-\rho)\,n}}^{1/2} \eqfs
    \end{equation*}
Up to the constant, this is the dependence on $n$ and $\rho$ of the quadratic bound \eqref{eq:2fmh:refresh:bound}, which is not available here, as $\Vmof{Y_1}=\infty$ is allowed. Moreover, $b_{ik}\leq1$ gives $r_n\leq1$, so \eqref{eq:power:decl:main} also applies.
Its leading term is at most $c_\alpha\,\sigma_{\alpha-1}^{\pwa}\,(\Vm^{2\alpha-2}[Y_1])^{1/2}((1+\rho)/((1-\rho)n))^{1/2}$ and involves only moments of order $\alpha-1$ and $2\alpha-2$, whereas the $\alpha$-th moment enters only terms of order $o(n^{-1/2})$ for fixed $\rho$. By \eqref{eq:power:decl:omega}, $\omega_n\to0$, as $r_n\to0$ and $\max_{j\in\nnset n}\Eof{\ol{Y_j}{\mbar}^\alpha\indOf{\{\ol{Y_j}{\mbar}>M\}}}$ does not depend on $n$ and tends to zero as $M\to\infty$.
\end{example}

\subsection*{Use of AI tools}
The core elements are due to the authors: the compensated replace-one stability technique (\cref{lmm:supp:loo}); the coupling bounds of \cref{sec:2fmh,sec:pos,sec:power} and their consequences under local dependence, with their proofs; and the exactness for Gaussian observations (\cref{prp:2fmh:gaussian-1d}).
The transport-sensitivity setting (\cref{ssec:dep:transport}) was developed in dialogue with Anthropic's Claude~Opus~5 and~5.5.
Claude also helped with some of the examples and results that rely on this setting. More generally, the authors used Claude to refine arguments, search the literature, check proofs and references, and improve the writing.
OpenAI's ChatGPT (GPT-6) was also used to polish the writing and to suggest references. The authors checked all AI-assisted content and take full responsibility for the article.

\printbibliography

\clearpage


\renewcommand{\theequation}{A\arabic{equation}}
\renewcommand{\thefigure}{A\arabic{figure}}
\renewcommand{\thetable}{A\arabic{table}}
\renewcommand{\thesection}{A\arabic{section}}
\renewcommand{\thesubsection}{\thesection.\arabic{subsection}}
\renewcommand{\thetheorem}{\thesection.\arabic{theorem}}

\renewcommand{\theHequation}{A\arabic{equation}}
\renewcommand{\theHfigure}{A\arabic{figure}}
\renewcommand{\theHtable}{A\arabic{table}}
\renewcommand{\theHsection}{A\arabic{section}}
\renewcommand{\theHsubsection}{\thesection.\arabic{subsection}}
\renewcommand{\theHtheorem}{\thesection.\arabic{theorem}}

\setcounter{equation}{0}
\setcounter{figure}{0}
\setcounter{table}{0}
\setcounter{section}{0}
\setcounter{subsection}{0}
\setcounter{theorem}{0}

\appendix
\addcontentsline{toc}{section}{Appendix}

\section{Auxiliary results}\label{sec:aux}

\subsection{Proof of Lemma~\ref{lmm:aux:couple:triple} }\label{proof_lmm:aux:couple:triple}
\begin{proof}
    Let $\mu$ be the marginal distribution of $A$ on $S_A$, and let $K(\cdot \mid a)$ be a regular conditional distribution of $(B, C)$ given $A = a$.
    Define a probability measure on  \[S_A \times S_B \times S_C \times S_A \times S_B \times S_C\] by
    \begin{equation*}
        \begin{split}
            &\PrOf{(A, B, C, A\pr, B\pr, C\pr) \in \cdot}
            \\
             &=
            \int_{S_A} \int_{S_B \times S_C} \int_{S_B \times S_C}
            \indOf{(\cdot)}(a, b, c, a, b\pr, c\pr)\,
            K(\dl b, \dl c \mid a)\,
            K(\dl b\pr, \dl c\pr \mid a)\,
            \mu(\dl a)
            \eqfs
        \end{split}
    \end{equation*}

    Its image under the projection onto the first three coordinates is the law of $(A,B,C)$.
    Disintegrating it with respect to this projection, which is possible as the spaces are Polish, gives a probability kernel from $S_A\times S_B\times S_C$ to $S_A\times S_B\times S_C$; drawing $(A\pr,B\pr,C\pr)$ from this kernel evaluated at $(A,B,C)$, on the product of the original probability space with $S_A\times S_B\times S_C$, realizes this law with the given $(A,B,C)$.
    By construction, $(A, B, C, A\pr, B\pr, C\pr) \stackrel{d}{=} (A\pr, B\pr, C\pr, A, B, C)$, and the joint law of $(A,A\pr) $ is concentrated on the diagonal, so $A\pr=A$ almost surely.
    Moreover, conditionally on $A=a$, the pairs $(B,C)$ and $(B\pr,C\pr)$ are iid with common law $K(\cdot \mid a)$.
    It remains to show that $C\pr$ is independent of $(A, B, C)$.
    Denote by $\nu$ the marginal distribution of $C$.
    Since $A$ and $C$ are independent, the $C$-marginal of $K(\cdot \mid a)$ equals $\nu$ for $\mu$-almost every $a$.
    Hence for measurable sets $E\subset S_A \times S_B \times S_C$ and $F\subset S_C$, we have
    \begin{align*}
        \PrOf{(A,B,C)\in E,\ C\pr\in F} &=  \int_{S_A}\int_{S_B\times S_C}\indOf{E}(a,b,c)\,\nu(F)\,K(\dl b,\dl c\mid a)\,\mu(\dl a) \\
        &= \PrOf{(A,B,C)\in E}\,\nu(F)
        \eqcm
    \end{align*}
    which shows that $C\pr$ is independent of $(A,B,C)$.
\end{proof}

We also use two elementary inequalities for powers \cite[Supplement, Lemma~S1.1]{Schoetz2026Transformed-arxiv}, \cite[Lemma~D.1]{Schoetz2026Weighted-arxiv}: for $x_1,x_2\geq0$ and $r\geq0$,
\begin{equation}\label{eq:aux:subadd}
    (x_1+x_2)^r\leq 2^{\max\{0,r-1\}}(x_1^r+x_2^r)
    \eqcm
\end{equation}
and, by the mean value theorem and the triangle inequality, for $\alpha\in(1,2]$ and $y,q,p\in\mc Q$ in a metric space, as in \cite[(119)]{Schoetz2026Weighted-arxiv},
\begin{equation}\label{eq:aux:mvt}
    \abs{\ol yq^\alpha-\ol yp^\alpha}\leq\alpha\brOf{\ol yp^{\alpha-1}+\ol pq^{\alpha-1}}\ol pq
    \eqfs
\end{equation}
The contraction property with a point mass gives a rough bound on the location of the mean; compare \cite[Supplement, Lemma~S7.4]{Schoetz2026Transformed-arxiv} and \cite[(105)]{Schoetz2026Weighted-arxiv}.

\subsection{Auxiliary lemmas}

\begin{lemma}\label{lmm:aux:rough}
    Let $(\mc Q,d)$ be a separable Hadamard space, let $\alpha\in(1,2]$, and let $\mu\in\mc P^{\alpha-1}(\mc Q)$.
    Then, for all $q\in\mc Q$,
    \begin{equation*}
        \ol{m_\alpha(\mu)}q^{\alpha-1}\leq c_\alpha\int\ol yq^{\alpha-1}\dl\mu(y)
        \eqfs
    \end{equation*}
\end{lemma}

\begin{proof}
    Set $t:=\ol{m_\alpha(\mu)}q$ and $A:=\int\ol yq^{\alpha-1}\dl\mu(y)$, so that $A=W_{\alpha-1}(\mu,\delta_q)$; we may assume $t>0$.
    If $\alpha=2$, \hyperref[sqr:CP]{(CP)} with $\nu=\delta_q$, which holds under finite first moments (\cref{ssec:2fmh:pre}), gives $t\leq A$.
    Let $\alpha<2$.
    As $m_\alpha(\delta_q)=q$, $\sigma_{\alpha-1}(\delta_q)=0$, and $(\alpha-1)\pwa=2-\alpha$, \hyperref[power:CP]{(CP$_\alpha$)} with $\nu=\delta_q$ gives
    \begin{equation*}
        t\leq\frac{2^{1-\alpha}}{\alpha-1}\,A\,\brOf{\brOf{\sigma_{\alpha-1}(\mu)+t^{\alpha-1}}^\pwa+t^{2-\alpha}}
        \eqfs
    \end{equation*}
    The triangle inequality and \eqref{eq:aux:subadd} give $\sigma_{\alpha-1}(\mu)\leq A+t^{\alpha-1}$.
    If $t^{\alpha-1}\leq A$, there is nothing to show.
    Otherwise, $\sigma_{\alpha-1}(\mu)+t^{\alpha-1}\leq3t^{\alpha-1}$, so $t\leq\frac{2^{1-\alpha}}{\alpha-1}(3^\pwa+1)A\,t^{2-\alpha}$, that is, $t^{\alpha-1}\leq\frac{2^{1-\alpha}(3^\pwa+1)}{\alpha-1}A$.
\end{proof}

\begin{lemma}\label{lmm:supp:corr}
    Let $\mo X = (X_1, \dots, X_n)$ be an $n$-tuple of $\R$-valued random variables with finite second moment and $\Vof{X_j}>0$ for all $j\in\nnset n$.
    For $i\in\nnset n$, let $\mo X^i$ be an exchangeable copy of $\mo X$ with $i$-independence, and assume \eqref{eq:2fmh:illu:corr}.
    Then, with $D_{ik}$ as in \eqref{eq:2fmh:illu:delta},
    \begin{equation*}
        \frac{1}{2} \EOf{D_{ii}^2}^{\frac{1}{2}} \EOf{D_{ik}^2}^{\frac{1}{2}}
        =
        \abs{\CovOf{X_i, X_k}}
        \eqfs
    \end{equation*}
\end{lemma}

\begin{proof}
    As $X_i^i$ is an independent copy of $X_i$, we have $\Vof{X_i} = \frac{1}{2} \Eof{D_{ii}^2}$.
    Furthermore, using $X_k \stackrel{d}{=}X_k^i$ and \eqref{eq:2fmh:illu:corr},
    \begin{equation*}
        \frac{1}{2} \EOf{D_{ik}^2}
        =
        \VOf{X_k}\br{1 - \CorOf{X_k, X_k^i}}
        =
        \VOf{X_k}\CorOf{X_i, X_k}^2
        \eqfs
    \end{equation*}
    Then \[\frac{1}{2} \Eof{D_{ii}^2}^{1/2} \Eof{D_{ik}^2}^{1/2}=\sqrt{\Vof{X_i} \Vof{X_k} \Cor(X_i, X_k)^2}=\lvert \Covof{X_i, X_k}\rvert.\]
\end{proof}

\section{Proof of Lemma \ref{lmm:supp:loo}}

\begin{proof}\label{proof_lmm:supp:loo}
We first check measurability and integrability, and then prove the coupling inequality.
The means exist and are unique by \cref{ssec:2fmh:pre,ssec:pos:pre,ssec:power:prelim}. In case (b), $\mbar\in\setC$, and the empirical measures $\muhat,\muhat^i$ belong to $\mc P(\setC)$ almost surely.

We show that the map $(y_1,\dots,y_n)\mapsto m_\alpha(n^{-1}\sum_k\delta_{y_k})$ is continuous on $\mc Q^n$ in case (a) and on $\setC^n$ in case (b). Let $\mo y^{(j)}\to\mo y$ in the relevant domain, and let $\nu_j,\nu$ be the corresponding empirical measures.
Matching corresponding observations gives $W_{\alpha-1}(\nu_j,\nu)
    \leq n^{-1}\sum_{k=1}^n  \ol{y_k^{(j)}}{y_k}^{\alpha-1} \to 0.$
For $\alpha=2$, continuity follows from \hyperref[sqr:CP]{(CP)} in case (a) and \hyperref[pos:CP]{(CP+)} in case (b).
If $\alpha<2$, by \cref{lmm:aux:rough}, applied at $q=y_1$, $\ol{m_\alpha(\nu_j)}{y_1}^{\alpha-1} \leq c_\alpha\,n^{-1} \sum_{k=1}^n
 \ol{y_k^{(j)}}{y_1}^{\alpha-1}.$
 
The right-hand side is bounded in $j$. Since the observations also remain at bounded distance from $y_1$, the triangle inequality shows that both $\sigma_{\alpha-1}(\nu_j)$ and $\ol{m_\alpha(\nu_j)}{m_\alpha(\nu)}$ are bounded. Thus the denominator in \hyperref[power:CP]{(CP$_\alpha$)} remains bounded. Together with $W_{\alpha-1}(\nu_j,\nu)\to0$, this gives
$m_\alpha(\nu_j)\to m_\alpha(\nu)$.
Hence $\mhat$ and $\mhat^i$ are random variables. In case (b), extend the mean map by a fixed point of $\setC$ outside $\setC^n$. Since $\setC$ is closed, this extension is measurable and defines the empirical means on the exceptional null events as well.
    
    In case (b), all distances are bounded by $\diam(\setC)$, so all expectations below are finite.
    In case (a), \cref{lmm:aux:rough} for $\muhat$ at $q=\mbar$ gives $\ol{\mhat}{\mbar}^{\alpha-1}\leq c_\alpha n^{-1}\sum_j\ol{Y_j}{\mbar}^{\alpha-1}$, and Jensen's inequality yields $\ol{\mhat}{\mbar}^\alpha\leq c_\alpha n^{-1}\sum_j\ol{Y_j}{\mbar}^\alpha$; thus $\Eof{\ol{\mhat}{\mbar}^\alpha}\leq c_\alpha n^{-1}\sum_j\Eof{\ol{Y_j}{\mbar}^\alpha}<\infty$, and the same holds for $\mhat^i$, as $\mo Y^i\stackrel d=\mo Y$.
    Together with \eqref{eq:aux:mvt} and H\"older's inequality with exponents $\alpha/(\alpha-1)$ and $\alpha$, this shows that all expectations below are finite.

The variable $Y_i^i$ has law $\mu_i$ and is independent of $\mo Y$, hence of $\mhat$. Fubini's theorem therefore gives
    \begin{align*}
        \EOf{\int\brOf{\ol y{\mhat}^\alpha-\ol y{\mbar}^\alpha}\dl\mubar(y)}
        &=\frac{1}{n}\sum_{i=1}^n\EOf{\int\brOf{\ol y{\mhat}^\alpha-\ol y{\mbar}^\alpha}\dl\mu_i(y)}
        \\&=\frac{1}{n}\sum_{i=1}^n\EOf{\ol{Y_i^i}{\mhat}^\alpha-\ol{Y_i^i}{\mbar}^\alpha}
        \eqfs
    \end{align*}
    Since $\Eof{\ol{Y_i^i}{\mbar}^\alpha}=\Eof{\ol{Y_i}{\mbar}^\alpha}$, the left-hand side of \eqref{eq:supp:loo} equals
    \begin{equation}\label{eq:supp:loo:sum}
        \frac{1}{n}\sum_{i=1}^n\EOf{\ol{Y_i^i}{\mhat}^\alpha-\ol{Y_i}{\mhat}^\alpha}
        \eqfs
    \end{equation}
    Fix $i\in\nnset n$. The quadruple inequality \hyperref[power:QI]{(QI$_\alpha$)} in case (a), which for $\alpha=2$ is \hyperref[sqr:QI]{(QI)}, and \hyperref[pos:QI]{(QI+)} in case (b), all four points lying in $\setC$, applied with $(y,z,q,p)=(Y_i^i,Y_i,\mhat,\mhat^i)$, give
    \begin{equation*}
        \ol{Y_i^i}{\mhat}^\alpha-\ol{Y_i}{\mhat}^\alpha
        \leq
        \ol{Y_i^i}{\mhat^i}^\alpha-\ol{Y_i}{\mhat^i}^\alpha
        +C_{\mathrm{QI}}\,\ol{Y_i}{Y_i^i}^{\alpha-1}\,\ol{\mhat}{\mhat^i}
        \eqfs
    \end{equation*}
    As $\mo Y^i\stackrel d=\mo Y$, we have $(Y_i^i,\mhat^i)\stackrel d=(Y_i,\mhat)$, and as $(\mo Y,\mo Y^i)\stackrel d=(\mo Y^i,\mo Y)$, we have $(Y_i,\mhat^i)\stackrel d=(Y_i^i,\mhat)$.
    Taking expectations therefore gives
    \begin{equation*}
        \EOf{\ol{Y_i^i}{\mhat}^\alpha-\ol{Y_i}{\mhat}^\alpha}
        \leq
        -\EOf{\ol{Y_i^i}{\mhat}^\alpha-\ol{Y_i}{\mhat}^\alpha}
        +C_{\mathrm{QI}}\,\EOf{\ol{Y_i}{Y_i^i}^{\alpha-1}\,\ol{\mhat}{\mhat^i}}
        \eqcm
    \end{equation*}
    that is, $\Eof{\ol{Y_i^i}{\mhat}^\alpha-\ol{Y_i}{\mhat}^\alpha}\leq\frac{C_{\mathrm{QI}}}2\Eof{\ol{Y_i}{Y_i^i}^{\alpha-1}\,\ol{\mhat}{\mhat^i}}$.
    Inserting this into \eqref{eq:supp:loo:sum} proves \eqref{eq:supp:loo}.
\end{proof}

\section{Proof of Example~\ref{exa:2fmh:decl:refresh-tripod}}\label{ssec:supp:2fmh:dep}

\begin{proof}
    The full tripod, rather than only its vertices, is the space over which Fréchet means are minimized.
    For a point at distance $t\in[0,1]$ from $o$ on a branch carrying endpoint mass $p$, the population objective is
    \[p(1-t)^2+(1-p)(1+t)^2 = 1+t^2+2t(1-2p)\eqfs\]
    Every branch has $p\leq1/2$, so the unique population minimizer is $o$, and $\Vfof{Y_1}=1$. Furthermore, $\VmOf{Y_1}=2^{-1}\EOf{\ol{Y_1}{Y_1\pr}^2}=5/4.$ The upper bound follows from \cref{prp:2fmh:refresh}.

    For the lower bound, set $n_j:=\sum_{t=1}^n\indOf{\{Y_t=x_j\}}$, $j\in\{1,2,3\}$.
    The same quadratic calculation for the empirical objective shows that
    \[\ol{\mhat}o^2 =\sum_{j=1}^3  \br{\br{\frac{2n_j}{n}-1} \vee 0 }^2\eqcm\]
    where at most one summand is positive.
    Let $\overline S_n:=n^{-1}\sum_{t=1}^nS_t$ with $S_t:=2\indOf{\{Y_t=x_1\}}-1$; then $\overline S_n=2n_1/n-1$.
    The sequence $(S_t)$ is a stationary symmetric two-state refresh chain: at a refresh, its sign is chosen independently and uniformly from $\{-1,1\}$.
    Its joint law is unchanged by simultaneous sign reversal, and $\Eof{S_iS_k}=\rho^{\lvert i-k\rvert}$.
    It follows that
    \[\EOf{\ol{\mhat}o^2} \geq\EOf{\br{\overline S_n\vee 0}^2} =\tfrac12\EOf{\overline S_n^2} =\tfrac12 a_n(\rho)\eqfs\]
    For the limit, the summand for $j=1$ equals $(\overline S_n\vee0)^2$ and contributes exactly $\frac{1}{2}a_n(\rho)$. It remains to show that the summands for $j\in\{2,3\}$ are $o(a_n(\rho))$, and as each is at most $1$ and $a_n(\rho)\geq1/n$, it suffices that $\Pr(n_2>n/2)=\landauO(n^{-2})$.
    Let $h:=\indOf{\{x_2\}}-\frac14$ on $\mc Q$, so that $\int h\,\dl\mu=0$, $\lvert h\rvert\leq1$, and $n_2-n/4=\sum_{t=1}^nh(Y_t)$.
    Iterating \eqref{eq:dep:refresh:kernel} gives $P^kf=\rho^kf+(1-\rho^k)\int f\,\dl\mu$ for bounded measurable $f$ and $k\in\Nn$.
    For $t_1\leq t_2\leq t_3\leq t_4$, the Markov property, stationarity, and $\int h\,\dl\mu=0$ therefore give
    \[\EOf{h(Y_{t_1})h(Y_{t_2})h(Y_{t_3})h(Y_{t_4})}=\rho^{(t_2-t_1)+(t_4-t_3)}\int h^2\,P^{t_3-t_2}(h^2)\,\dl\mu\in\left[0,\rho^{(t_2-t_1)+(t_4-t_3)}\right]\eqfs\]
    Every quadruple in $\nnset n^4$ has a nondecreasing rearrangement, which is shared by at most $4!$ quadruples, then
    \[\EOf{(n_2-n/4)^4}\leq24\sum_{t_1\leq t_2\leq t_3\leq t_4}\rho^{(t_2-t_1)+(t_4-t_3)}\leq\frac{24n^2}{(1-\rho)^2}.\]
Markov's inequality gives $\Pr(n_2>n/2)\leq4^4n^{-4}\Eof{(n_2-n/4)^4}=\landauO(n^{-2})$, the same holds for $n_3$.

Finally, the tripod cannot embed isometrically into a Hilbert space: $o$ is simultaneously the midpoint between $x_1,x_2$ and between $x_1,x_3$, which in a Hilbert space would force $x_2=x_3$.
\end{proof}

\section{An example for Remark~\ref{rem:2fmh:onMainThm} (iii)}
\begin{example}[Pair laws do not determine the risk]\label{exa:supp:pairlaws}
    Let $\mc Q$ be the tripod of \cref{exa:2fmh:decl:refresh-tripod} and $n=3$.
    For a law with mass $p_j$ at $x_j$, $j\in\nnset3$, the objective at the point at distance $t\in[0,1]$ from $o$ on the branch of $x_j$ is $p_j(1-t)^2+(1-p_j)(1+t)^2$, which is minimal at $t=(2p_j-1)_+$.
    Hence the Fréchet mean of such a law is $o$ if $p_j\leq1/2$ for all $j$, and it is the point at distance $2p_j-1$ from $o$ on the branch of $x_j$ if $p_j>1/2$.
    First, let $Y_1,Y_2,Y_3$ be iid and uniform on $\{x_1,x_2,x_3\}$.
    Second, with probability $1/3$, let $Y_1=Y_2=Y_3$ be uniform on $\{x_1,x_2,x_3\}$, and with probability $2/3$, let $(Y_1,Y_2,Y_3)$ be a uniformly chosen permutation of $(x_1,x_2,x_3)$.
    In both cases, each pair $(Y_j,Y_k)$, $j\neq k$, is uniform on $\{x_1,x_2,x_3\}^2$, and $\mbar=o$.
    In the first case, all three observations agree with probability $1/9$ and exactly two agree with probability $2/3$, so $\Eof{\ol{\mbar}{\mhat}^2}=1/9+(2/3)(1/3)^2=5/27$.
    In the second case, $\mhat=o$ unless all three observations agree, so $\Eof{\ol{\mbar}{\mhat}^2}=1/3$.
\end{example}

\section{Proof of Example~\ref{exa:pos:sphere}}\label{sec:supp:pos}

\begin{proof}
    The sphere of radius $1/\sqrt\kappa$ is a complete $\CAT(\kappa)$ space, and its closed balls of radius $r<\pi/(2\sqrt\kappa)$ are convex \cite[Proposition~II.1.4(3)]{BridsonHaefliger1999NPC-Springer}.
    Every point of $\setC$ has distance at most $r$ from $o$, and for every $o\pr\in\mc Q$ and two boundary points $y,y^*$ of $\setC$ on a geodesic through $o$, we have $2r=\ol y{y^*}\leq\ol{o\pr}y+\ol{o\pr}{y^*}$; hence $\varrho=r$.

    The rotation of $\mc Q$ by the angle $2\pi/3$ about the axis through $o$ is an isometry that preserves $\mu$ and $\setC$.
    By uniqueness, it fixes $\mbar=m_2(\mu)\in\setC$, and its only fixed point in $\setC$ is $o$.
    Hence $\mbar=o$ and $\Vfof{Y_1}=r^2$.
    For $j\neq k$, the spherical law of cosines gives $\cos(\sqrt\kappa\,\ol{x_j}{x_k})=\cos^2(\sqrt\kappa\,r)-\frac{1}{2}\sin^2(\sqrt\kappa\,r)=1-\frac32\sin^2(\sqrt\kappa\,r)$, that is, $\sin(\sqrt\kappa\,\ol{x_j}{x_k}/2)=\frac{\sqrt3}{2}\sin(\sqrt\kappa\,r)$.
    Thus $\ol{x_j}{x_k}=\ell_r$, and $\Vmof{Y_1}=\frac{1}{2}\cdot\frac23\,\ell_r^2=\ell_r^2/3$.
    As $\arcsin$ is strictly convex on $[0,1]$ with $\arcsin(0)=0$, we have $\arcsin(\lambda x)<\lambda\arcsin(x)$ for $\lambda\in(0,1)$ and $x\in(0,1]$. With $\lambda=\sqrt3/2$ and $x=\sin(\sqrt\kappa\,r)$, this gives $\ell_r<\sqrt3\,r$.
    The upper bounds in \eqref{eq:pos:sphere} now follow from \eqref{eq:pos:refresh}.

    For the lower bound, let $T_o$ be the tangent plane at $o$, let $Z_t:=\log_o(Y_t)\in T_o$ be the initial velocity of the geodesic from $o$ to $Y_t$ on $[0,1]$, and let $\widehat F_n(q):=\frac1{2n}\sum_{t=1}^n\ol q{Y_t}^2$ for $q\in\mc Q$.
    For $q\in\setC$ and $y\in\{x_1,x_2,x_3\}$, we have $\ol qy\leq2r<\pi/(2\sqrt\kappa)$.
    Hence $q\mapsto\frac{1}{2}\ol qy^2$ is smooth on a neighborhood of $\setC$, with gradient $-\log_q(y)$, and its Hessian has the eigenvalue $1$ in the direction of $\log_q(y)$ and the eigenvalue $\sqrt\kappa\,\ol qy\cot(\sqrt\kappa\,\ol qy)\in(0,1]$ in the orthogonal direction; see \cite[Proposition~A.1]{Tran2021Spheres-BJPS} for the unit sphere and the comparison results of \cite[Section~2]{Afsari2011CenterOfMass-PAMS}, and rescale the metric.
   
    At $q=y$, the Hessian is the identity.
    Thus the Hessian of $\widehat F_n$ has operator norm at most $1$ on $\setC$. Since $\muhat\in\mc P(\setC)$, the minimizer $\mhat$ of $\widehat F_n$ over $\mc Q$ lies in $\setC$ (\cref{ssec:pos:pre}), possibly on its boundary. As $\widehat F_n$ is smooth on a neighborhood of $\setC$, $\nabla \widehat F_n(\mhat)=0$. Let $\gamma\colon[0,1]\to\setC$ be the geodesic from $o$ to $\mhat$, and let $P_s$ denote parallel transport along $\gamma$ from $\gamma(s)$ to $o$.
    As $\nabla \widehat F_n(o)=-\frac{1}{n}\sum_{t=1}^nZ_t$,
    \begin{equation*}
        \normOf{\frac{1}{n}\sum_{t=1}^nZ_t}
        =\normof{\nabla \widehat F_n(o)-P_1\nabla \widehat F_n(\mhat)}
        \leq\int_0^1\normof{\operatorname{Hess}\widehat F_n(\gamma(s))\,\dot\gamma(s)}\dl s
        \leq\ol o{\mhat}
        \eqfs
    \end{equation*}
    The sequence $(Z_t)_{t\in\nnset n}$ is a refresh chain in $T_o$ whose marginal law is uniform on the three vectors $\log_o(x_j)$, which have norm $r$ and sum to $0$.
    Iterating \eqref{eq:dep:refresh:kernel}, and using that these three vectors sum to $0$, gives $\Eof{Z_k\mid Z_1,\dots,Z_i}=\rho^{k-i}Z_i$ for $i\leq k$.
    Hence $\Eof{\langle Z_i, Z_k\rangle}=r^2\rho^{\lvert i-k\rvert}$ and $\Eof{\ol{\mbar}{\mhat}^2}\geq\Eof{\normof{n^{-1}\sum_{t=1}^nZ_t}^2}=r^2a_n(\rho)$.
\end{proof}

\section{Proofs of Section \ref{sec:power}}\label{ssec:supp:power:dep}

\subsection{Proof of Lemma~\ref{lmm:supp:power:conversion}}\label{proof_lmm:supp:power:conversion}

\begin{proof}
    By \cite[Proposition~D.4]{Schoetz2026Weighted-arxiv}, which is the second inequality of \cite[Supplement, Lemma~S3.1(i)]{Schoetz2026Transformed-arxiv},
    \begin{equation*}
        \EOf{X^\alpha}\leq v^{\pwa\alpha/2}P^{\alpha/2}\brOf{1+v^{-\alpha/2}P^{(\alpha-1)/2}}^{\pwa}
        \eqcm\quad
        P:=\EOf{X^2(v+X^{\alpha-1})^{-\pwa}}
        \eqfs
    \end{equation*}
    The right-hand side is nondecreasing in $P$, so $R$ may be substituted for $P$.
    Taking the $\alpha$-th root and using $(1+x)^{\pwa/\alpha}\leq c_\alpha(1+x^{\pwa/\alpha})$ for $x\geq0$ gives
    \begin{equation*}
        \EOf{X^\alpha}^{1/\alpha}
        \leq c_\alpha\br{v^{\pwa/2}R^{1/2}+v^{\pwa/2}R^{1/2}\,v^{-\pwa/2}R^{\frac{(\alpha-1)\pwa}{2\alpha}}}
        = c_\alpha\br{v^{\pwa/2}R^{1/2}+R^{1/\alpha}}
        \eqcm
    \end{equation*}
    because $\frac{1}{2}+\frac{(\alpha-1)\pwa}{2\alpha}=\frac{1}{2}+\frac{2-\alpha}{2\alpha}=\frac1\alpha$.
\end{proof}

\subsection{Proof of Lemma~\ref{lmm:supp:power:absorption}}\label{proof:lmm:supp:power:absorption}
\begin{proof}
    Set $\vartheta:=\sigma_{\alpha-1}^{\alpha/(\alpha-1)}$, so that $\vartheta^{2-\alpha}=\sigma_{\alpha-1}^{\pwa\alpha}$ and $\vartheta^{(2-\alpha)/(2\alpha)}=\sigma_{\alpha-1}^{\pwa/2}$.
    If $x\leq\vartheta$, then $x^{1/\alpha}=x^{1/2}x^{(2-\alpha)/(2\alpha)}\leq\sigma_{\alpha-1}^{\pwa/2}x^{1/2}$.
    Otherwise, $\sigma_{\alpha-1}^{\pwa\alpha}=\vartheta^{2-\alpha}<x^{2-\alpha}$, so the assumption gives $x^{\alpha-1}\leq Cy^{\alpha-1}$, that is, $x^{1/\alpha}\leq C^{1/(\alpha(\alpha-1))}y^{1/\alpha}$.
    In both cases, $x^{1/\alpha}\leq\sigma_{\alpha-1}^{\pwa/2}x^{1/2}+C^{1/(\alpha(\alpha-1))}y^{1/\alpha}$, which gives \eqref{eq:supp:power:absorption}, as $\Phi(x)=\sigma_{\alpha-1}^{\pwa/2}x^{1/2}+x^{1/\alpha}$.

    For the second claim, we set
    \begin{equation*}
        D_n(r):=\frac{1}{n^2}\sum_{i,k=1}^n\EOf{\brOf{\ol{Y_i}{Y_i^i}\,\ol{Y_k}{Y_k^i}}^r}
        \eqcm\quad r>0,
    \end{equation*}
    so that $D_n=D_n(\alpha-1)$ and $\widetilde D_n=D_n(\alpha/2)$.
    As $\alpha-1=(2-\alpha)\frac{\alpha-1}2+(\alpha-1)\frac\alpha2$, H\"older's inequality with the exponents $\frac1{2-\alpha}$ and $\frac1{\alpha-1}$, applied first to each expectation and then to the sum, gives $D_n\leq D_n(\frac{\alpha-1}2)^{2-\alpha}\,\widetilde D_n^{\alpha-1}$.
    By the Cauchy--Schwarz inequality,
    \begin{equation*}
        \EOf{\brOf{\ol{Y_i}{Y_i^i}\,\ol{Y_k}{Y_k^i}}^{(\alpha-1)/2}}\leq\EOf{\ol{Y_i}{Y_i^i}^{\alpha-1}}^{1/2}\EOf{\ol{Y_k}{Y_k^i}^{\alpha-1}}^{1/2}.    \end{equation*}
    Then, the triangle inequality, \eqref{eq:aux:subadd}, and $Y_k^i\stackrel d=Y_k$ give $\Eof{\ol{Y_k}{Y_k^i}^{\alpha-1}}\leq2s_k$ for all $i,k$, where $s_k:=\Eof{\ol{Y_k}{\mbar}^{\alpha-1}}$.
    We observe that $n^{-1}\sum_ks_k=\sigma_{\alpha-1}$, then Jensen's inequality yields $D_n(\frac{\alpha-1}2)\leq2(n^{-1}\sum_ks_k^{1/2})^2\leq2\sigma_{\alpha-1}$.
    Thus $\sigma_{\alpha-1}^\pwa D_n\leq2^{2-\alpha}\sigma_{\alpha-1}^{\pwa+2-\alpha}\widetilde D_n^{\alpha-1}$, and $\pwa+2-\alpha=\pwa\alpha$.
\end{proof}

\subsection{Proof of Proposition \ref{prp:supp:power:dep} }

\begin{proof}\label{proof_prp:supp:power:dep}
We first prove \eqref{eq:supp:power:dep}.
Applying \cref{lmm:supp:loo}\,(a), discarding its nonnegative second term, and using \hyperref[power:VI]{(VI$_\alpha$)} for $\mubar$ at $q=\mhat$, we obtain
\begin{equation}\label{eq:power:first}
    \EOf{\psi_\alpha(\ol{\mbar}{\mhat})}
    \leq
    \frac{2^{2-\alpha}}{\alpha-1}\,\frac{1}{n}\sum_{i=1}^n\EOf{\eta_i\delta_i}
    \eqfs
\end{equation}
This is the exchangeable-copy counterpart of \cite[Supplement, Lemma~S7.1]{Schoetz2026Transformed-arxiv} and of the power case of \cite[Lemma~C.2]{Schoetz2026Weighted-arxiv}.

To obtain the bound in terms of $R_n$, we control $\delta_i$ by the deterministic stability argument of \cite[proof of Proposition~C.3]{Schoetz2026Weighted-arxiv}.
By \cref{lmm:aux:rough} for $\muhat$ at $q=\mbar$, $\ol{\mhat}{\mbar}^{\alpha-1}\leq c_\alpha\hat\sigma_{\alpha-1}$.
The triangle inequality and \eqref{eq:aux:subadd} then give $\sigma_{\alpha-1}(\muhat)\leq\hat\sigma_{\alpha-1}+\ol{\mhat}{\mbar}^{\alpha-1}\leq c_\alpha\hat\sigma_{\alpha-1}$.
The same estimates hold for the copied sample.
In particular, $\sigma_{\alpha-1}(\muhat^i)\leq c_\alpha\hat\sigma_{\alpha-1}^i$, and
\begin{equation*}
    \delta_i^{\alpha-1}
    \leq
    \ol{\mhat}{\mbar}^{\alpha-1}+\ol{\mbar}{\mhat^i}^{\alpha-1}
    \leq
    c_\alpha\br{\hat\sigma_{\alpha-1}+\hat\sigma_{\alpha-1}^i}
    \eqfs
\end{equation*}
Substituting these estimates into \hyperref[power:CP]{(CP$_\alpha$)} for $\muhat$ and $\muhat^i$ and applying \eqref{eq:aux:subadd} once more yields
\begin{equation}\label{eq:power:dep:delta-final}
    \delta_i
    \leq
    c_\alpha T_i\br{\hat\sigma_{\alpha-1}^\pwa+(\hat\sigma_{\alpha-1}^i)^\pwa}
    \eqfs
\end{equation}

The mixed moments below are finite.
Indeed, the finite $\alpha$-moment assumption and the triangle inequality give $\eta_i,\hat\sigma_{\alpha-1},\hat\sigma_{\alpha-1}^i\in L^\beta$, with $\beta$ as in \eqref{eq:power:structured-exponents}.
The coupling that matches corresponding observations gives $T_i\leq\hat\sigma_{\alpha-1}+\hat\sigma_{\alpha-1}^i$, so $T_i\in L^\beta$ as well.
H\"older's inequality with exponents $\beta,\beta,\alpha^\ast$ now applies, since $\pwa\alpha^\ast=\beta$.

Exchanging $\mo Y$ and $\mo Y^i$ leaves $\eta_iT_i$ unchanged and interchanges $\hat\sigma_{\alpha-1}$ and $\hat\sigma_{\alpha-1}^i$.
Thus, by exchangeability, $\Eof{\eta_iT_i(\hat\sigma_{\alpha-1}^i)^\pwa}=\Eof{\eta_iT_i\hat\sigma_{\alpha-1}^\pwa}$.
Combining this identity with \eqref{eq:power:first} and \eqref{eq:power:dep:delta-final} gives
\begin{equation*}
    \EOf{\psi_\alpha(\ol{\mbar}{\mhat})}
    \leq
    \frac{c_\alpha}{n}\sum_{i=1}^n\EOf{\hat\sigma_{\alpha-1}^\pwa\eta_iT_i}
    =
    c_\alpha R_n
    \eqfs
\end{equation*}

It remains to prove the two bounds on $R_n$.
The coupling $n^{-1}\sum_{k=1}^n\delta_{(Y_k,Y_k^i)}$ of $\muhat$ and $\muhat^i$ gives $T_i\leq n^{-1}\sum_{k=1}^n\eta_{ik}$, so that $R_n\leq n^{-2}\sum_{i,k=1}^n\Eof{\hat\sigma_{\alpha-1}^\pwa\eta_i\eta_{ik}}$.
For the first bound, \eqref{eq:aux:subadd} gives $\hat\sigma_{\alpha-1}^\pwa\leq c_\alpha(\sigma_{\alpha-1}^\pwa+\lvert \hat\sigma_{\alpha-1}-\sigma_{\alpha-1}\rvert^\pwa)$.
The constant term contributes $c_\alpha\sigma_{\alpha-1}^\pwa D_n$.
For the fluctuation term, H\"older's inequality with the conjugate exponents $\alpha^\ast,\beta^\ast$ of \eqref{eq:power:structured-exponents} gives
\begin{equation*}
    \EOf{\abs{\hat\sigma_{\alpha-1}-\sigma_{\alpha-1}}^\pwa\eta_i\eta_{ik}}
    \leq
    \omega_n^\pwa\,\EOf{\ol{Y_i}{Y_i^i}^{\alpha/2}\,\ol{Y_k}{Y_k^i}^{\alpha/2}}^{\frac{2\alpha-2}{\alpha}}
    \eqcm
\end{equation*}
where we used $\pwa\alpha^\ast=\beta$ and $(\alpha-1)\beta^\ast=\alpha/2$.
Summing over $i$ and $k$ proves the first bound in \eqref{eq:supp:power:dep:final}.
For the second bound, Jensen's inequality gives $\Eof{\hat\sigma_{\alpha-1}^\beta}\leq n^{-1}\sum_{j=1}^n\Eof{\ol{Y_j}{\mbar}^\alpha}=\sigma_\alpha$.
The claim follows by applying the same H\"older estimate directly to $\hat\sigma_{\alpha-1}^\pwa\eta_i\eta_{ik}$ and summing over $i$ and $k$, as $1/\alpha^\ast=(2-\alpha)/\alpha$.
\end{proof}

\subsection{Proof of Lemma~\ref{lmm:supp:power:omega}}

\begin{proof}\label{proof_lmm:supp:power:omega}
We first prove \eqref{eq:power:dep:omega} by truncation. Set $\xi_k:=\ol{Y_k}{\mbar}^{\alpha-1}$ and $\xi_k^i:=\ol{Y_k^i}{\mbar}^{\alpha-1}$, and recall that $\beta=\alpha/(\alpha-1)$. Since $\abs{a^{\alpha-1}-b^{\alpha-1}}
\leq\abs{a-b}^{\alpha-1}$ for $a,b\geq0$, the triangle inequality gives
$\abs{\xi_k-\xi_k^i}\leq\eta_{ik}$. 

Fix $M\geq0$ and let $g(x):=\min\{x,M^{\alpha-1}\}$
for $x\geq0$. For the truncated sums, write  $G:=\sum_{k=1}^ng(\xi_k)$ and $G^i:=\sum_{k=1}^ng(\xi_k^i)$. By exchangeability, $\Eof{g(\xi_i^i)G^i}=\Eof{g(\xi_i)G}$ and $\Eof{g(\xi_i)G^i}=\Eof{g(\xi_i^i)G}$, where the second equality uses that $\xi_i^i$ is independent of $G$ and has the same law as $\xi_i$.
Summing over $i$ and using that $g$ is $1$-Lipschitz gives
    \[
        \VOf{G}
        =\frac{1}{2}\sum_{i=1}^n\EOf{\brOf{g(\xi_i)-g(\xi_i^i)}(G-G^i)}
        \leq\frac{1}{2}\sum_{i,k=1}^n\EOf{\eta_i\eta_{ik}}
        =\frac{n^2D_n}2
        \eqfs
    \]
Since $G/n$ takes values in $[0,M^{\alpha-1}]$ and $(\alpha-1)(\beta-2)=2-\alpha$, we obtain $\Eof{\lvert G/n-\Eof{G/n}\rvert^\beta}\leq M^{(\alpha-1)(\beta-2)}\Vof{G/n}\leq M^{2-\alpha}D_n/2$. We next control the remainder 
$H:=\hat\sigma_{\alpha-1}-G/n=n^{-1}\sum_{k=1}^n(\xi_k-M^{\alpha-1})_+.$ The pointwise bound $(\xi_k-M^{\alpha-1})_+^\beta\leq\ol{Y_k}{\mbar}^\alpha\indOf{\{\ol{Y_k}{\mbar}>M\}}$ and Jensen's inequality imply 
$\EOf{H^\beta} \leq \max_{j\in\nnset n} \EOf{\ol{Y_j}{\mbar}^\alpha\indOf{\{\ol{Y_j}{\mbar}>M\}}}.$
Moreover, Minkowski's inequality gives $\Eof{\abs{H-\Eof{H}}^\beta}^{1/\beta}
\leq2\Eof{H^\beta}^{1/\beta}$.
We note that $\overline X_n=(G/n-\Eof{G/n})+(H-\Eof{H})$, then Minkowski's inequality implies
\[\omega_n\leq(M^{2-\alpha}D_n/2)^{1/\beta}+2(\max_{j\in\nnset n}\Eof{\ol{Y_j}{\mbar}^\alpha\indOf{\{\ol{Y_j}{\mbar}>M\}}})^{1/\beta}\eqfs\]
Applying \eqref{eq:aux:subadd}, and taking the infimum over $M\geq0$, we obtain \eqref{eq:power:dep:omega}.

For the second assertion, no truncation is needed.
Let $Y_k\pr$ be an independent copy of $Y_k$, and write $\xi_k\pr:=\ol{Y_k\pr}{\mbar}^{\alpha-1}$. Then  $\xi_k-\Eof{\xi_k}=\Eof{\xi_k-\xi_k\pr\mid Y_k}$. Conditional Jensen, the bound $\abs{\xi_k-\xi_k\pr} \leq\ol{Y_k}{Y_k\pr}^{\alpha-1}$, and $(\alpha-1)\beta=\alpha$ give  
    $\Eof{\lvert\xi_k-\Eof{\xi_k}\rvert^\beta}\leq\Eof{\ol{Y_k}{Y_k\pr}^\alpha}=2\Vm^\alpha[Y_k]\leq2V^{\max}_\alpha$.Finally, Minkowski's inequality yields $\omega_n\leq(2V^{\max}_\alpha)^{1/\beta}$.
\end{proof}

\subsection{Dependency neighborhoods}\label{ssec:supp:power:strucdep}

Recall the maximal size $N_{\max}=\max_{i\in\nnset n}\#N_i$ of the dependency neighborhoods from \eqref{eq:dep:L}.

\subsubsection{Existence of a partition}

\begin{lemma}\label{lmm:power:locdep:partition}
    Let $N_i$ be a dependency neighborhood of $Y_i$ for every $i\in\nnset n$.
    Then there exists a partition $\mc B_1,\dots,\mc B_\chi$ of $\nnset n$ with $\chi \leq N_{\max}$ such that, for every $t\in\nnset\chi$, the family $(Y_j)_{j\in\mc B_t}$ is mutually independent.
\end{lemma}

\begin{proof}
    We construct the partition by a first-fit rule. Process the indices $j=1,\dots,n$ in increasing order. When $j$ is processed, add $j$ to the first existing block $\mc B_t$ with $\mc B_t\cap N_j=\emptyset$. If there is no such block, open a new block $\{j\}$.

    Each block is mutually independent.
    Indeed, let $\mc B_t=\{j_1<\dots<j_M\}$.
    When $j_v$ was added, $\{j_1,\dots,j_{v-1}\}\subseteq N_{j_v}\compl$, so $Y_{j_v}$ is independent of $(Y_{j_1},\dots,Y_{j_{v-1}})$.
    Induction on $v$ yields mutual independence of $(Y_{j_1},\dots,Y_{j_M})$.

    At most $N_{\max}$ blocks are opened.
    Indeed, when $j$ is processed, $j$ lies in no block yet, and a block is excluded only if it meets $N_j\setminus\{j\}$.
    As the blocks are disjoint and $\#(N_j\setminus\{j\})\leq N_{\max}-1$, at most $N_{\max}-1$ blocks are excluded.
    Hence a new block is opened only if at most $N_{\max}-1$ blocks exist, and $\chi\leq N_{\max}$.
\end{proof}

The bound $\chi\leq N_{\max}$ cannot be improved in general.
For $k,a\in\N$ and $n=ka$, let $Y_j:=\xi_g$ for $j\in G_g:=\{(g-1)a+1,\dots,ga\}$, where $\xi_1,\dots,\xi_k$ are iid and nondegenerate.
Then the groups $G_g$ are dependency neighborhoods, so that $N_{\max}=a$, and every mutually independent block contains at most one index of each group, so $\chi\geq a=N_{\max}$.
For $\ell$-dependence with $N_i=\setByEleInText{j\in\nnset n}{\lvert i-j\rvert\leq\ell}$, the first-fit rule returns the residue classes modulo $\ell+1$: inductively, when $j>\ell+1$ is processed, the blocks are the residue classes of $\{1,\dots,j-1\}$, and all of them except the class of $j$ meet $\{j-\ell,\dots,j-1\}$.
Thus $\chi=\min\{n,\ell+1\}$, which is at most the general bound $N_{\max}=\min\{n,2\ell+1\}$.

\subsubsection{Proof of Lemma~\ref{lmm:power:locdep:Xbar}}\label{proof_lmm:power:locdep:Xbar} 

\begin{proof}
    Set $S^{\mc B_t} := \sum_{j\in\mc B_t}X_j$ and define $\normof{Z}_\beta:=\Eof{\lvert Z\rvert^\beta}^{1/\beta}$ for a real random variable $Z$.
    For each $j$, let $Y_j\pr$ be an independent copy of $Y_j$.
    As $\lvert a^{\alpha-1}-b^{\alpha-1}\rvert\leq\lvert a-b\rvert^{\alpha-1}$ for $a,b\geq0$, Jensen's inequality, the triangle inequality, and $(\alpha-1)\beta = \alpha$ give
    \begin{equation*}
        \EOf{\abs{X_j}^\beta} \leq \EOf{\abs{\ol{Y_j}{\mbar}^{\alpha-1}-\ol{Y_j\pr}{\mbar}^{\alpha-1}}^\beta} \leq \EOf{\ol{Y_j}{Y_j\pr}^\alpha} = 2\VmsOf{\alpha}{Y_j} \leq 2V^{\max}_\alpha
    \eqfs
    \end{equation*}
    The variables $(X_j)_{j\in\mc B_t}$ are independent and centered, and $\beta>2$, so the inequality
    \begin{equation*}
        \EOf{\abs{\textstyle\sum_{j}X_j}^\beta}^{2/\beta}\leq(\beta-1)\sum_j\EOf{\abs{X_j}^\beta}^{2/\beta}
        \eqcm
    \end{equation*}
    a special case of \cite[Theorem~2.1]{Rio2009MomentInequalities-JTP}, gives
    \begin{equation*}
        \normof{S^{\mc B_t}}_\beta \leq (\beta-1)^{1/2}\left(\sum_{j\in\mc B_t}\normof{X_j}_\beta^{2}\right)^{1/2} \leq c_\alpha(\#\mc B_t)^{1/2}\brOf{V^{\max}_\alpha}^{1/\beta}
    \eqfs
    \end{equation*}
    By Minkowski's inequality, the Cauchy--Schwarz inequality, and $\sum_{t = 1}^\chi\#\mc B_t = n$,
    \begin{equation*}
        \normOf{\sum_{j = 1}^nX_j}_\beta \leq \sum_{t = 1}^\chi\normof{S^{\mc B_t}}_\beta \leq c_\alpha\brOf{V^{\max}_\alpha}^{1/\beta}\sum_{t = 1}^\chi(\#\mc B_t)^{1/2} \leq c_\alpha\brOf{V^{\max}_\alpha}^{1/\beta}(\chi n)^{1/2}
    \eqfs
    \end{equation*}
    Raising this to the power $\beta$ and dividing by $n^\beta$ proves \eqref{eq:power:locdep:Xbar}.
\end{proof}

\section{An \texorpdfstring{$\ms{AR}(1)$}{AR(1)} comparison of the dependence bounds}\label{sec:aux:ar1}

Let $(\xi_t)_{t\in\Z}$ be iid, either uniform on $\{-1,+1\}$ or standard Gaussian, let $\rho\in(0,1)$, and consider the stationary process
\begin{equation*}
    Y_t:=\sum_{j=0}^{\infty}\rho^j\xi_{t-j}
    \eqcm\quad
    Y_{t+1}=\rho Y_t+\xi_{t+1}
    \eqcm\quad
    \sigma^2:=\VOf{Y_t}=\frac1{1-\rho^2}
    \eqfs
\end{equation*}
Here $\mc Q=\R$, $\mbar=0$, and $\mhat=\overline Y_n:=n^{-1}\sum_{t=1}^nY_t$.
In both models, independence of the innovations gives $\Covof{Y_i,Y_k}=\sigma^2\rho^{\lvert i-k\rvert}$ and therefore
\begin{equation}\label{eq:2fmh:decl:limits:variance}
    \EOf{\abs{\mhat-\mbar}^2}
    =\frac{\sigma^2}{n^2}\br{n+2\sum_{h=1}^{n-1}(n-h)\rho^h}
    \leq\frac1{(1-\rho)^2n}
    \eqfs
\end{equation}
In fact, $n\Eof{\lvert \mhat-\mbar\rvert^2}\to(1-\rho)^{-2}$.
Although the two innovation laws give the same exact error, \cref{thm:2fmh:decl,thm:2fmh:dependent} behave differently for them.

\subsection{Rademacher noise}\label{ssec:aux:ar1:rademacher}

Assume that the innovations are uniform on $\{-1,+1\}$ and that $\rho<1/2$.

\subsubsection{Transport sensitivity}\label{sssec:aux:ar1:decl}
The difficulty with transport sensitivity comes from the past: each observation determines its predecessors through an expanding map.
Set $M:=(1-\rho)^{-1}$ and $a:=(1-2\rho)/(1-\rho)>0$.
Since $\lvert Y_t\rvert\leq M$ and $\lvert \rho Y_{t-1}\rvert\leq\rho M<1$, we have $\operatorname{sgn}(Y_t)=\xi_t$ and $Y_t\in[-M,-a]\cup[a,M]$, and the past is recovered recursively by
\begin{equation*}
    T(y):=\frac{y-\operatorname{sgn}(y)}{\rho}
    \eqcm\quad
    Y_{i-h}=T^h(Y_i)
    \eqcm\quad h\geq1
    \eqfs
\end{equation*}

Let $\mu$ denote the stationary law and $S$ its support.
Since $T$ is continuous on $[-M,-a]\cup[a,M]$ and $T(Y_t)=Y_{t-1}\stackrel d=Y_t$, we have $T(S)\subseteq S$.
First suppose that $\rho\leq1/3$. On each sign branch, $T$ has slope $\rho^{-1}$. For $y,y\pr\in S$ on opposite branches, $\lvert y-y\pr\rvert\geq2a$ and $\lvert T(y)-T(y\pr)\rvert\leq2M$, so
\begin{equation*}
    \frac{\abs{T(y)-T(y\pr)}}{\abs{y-y\pr}}
    \leq\frac{M}{a}
    =\frac1{1-2\rho}
    \leq\frac1\rho
    \eqfs
\end{equation*}
Thus $T$ is $\rho^{-1}$-Lipschitz on $S$.
As $T(S)\subseteq S$, iteration gives
\begin{equation}\label{eq:2fmh:decl:limits:inverse}
    \abs{T^h(y)-T^h(y\pr)}\leq\rho^{-h}\abs{y-y\pr}
    \eqcm\quad y,y\pr\in S
    \eqcm\quad h\geq1
    \eqfs
\end{equation}

Given $Y_i=y$ and $Y_i\pr=y\pr$, couple the futures using the same innovations and the pasts by their deterministic reconstructions.
This defines a measurable, symmetric coupling of the entire conditional samples, with forward differences $\rho^h\lvert y-y\pr\rvert$ and backward differences bounded by \eqref{eq:2fmh:decl:limits:inverse}. Thus \cref{def:dep:transport} holds at exponent $1$ with $b_{ik}:=\rho^{k-i}$, $i,k\in\nnset n$.
\Cref{thm:2fmh:decl} gives
\begin{equation*}
    \EOf{\abs{\mhat-\mbar}^2}
    \leq A_n
    :=\frac{\sigma^2}{n^2}\br{n+\sum_{h=1}^{n-1}(n-h)\br{\rho^h+\rho^{-h}}}
    \eqfs
\end{equation*}
This bound diverges as $n\to\infty$.
For general $\rho\in(0,1/2)$, the problem is unavoidable within this theorem: conditional pasts are point masses, and, as the law of $Y_i$ is atomless, pairs $y\neq y\pr$ with the same first $h$ signs have positive probability; on them $\lvert T^h(y)-T^h(y\pr)\rvert=\rho^{-h}\lvert y-y\pr\rvert$, which forces $b_{i,i-h}\geq\rho^{-h}$ for every admissible coupling.
Thus, for every admissible matrix $\mo b$, the bound of \cref{thm:2fmh:decl} is at least $\sigma^2b_{n1}/n^2\geq\sigma^2\rho^{-(n-1)}/n^2$ and diverges, and every profile admissible in \cref{cor:2fmh:decl:profile} satisfies $b(h)\geq\rho^{-h}$ for $h\leq n-1$.

The forward dynamics nevertheless contract.
Starting from $y,y\pr$ and using common innovations gives a difference of $\rho^h\lvert y-y\pr\rvert$ after $h$ steps.
For every $p>0$, this yields geometric moment contraction in the sense of \cite{WuShao2004IteratedRandom-JAP}, with $r=\rho^p$.
The forward $\tau$-coefficients also decay geometrically.
Indeed, $Y_{i+h}=\rho^hY_i+\sum_{j=0}^{h-1}\rho^j\xi_{i+h-j}$, where the sum is independent of $(Y_t)_{t\leq i}$. Replacing $Y_i$ by an independent copy $Y_i\pr$ and keeping this sum unchanged couples the conditional law given the past with the marginal law. Hence, with the coefficients of \cite{DedeckerPrieur2005Dependence-PTRF}, $\tau(\sigma(Y_t:t\leq i),Y_{i+h})\leq\rho^h\Eof{\lvert Y_i-Y_i\pr\rvert}$.

\subsubsection{The coupling bound}\label{sssec:aux:ar1:general}

We now minimize the two bounds of \cref{thm:2fmh:dependent} over all exchangeable copies with $i$-independence.
Let $(\widetilde\xi_t)_{t\in\Z}$ be an independent copy of the innovations, and let $\widetilde Y_t:= \sum_{j\geq0}\rho^j\widetilde\xi_{t-j}$.
For each $i$, replace $\xi_t$ with $\widetilde\xi_t$ for $t\leq i$ and retain $\xi_t$ for $t>i$.
Swapping the iid pasts preserves the law, and the new $i$-th observation is independent of $\mo Y$, giving an exchangeable copy $\mo Y^i$ with $i$-independence.
Setting $Z_t:=Y_t-\widetilde Y_t$, we have
\begin{equation*}
    Y_k-Y_k^i=
    \begin{cases}
        Z_k, & k\leq i,\\
        \rho^{k-i}Z_i, & k>i.
    \end{cases}
\end{equation*}

We first compute the mixed moments in \eqref{eq:2fmh:result_final}.
Since $Z_{t+1}=\rho Z_t+\xi_{t+1}-\widetilde\xi_{t+1}$, where $\xi_{t+1}-\widetilde\xi_{t+1}$ is independent of $(Z_s)_{s\leq t}$, and $\lvert \rho Z_t\rvert\leq2\rho M<2$, averaging over the next pair of innovations gives $\Eof{\lvert Z_{t+1}\rvert\mid Z_s,\,s\leq t}=1+\frac\rho2\lvert Z_t\rvert$.
Stationarity and iteration then yield
\begin{equation*}
    d_\rho:=\EOf{\abs{Z_0}}=\frac1{1-\rho/2}
    \eqcm\quad
    f_h:=\EOf{\abs{Z_0Z_h}}
    =d_\rho^2+\br{2\sigma^2-d_\rho^2}\br{\frac\rho2}^h
    \eqfs
\end{equation*}
Therefore, \eqref{eq:2fmh:result_final} becomes
\begin{equation*}
    \EOf{\abs{\mhat-\mbar}^2}
    \leq C_n
    :=\frac{\sigma^2}{n}
    +\frac1{2n^2}\sum_{h=1}^{n-1}(n-h)\br{2\sigma^2\rho^h+f_h}
    \to \frac{d_\rho^2}{4}>0
    \eqfs
\end{equation*}
Here the limit is taken as $n\to\infty$ for fixed $\rho$.
We now show that no admissible copies give a smaller coordinatewise bound.
Fix $i$.
For every such copy, $Y_i^i\perp\mo Y$ and, by exchangeability, $Y_i\perp\mo Y^i$.

For $k<i$, $\mo Y^i\stackrel d=\mo Y$ gives $Y_k^i=T^{i-k}(Y_i^i)$ almost surely, so the law of $(Y_i-Y_i^i,Y_k-Y_k^i)$ is determined by the law $\mu\otimes\mu$ of $(Y_i,Y_i^i)$ and does not depend on the copy.

For $k>i$, $\Eof{Y_k\mid Y_i,Y_i^i}=\rho^{k-i}Y_i$ as $Y_i^i\perp\mo Y$, and $\Eof{Y_k^i\mid Y_i,Y_i^i}=\rho^{k-i}Y_i^i$ as $Y_i\perp\mo Y^i$.
Then the conditional Jensen inequality bounds the mixed moment below by $\rho^{k-i}\Eof{(Y_i-Y_i^i)^2}=2\sigma^2\rho^{k-i}$.
Our copies attain this lower bound by sharing the future innovations.
Thus $C_n$ is the smallest possible coordinatewise bound.

The intermediate bound \eqref{eq:2fmh:result_intermediate} retains cancellation.
Set $\zeta_j:=\xi_j-\widetilde\xi_j$, $q_n:=\sum_{k=1}^n\rho^k$, and $w_{j,n}:=\sum_{k=j}^n\rho^{k-j}$.
Then
\begin{equation*}
    H_{i,n}:=n(\mhat-\mhat^i)
    =q_nZ_0+\sum_{j=1}^iw_{j,n}\zeta_j
    \eqcm\quad
    J_n:=\frac1{2n^2}\sum_{i=1}^n\EOf{\abs{Z_iH_{i,n}}}
    \eqfs
\end{equation*}
The resulting bound satisfies
\begin{equation*}
\begin{aligned}
    \EOf{\abs{\mhat-\mbar}^2}\leq J_n\eqcm \quad
    \frac{a}{2n^2}\sum_{i=1}^n\sqrt{i-1}
    \leq J_n
    \leq\frac{\sigma}{(1-\rho)n^2}\sum_{i=1}^n\sqrt{i+\rho^2\sigma^2}
    \eqfs
\end{aligned}
\end{equation*}
For the upper estimate, independence of $Z_0,\zeta_1,\dots,\zeta_i$ gives
\begin{equation*}
    \EOf{H_{i,n}^2}=2\sigma^2q_n^2+2\sum_{j=1}^iw_{j,n}^2\leq\frac{2(\rho^2\sigma^2+i)}{(1-\rho)^2}
    \eqcm
\end{equation*}
and the Cauchy--Schwarz inequality and $\Eof{Z_i^2}=2\sigma^2$ apply.

For the lower estimate, define $H_{i,n}=R_{i,n}+w_{i,n}\zeta_i$, where $R_{i,n}$ is independent of $\zeta_i$.
On $\{\zeta_i\neq0\}$, which has probability $1/2$, we have $\lvert Z_i\rvert\geq2a$.
Averaging over $\zeta_i=\pm2$ gives $\Eof{\lvert Z_iH_{i,n}\rvert}\geq a\Eof{\lvert R_{i,n}\rvert}$.
Writing each $\zeta_j=\xi_j-\widetilde\xi_j$ as a sum of two independent Rademacher variables, $R_{i,n}$ is an $L^2$-convergent series of independent Rademacher variables with deterministic coefficients. As $\Eof{R_{i,n}^2}\geq2(i-1)$, the Khinchin inequality with Szarek's optimal constant \cite[Theorem~1]{Szarek1976Khinchin-StudiaMath} yields $\Eof{\lvert R_{i,n}\rvert}\geq\Eof{R_{i,n}^2}^{1/2}/\sqrt2\geq\sqrt{i-1}$.
Summing over $i$ proves the lower bound on $J_n$.

It remains to show that our copies minimize the intermediate bound.
Fix $i$, and define $g(y):=\Eof{\sum_{k=1}^nY_k\mid Y_i=y}$.
For any admissible copy, the independence relations $Y_i^i\perp\mo Y$ and $Y_i\perp\mo Y^i$ imply $\Eof{\sum_{k=1}^n(Y_k-Y_k^i)\mid Y_i,Y_i^i}=g(Y_i)-g(Y_i^i)$.
The conditional Jensen inequality therefore bounds the $i$-th summand of \eqref{eq:2fmh:result_intermediate} below by
\begin{equation*}
    \frac1{2n^2}\EOf{\abs{Y_i-Y_i^i}\,\abs{g(Y_i)-g(Y_i^i)}}
    \eqfs
\end{equation*}
This lower bound depends only on $\mu\otimes\mu$.
Our copies attain it, since the shared future innovations cancel and $\sum_{k=1}^n(Y_k-Y_k^i)=g(Y_i)-g(Y_i^i)$.
Thus $J_n$ is the smallest possible intermediate bound.
For fixed $\rho$, it has order $n^{-1/2}$, whereas the exact risk in \eqref{eq:2fmh:decl:limits:variance} has order $n^{-1}$.

\subsection{Gaussian noise}\label{ssec:aux:ar1:gaussian}

Assume now that the innovations are iid $N(0,1)$; here every $\rho\in(0,1)$ is allowed.
For $h\geq1$, the backward conditional law is
\begin{equation*}
    \mc L(Y_{i-h}\mid Y_i=y)
    =N\brOf{\rho^h y,\sigma^2(1-\rho^{2h})}
    \eqfs
\end{equation*}
Gaussian regression gives $b_{ik}=\rho^{\lvert i-k\rvert}$ in both directions.
Explicitly, choose $Y_i^i\stackrel d=Y_i$ independently of $\mo Y$ and set $Y_k^i:=Y_k-\rho^{\lvert i-k\rvert}(Y_i-Y_i^i)$ for $k\in\nnset n$.
By \cref{prp:2fmh:gaussian-1d}, this is an exchangeable copy with $i$-independence.
All coordinate differences are nonnegative multiples of $Y_i-Y_i^i$, so both \eqref{eq:2fmh:result_intermediate} and \eqref{eq:2fmh:result_final}, as well as the bound of \cref{thm:2fmh:decl}, equal the exact variance for every $n$.
In particular, the regression coupling gives
\begin{equation*}
    \frac1{2n^2}\sum_{i,k=1}^n\EOf{\abs{Y_i-Y_i^i}\,\abs{Y_k-Y_k^i}}
    =\frac{\sigma^2}{n^2}\sum_{i,k=1}^n\rho^{\abs{i-k}}
    =\EOf{\abs{\mhat-\mbar}^2}
    \eqfs
\end{equation*}
\cref{cor:2fmh:decl:profile}\,(ii) gives the bound $\frac{1+\rho}{1-\rho}\cdot\frac{\sigma^2}{n}=\frac1{(1-\rho)^2n}$, which has the correct asymptotic constant.

\subsection{Gaussian noise with negative parameter}\label{ssec:aux:ar1:negative}

We now take $\rho\in(-1,0)$ and iid standard Gaussian innovations.
The moving-average representation at the beginning of \cref{sec:aux:ar1} remains valid, with $\sigma^2=(1-\rho^2)^{-1}$ and $\Covof{Y_i,Y_k}=\sigma^2\rho^{\lvert i-k\rvert}$.
The covariances alternate in sign, so the coordinatewise bound \eqref{eq:2fmh:result_final} is strictly larger than $\Vof{\overline Y_n}$ for $n\geq2$ and every exchangeable copy with $i$-independence (\cref{sssec:2fmh:illu:gauss}).
The intermediate bound \eqref{eq:2fmh:result_intermediate}, however, can retain the cancellation.
The exact variance is
\begin{equation*}
    \VOf{\overline Y_n}
    =\frac{\sigma^2}{n^2}\br{n+2\sum_{h=1}^{n-1}(n-h)\rho^h}
    \sim\frac{\sigma^2}{n}\frac{1+\rho}{1-\rho}
    \eqfs
\end{equation*}
For the Gaussian copies of \cref{prp:2fmh:gaussian-1d}, the intermediate and coordinatewise bounds are, by \eqref{eq:2fmh:illu:gaussian-1d-sharp},
\begin{equation*}
    \tilde J_n
    :=\frac{\sigma^2}{n^2}\sum_{i=1}^n\abs{\sum_{k=1}^n\rho^{\abs{i-k}}}
    \eqcm\quad
    \tilde C_n
    :=\frac{\sigma^2}{n^2}\br{n+2\sum_{h=1}^{n-1}(n-h)\abs{\rho}^h}
    \sim\frac{\sigma^2}{n}\frac{1+\abs{\rho}}{1-\abs{\rho}}
    \eqcm
\end{equation*}
respectively.
Thus $\Vof{\overline Y_n}\leq\tilde J_n\leq\tilde C_n$.
The transport-sensitivity bound of \cref{thm:2fmh:decl} also equals $\tilde C_n$, using the coefficients of \cref{prp:dep:gaussian}.
All asymptotic statements here are as $n\to\infty$ for fixed $\rho$.

To compare $\tilde J_n$ with the exact variance, consider the row sums
\begin{equation*}
    r_i:=\sum_{k=1}^n\rho^{\abs{i-k}}
    \eqcm\quad\text{so that}\quad
    \VOf{\overline Y_n}=\frac{\sigma^2}{n^2}\sum_{i=1}^n r_i
    \quad\text{and}\quad
    \tilde J_n=\frac{\sigma^2}{n^2}\sum_{i=1}^n\abs{r_i}
    \eqfs
\end{equation*}
In particular, the two quantities agree whenever all row sums are nonnegative.
With $P_j:=\sum_{h=1}^j\rho^h$, we have $r_i=1+P_{i-1}+P_{n-i}$, and $P_0=0$ and $P_j-\rho=\rho^2(1-\rho^{j-1})/(1-\rho)\geq0$ for $j\geq1$; hence $r_i\geq1+2\rho$.
For $\rho\in[-1/2,0)$, this gives $\tilde J_n=\Vof{\overline Y_n}$ for every $n$.
For $\rho\in(-1,-1/2)$, some row sums may be negative.
However, $\lvert r_i-(1+\rho)/(1-\rho)\rvert\leq(\lvert \rho\rvert^i+\lvert \rho\rvert^{n-i+1})/(1-\rho)$ and $(1+\rho)/(1-\rho)>0$, so negative row sums can occur only within a distance of the two ends bounded in terms of $\rho$.
Furthermore, $r_i\geq1+2\rho>-1$. Then,  
\begin{equation*}
    \tilde J_n-\VOf{\overline Y_n}
    =\frac{2\sigma^2}{n^2}\sum_{i:r_i<0}\abs{r_i}
    =\landauO(n^{-2}).
\end{equation*}
Thus the intermediate bound is exact for $\rho\in[-1/2,0)$ and asymptotically exact for every fixed $\rho\in(-1,0)$.

\end{document}